\documentclass[reqno]{amsart} 

\usepackage[utf8]{inputenc} 
\usepackage[english]{babel}
\usepackage{mathabx}
\usepackage{amstext}
\usepackage{euscript}
\usepackage{bbm}
\usepackage{bm}
\usepackage{mathrsfs}
\usepackage{enumerate}
\usepackage{graphicx}
\usepackage{caption}
\usepackage{amscd}
\usepackage{verbatim}
\usepackage{amssymb}
\usepackage{amsthm}
\usepackage{amsmath}
\usepackage{amstext}
\usepackage{amsfonts}
\usepackage{latexsym}
\usepackage{mathtools} 
\usepackage{enumitem} 
\usepackage{accents} 
\usepackage[foot]{amsaddr} 
\usepackage{cite} 
\usepackage{textcomp} 
\usepackage{dsfont}

\usepackage[usenames,dvipsnames]{xcolor} 
\definecolor{dkblue}{RGB}{1,31,91} 
\usepackage[colorlinks=true, pdfstartview=FitV, linkcolor=dkblue, citecolor=dkblue, urlcolor=dkblue]{hyperref}

\definecolor{mygreen}{RGB}{44,85,17}
\definecolor{myblue}{RGB}{34,31,217}
\definecolor{mybrown}{RGB}{194,164,113}
\definecolor{myred}{RGB}{255,66,56}

\theoremstyle{definition}
\newtheorem{theorem}{Theorem}
\newtheorem{corollary}[theorem]{Corollary}
\newtheorem{lemma}[theorem]{Lemma}
\newtheorem{proposition}[theorem]{Proposition}
\newtheorem{remark}[theorem]{Remark}
\newtheorem{definition}[theorem]{Definition}

\numberwithin{equation}{section}
\numberwithin{theorem}{section}

\allowdisplaybreaks

\newcommand{\eqdef }{\overset{\mbox{\tiny{def}}}{=}}

\newcommand{\R}{\mathbb R}
\newcommand{\Z}{\mathbb Z}
\newcommand{\T}{\mathbb T}
\newcommand{\N}{\mathbb N}
\newcommand{\C}{\mathbb C}
\newcommand{\1}{\mathds{1}}

\newcommand{\bT}{\mathbf{T}}
\newcommand{\BT}{\mathbf{T}}
\newcommand{\TE}{\mathcal{T}}

\newcommand{\SL}{\mathcal{S}}
\newcommand{\subw}{\mu}
\newcommand{\BS}{\mathcal{B}^\subw}
\newcommand{\CM}{\BS_T}
\newcommand{\DM}{\mathcal{D}_T^\subw}
\newcommand{\DMD}{d}

\newcommand{\CTS}{C_*}

\newcommand{\BN}{\mathcal{B}^\nu}
\newcommand{\CN}{\BN_T}
\newcommand{\DN}{\mathcal{D}_T^\nu}

\newcommand{\BldU}{\bm{u}}
\newcommand{\BX}{\bm{X}}
\newcommand{\bX}{\bm{X}}
\newcommand{\BY}{\bm{Y}}
\newcommand{\bY}{\bm{Y}}
\newcommand{\BZ}{\bm{Z}}
\newcommand{\bZ}{\bm{Z}}

\newcommand{\BU}{\bm{U}^{\Theta_0}}
\newcommand{\bU}{\bm{U}^{\Theta_0}}

\newcommand{\bV}{\bm{V}^{\Theta_0}}

\newcommand{\bF}{\bm{F}^{\Theta_0}}

\newcommand{\DAL}{D_\alpha}

\newcommand{\DATX}{\DAL \BX (\theta)}
\newcommand{\DATY}{\DAL \BY (\theta)}

\newcommand{\RO}{\mathcal{R}}
\newcommand{\DO}{\mathcal{P}}
\newcommand{\IO}{\mathcal{I}}

\newcommand{\HEU}{\mathcal{H}}
\newcommand{\HEK}{\mathcal{H}_1}
\newcommand{\HUK}{\mathcal{H}_2}
\newcommand{\HJK}{\mathcal{H}_3}
\newcommand{\WK}{\mathcal{W}}

\newcommand{\BB}{\mathcal{U}_1}
\newcommand{\BK}{\mathcal{U}_2}
\newcommand{\BC}{\mathcal{U}_3}
\newcommand{\BG}{\mathcal{U}_4}
\newcommand{\BPU}{\mathcal{U}}

\newcommand{\JT}{\mathcal{J}}
\newcommand{\LLT}{\mathcal{L}}

\newcommand{\TT}{\mathcal{N}}
\newcommand{\STT}{\mathcal{Z}}
\newcommand{\YTT}{\mathcal{Y}}
\newcommand{\VTT}{\mathcal{V}}

\newcommand{\jj}{j_*}

\newcommand{\DTT}{\mathcal{C}_{1\TE}}
\newcommand{\DDTT}{\mathcal{C}_{2\TE}}
\newcommand{\DDDTT}{\mathcal{C}_{3\TE}}
\newcommand{\DKTT}{\mathcal{C}_{k\TE}}

\newcommand{\KTA}{\mathcal{K}(\theta, \alpha)}
\newcommand{\KXTA}{\mathcal{K}[\bX](\theta, \alpha)}
\newcommand{\HXTA}{\mathcal{K}_0[\bX](\theta, \alpha)}

\newcommand{\DAX}{\DAL \BX }
\newcommand{\DAY}{\DAL \BY }
\newcommand{\DAP}{\delta_\alpha^+ \BX' }
\newcommand{\DAM}{\delta_\alpha^- \BX' }
\newcommand{\DAYP}{\delta_\alpha^+ \BY' }
\newcommand{\DAYM}{\delta_\alpha^- \BY' }
\newcommand{\DAPXY}{\delta_\alpha^+ (\bX'-\bY') }
\newcommand{\DAMXY}{\delta_\alpha^- (\bX'-\bY') }

\newcommand{\DBT}{\overline{D\bT} }
\newcommand{\DDBT}{\overline{D^2\bT} }
\newcommand{\DDTB}{\widetilde{D^2\bT} }
\newcommand{\DDDBT}{\overline{D^3\bT} }
\newcommand{\DBTX}{\overline{D\bT}[\BX'] }
\newcommand{\DBTY}{\overline{D\bT}[\BY'] }

\newcommand{\LDS}{\frac{\lambda}{64}}
\newcommand{\CDS}{\frac{1}{64}}

\newcommand{\MA}{\subw}
\newcommand{\BA}{\mathcal{B}^\MA}

\newcommand{\KCC}{C_3}
\newcommand{\KC}{\KCC \max\{1,M\}}

\newcommand{\TLam}{\widetilde{\Lambda}}

\newcommand{\p}{\partial}
\newcommand{\mc}[1]{\mathcal{#1}}
\newcommand{\abs}[1]{\left\lvert #1 \right\rvert}
\newcommand{\PD}[2]{\frac{\partial#1}{\partial#2}}
\newcommand{\paren}[1]{\left(#1\right)}
\newcommand{\pv}{\text{pv}\hspace{-0.1cm}}
\DeclareMathOperator*{\esssup}{ess\,sup}

\newcommand{\jump}[1]{[\![#1]\!]}

\newcommand{\secref}[1]{\S\ref{#1}}

\begin{document}


%
\title[Critical local well-posedness for the Peskin problem]{Critical local well-posedness for the fully nonlinear Peskin problem}

\author[S. Cameron]{Stephen Cameron$^\dagger$}
\address{$^\dagger$Courant Institute, New York University, New York, NY 10012. \href{mailto:spc6@cims.nyu.edu}{spc6@cims.nyu.edu} }
\thanks{$^\dagger$Partially supported by the NSF grant DMS-1902750 of the USA}

\author[R. M. Strain]{Robert M. Strain$^{\ddagger}$}
\address{$^\ddagger$Department of Mathematics, University of Pennsylvania, Philadelphia, PA 19104, USA. \href{mailto:strain@math.upenn.edu}{strain@math.upenn.edu} (\href{https://orcid.org/0000-0002-1107-8570}{https://orcid.org/0000-0002-1107-8570})}
\thanks{$^\ddagger$Partially supported by the NSF grants DMS-1764177 and DMS-2055271 of the USA}

\begin{abstract}
We study the problem where a one-dimensional elastic string is immersed in a two-dimensional steady Stokes fluid.  This is known as the Stokes immersed boundary problem and also as the Peskin problem.   We consider the case with equal viscosities and with a fully non-linear tension law; this model has been called the fully nonlinear Peskin problem.  In this case we prove local in time well-posedness for arbitrary initial data in the scaling critical Besov space $\dot{B}^{3/2}_{2,1}(\T ; \R^2)$.  We additionally prove the optimal higher order smoothing effects for the solution.  To prove this result we derive a new formulation of the boundary integral equation that describes the parametrization of the string, and we crucially utilize a new cancellation structure.  
\end{abstract}

\thispagestyle{empty}
\maketitle
\tableofcontents

\section{Introduction and main results}\label{sec:firstpart}

The immersed boundary method, as formulated by Peskin in \cite{PeskinThesis1972,PESKIN1972252}, has become a useful and effective method to computationally solve fluid-structure interaction (FSI) problems \cite{MR2009378}.  This method has developed numerous applications in different fields of science \cite{MR2115343,MR1156495}.     And the scientific computing of FSI problems has remained an active area of research \cite{MR2242805,MR2009378,TRYGGVASON2001708,richter2017fluid,Bazilevs2013}. 

The {\em Peskin problem}, considered in this paper, describes the time evolution of an elastic simple closed string immersed in a 2D incompressible Stokes flow.  The string exerts a singular force which generates the flow, and then the configuration of the string evolves over time according to the local fluid velocity.  This model is probably among the simplest FSI problems and it has been used extensively as a test problem in the development of numerical algorithms in addition to being used in physical modeling.  We assume that the string $\Gamma(t)$ splits $\R^2$ into two simply connected domains $\Omega(t)$ (interior) and $\R^2 \backslash \Omega(t)$ (exterior). We shall consider the problem when the viscosities, $\mu_i$, in both fluids are equal, and we set them equal to one for simplicity $\mu_1 = \mu_2 =1$.  Then there are several formulations of this problem, all of which are equivalent assuming we have a sufficiently smooth solution.

The first formulation is at the level of the fluid; for each fixed time $t>0$, both the fluid velocity $\BldU$ and pressure $p$ solve the equations
\begin{equation}\label{formulation.first}
\left\{\begin{array}{ll}
 \Delta \BldU + \nabla p = 0, & x\in \R^2\backslash \Gamma(t), \\ 
\nabla_x \cdot \BldU = 0, & x \in \R^2\backslash \Gamma(t) \\
\BldU,p\to 0, & \text{ as }x\to \infty \end{array} \right.
\end{equation}
We are left to describe the time evolution of $\Gamma(t)$ as well as the appropriate boundary conditions for $\BldU$ and $p$ at $\Gamma(t)$.  Parametrize  $\Gamma(t)$ 
by the Lagrangian coordinate $\theta \in\T=\mathbb{R}/(2\pi\mathbb{Z})=[-\pi, \pi]$, and let $\BX(t,\theta):\T \to \R^2$
denote the coordinate position of $\Gamma$ at time $t$.  Here $\BX=(X_1, X_2)^T$ and $|\BX|^2 \eqdef X_1^2+X^2_2$.
 Then the evolution of $\BX$ is given by 
\begin{equation}\label{eqn.u.formula}
 \partial_t \BX(t,\theta) = \BldU(t,\BX(t,\theta)).
\end{equation}
Define $\jump{w}=\jump{w}(\BX(\theta))$ as the jump across the filament $\Gamma$:
\begin{equation}\notag
\jump{w}(\BX(\theta)) = \lim\limits_{\Omega \ni x\to \bX(\theta)} w(x) -  \lim\limits_{\R^2 \backslash\Omega \ni x\to \bX(\theta)} w(x).
\end{equation}
Then the final boundary conditions for $\BldU$ and $p$ are given by 
\begin{equation}\label{jump.first}
\left\{\begin{array}{l}
\jump{\BldU} =0, \\
\jump{\paren{ \paren{\nabla \bm{u}+(\nabla \bm{u})^{\rm T}}-p\IO}\bm{n}}= \bm{F}_{\rm el}\abs{\p_\theta \BX}^{-1}. \end{array} \right.
\end{equation}
Above $\IO$ is the $2\times 2$ identity matrix and $\bm{n}$ is the outward pointing unit normal vector on $\Gamma$: 
\begin{equation*}
\bm{n}=\begin{bmatrix} 0 & 1 \\ -1 & 0 \end{bmatrix}\widehat{\BX^\prime}, \; \widehat{\BX^\prime} = \frac{\BX^\prime}{\abs{\BX'}}, \; \BX^\prime = \p_\theta \BX=\PD{\BX}{\theta}.
\end{equation*}
Since we will frequently be working with the parameterization $\BX$ at fixed times, we will often omit the time variable and denote derivatives in $\theta$ of $\BX$ as $\BX^\prime$.  
Lastly we denote 
$\bm{F}_{\text{el}}$
as the elastic force exerted by the string $\Gamma$. In the case that the elastic string obey’s Hooke’s law, we have a simple tension given by:
\begin{equation}\label{linearF}
\bm{F}_{\rm el}=k_0\p_\theta^2 \BX, \quad k_0>0,
\end{equation}
where $k_0$ is the elasticity constant of the string $\Gamma(t)$.  The general tension force law is given by
\begin{equation}\label{nonlinearF}
 \bm{F}_{\rm el}=\p_\theta\paren{\mc{T}(\abs{\p_\theta \BX})\frac{\p_\theta\BX}{\abs{\p_\theta \BX}}}
 \end{equation}
This is also called the fully nonlinear force law in \cite{rodenberg_thesis}.  Here $\TE(s)$ is a coefficient modeling the elastic tension in the filament that satisfies the structure condition $\TE>0$ and $d\TE/ds>0$.
 Note that \eqref{nonlinearF} is reduced to \eqref{linearF} if we take $\TE(s)=k_0s$, hence $k_0=\TE(1)=d\TE/ds$.

The set of equations \eqref{formulation.first}-\eqref{eqn.u.formula}-\eqref{jump.first} above was first proposed as a simplified model to study blood flow through heart valves \cite{PeskinThesis1972,PESKIN1972252}.  A second equivalent formulation of  \eqref{formulation.first}-\eqref{eqn.u.formula}-\eqref{jump.first} is the following immersed boundary formulation 
\begin{equation}\label{immersed.formulation}
\Delta \BldU +\nabla p = \int_{\T}  \partial_\theta \left(\TE(|\partial_\theta \BX|) \frac{\partial_\theta \BX}{|\partial_\theta \BX|}\right) \delta(x - \BX(\theta)) d\theta, \qquad \nabla\cdot \BldU = 0,
\end{equation}
which is very useful for numerical analysis. Then \eqref{immersed.formulation} combined with \eqref{eqn.u.formula} allows us to discretize the fluid domain in $x$ and the elastic string in $\theta$ independently of each other, with all communication between the two domains coming from the singular forcing of the fluid in  \eqref{immersed.formulation}, and the time evolution of the string in \eqref{eqn.u.formula}. This became the basis for the immersed boundary method, which has been applied to numerous problems and is of great use in applications \cite{MR1156495}.

The third formulation which we will primarily be using is the following boundary integral formulation for the general force law \eqref{nonlinearF}:
\begin{equation}\label{e:boundaryintegral}
\partial_t \BX(\theta) = \int_{\T} G(\delta_\alpha \BX(\theta) )\partial_\alpha \left(\TE(|\BX'(\theta+\alpha) |) \frac{\BX'(\theta+\alpha)}{|\BX'(\theta+\alpha)|}\right) d\alpha.
\end{equation}
Here, for a generic function $f:\T \to \R^2$, we define the standard partial difference operator by
    \begin{equation}\label{delta.notation}
        \delta_\alpha f(\theta) \eqdef f(\theta+\alpha)- f(\theta).
    \end{equation}
For $z\in \R^2$, then $G(z)$ is the Stokeslet given by
\begin{equation}\label{stokeslet.def}
G(z) = G_1(z) + G_2(z), 
\quad G_1(z) \eqdef - \frac{1}{4\pi}\log(|z|) \IO, 
\quad  G_2(z) \eqdef  \frac{1}{4\pi}\frac{z\otimes z}{|z|^2}.
\end{equation}
Notice that in the simple tension case \eqref{linearF} the equation \eqref{e:boundaryintegral} takes the form
\begin{equation}\notag
\partial_t \BX(\theta) = k_0 \int_{\T} G(\delta_\alpha \BX(\theta) ) \partial_\alpha^2 \BX(\theta+\alpha) d\alpha,
\end{equation}
which contains the second order derivative $\partial_\alpha^2 \BX$ inside the equation.
We also define
\begin{equation}\label{distance.X.notation}
    \DATX \eqdef \frac{\delta_\alpha  \BX (\theta)}{ \alpha}.
\end{equation}
Then we introduce the arc-chord number 
\begin{equation}\label{arc.cord.number}
    |\BX|_*  \eqdef \inf\limits_{\theta, \alpha\in \T,\alpha \ne 0} |\DATX|.
\end{equation}   
The evolution equation \eqref{e:boundaryintegral} is then is well-defined for a sufficiently regular function $\BX(t,\theta)$ that satisfies $|\BX(t)|_*>0$.  If the parametrization $\BX(t,\theta)$ is sufficiently regular, it has been proven that all three formulations \eqref{formulation.first}-\eqref{eqn.u.formula}-\eqref{jump.first}, \eqref{eqn.u.formula}-\eqref{immersed.formulation}, and \eqref{e:boundaryintegral} are equivalent \cite{MR1808257}.  Considering the importance of the Peskin problem in applications, establishing the existence of smooth solutions is vitally important in order to guarantee that various numerical methods based on different formulations of
the problem all approximate the same solution.

The Peskin problem has several known similarities with the Muskat problem.  The Muskat problem is also a free boundary problem that can be written in a boundary integral formulation \cite{MR2472040}.  Also, both systems satisfy an energy balance law \cite{MR3935476,CCGS13,CCGRS16}.  Further both equations have the invariant scaling $g_\lambda(t,\theta) = \lambda^{-1}g(\lambda t,\lambda \theta)$ (see also \secref{sec:scaling}).   Lastly, both systems of equations can be written in the form 
\begin{equation}\notag
\partial_t g + (-\Delta)^{\frac12}g = \mathfrak{R},
\end{equation}
with a ``remainder'' term $\mathfrak{R}$.  For the Peskin problem $g=\bX(t,\theta)$ and the remainder is $\mathfrak{R} = \mathcal{R}(t,\theta)$ as in \eqref{e:peskinapprox} below.  Recently there has been a large amount of research work studying the local- and global-in-time well-posedness for the Muskat problem \cite{CCG11,FlynnNguyen2020,NguyenST2019,NguyenPausader2019,2103.14535,AlaNgu2021lipsh,2009.08442,2010.06915,CCGS13,CCGRS16,Cam17,Cam20,CGS16,MR3639321} and break-down \cite{CCFGL12}.  This work was motivated by recent results on scaling critical local-in-time  well-posedness for the Muskat problem in \cite{AlaNgu2021lipsh,2009.08442,2010.06915}, as well as recent analytical work on the Peskin problem in \cite{MR3935476,MR3882225}.

Analytical study of the Peskin problem began very recently, with all but one paper focussing on the case of simple tension $\TE(r) = k_0 r$ in \eqref{linearF}.  Lin and Tong were able to prove local well-posedness for the boundary integral formulation \eqref{e:boundaryintegral} with initial data $\bX_0\in H^{5/2}(\T; \R^2)$ using energy methods and the Schauder fixed point theorem \cite{MR3882225}.  At the same time Mori, Rodenberg, and Spirn  proved local well-posedness for initial data $\bX_0\in C^{1,\gamma}(\T;\R^2)$ for any $0<\gamma<1$ using semigroup theory \cite{MR3935476}.  In particular the result of \cite{MR3935476} is barely subcritical, but the semi-group approach used in the proof makes a scaling critical result difficult.  The only equilibrium states are uniformly parametrized circles \cite{MR3935476}, and both groups were able to prove global well-posedness and exponential convergence to equilibrium for initial data sufficiently close to a circle \cite{MR3882225,MR3935476}.  Additionally, \cite{MR3935476} was able to prove that solutions to the Peskin problem immediately become $C^\infty$ for positive time, and that if $\bX(t)$ blows up in finite time, either a chord arc condition fails or the $C^{1,\gamma}$ norm must blow up for any small $\gamma>0$.  Tong \cite{1904.09528} then further established the global well-posedness of the regularized Peskin problem and proved convergence as the regularization parameter diminishes.  

Regarding scaling critical initial data for \eqref{e:boundaryintegral}, recently Garc\'ia-J\'uarez, Mori, and Strain were able to prove global well-posedness if the initial data is sufficiently close to a uniformly parametrized circle in the Wiener algebra $\dot{\mathcal{F}}^{1,1} : = \{f:\T\to \R^2| \sum\limits_{k\in \Z^2} |k| |\hat{f}(k)| <\infty\}$.  This result uses the spectral decomposition of the linearized operator \cite{2009.03360}, and it holds even in the case that the interior and exterior fluids have different viscosities--it is the first analytical result in that case. Recently, Gancedo, Belinch\'{o}n and Scrobogna \cite{2011.02294} studied a toy model of the Peskin problem and proved global existence and uniqueness in the critical Lipschitz space.  Then more recently, Chen and Nguyen were able to prove local well-posedness for \eqref{e:boundaryintegral} whenever $\bX_0'$ is in VMO using estimates on the fundamental solution of $(-\Delta)^{\frac12}$ and interpolation results, and they further prove global existence when $\bX_0'$ is in BMO for initial data that is close to equilibrium \cite{2107.13854}.

The previously mentioned results in a sense rely on rewriting \eqref{e:boundaryintegral} with \eqref{linearF} as 
\begin{equation}\label{e:peskinapprox}
\partial_t \bX(t,\theta) + (-\Delta)^{\frac12}\bX(t,\theta) = \mathcal{R}(t,\theta),
\end{equation}
for some remainder $\mathcal{R}$.  And then controlling this remainder further requires controlling the derivative $\bX'$.  These results then make use of properties that are particular to the fractional heat equation such as the fundamental solution, the semigroup property, and the spectral decomposition.  Once we consider a general tension $\mathcal{T}$ as in \eqref{nonlinearF} though, we lose access to the full power of these properties, and major alterations to the approach are needed.  The only paper before which that has dealt with a general tension is Rodenberg's thesis \cite{rodenberg_thesis}.  By localizing around the initial data, Rodenberg was able to apply the semigroup method from \cite{MR3935476} again and prove local existence when $\bX_0, \mathcal{T}\in C^{1,\gamma}$.  However, the result is weakened because the approach to localizing the intial data  and thereby patching the semigroup method in \cite{rodenberg_thesis} didn't allow to also prove the smoothing effects, only guaranteeing that the solution $\bX(t)$ remains in $C^{1,\gamma}$ even if the initial data and tension are $C^\infty$.

In order to further develop the fully nonlinear case \eqref{nonlinearF}, it's vital to understand exactly how the addition of a nonlinear tension $\mathcal{T}$ changes the problem.  In particular, its important to understand how this affects the evolution of the derivative $\bX'$, as the regularity of and behavior of the remainder $\mathcal{R}$ in \eqref{e:peskinapprox} has been controlled by that.  In this article, we propose a new representation of the boundary integral equation for the problem.  We write the equation \eqref{e:boundaryintegral} in the following equivalent formulation that will cancel out the terms featuring an $\partial_\alpha^2 \BX = \BX''$.  In \eqref{e:boundaryintegral}  we integrate by parts against $G_1(z)$ while leaving $G_2(z)$ alone to obtain
\begin{equation}\notag
\begin{split}
\partial_t \BX(\theta) =& \int_{\T} \partial_\alpha \left( \frac{\TE(|\BX'|)}{|\BX'|} \partial_\alpha(G_1(\delta_\alpha \BX))\right) \delta_\alpha \BX(\theta) d\alpha
\\ &+ \int_{\T} G_2(\delta_\alpha \BX) \partial_\alpha \left(\TE(|\BX'(\theta+\alpha) |) \frac{\BX'(\theta+\alpha)}{|\BX'(\theta+\alpha)|}\right)d\alpha 
\\
 =& \frac{1}{4\pi}\int_{\T} \frac{2 \left(\BX'(\theta+\alpha)\cdot \frac{\delta_\alpha \BX}{|\delta_\alpha \BX|} \right)^2 - |\BX'(\theta+\alpha)|^2}{ |\delta_\alpha \BX|^2} \frac{\TE(|\BX'(\theta+\alpha)|)}{|\BX'(\theta+\alpha)|} \delta_\alpha \BX(\theta) d\alpha.
\end{split}
\end{equation}
The  calculation is performed in full detail in \secref{sec:GenTenDerivation}.

This property of the cancellation of the highest order derivatives is also satisfied by the equation for $\bX'(t,\theta)$.  Let $\bX(t,\theta)$ be the solution of the Peskin problem with initial data $\bX_0$ and tension $\TE$, Then  $\bX'(t,\theta)$ solves the following equation 
\begin{equation}\label{peskin.general.tension}
    \partial_t \BX'(\theta) = \int_{\T} \frac{d\alpha}{\alpha^2}~ \mathcal{K}[\BX](\theta, \alpha)\delta_\alpha \bT(\BX'(\theta)),
\end{equation}
where $\bT:\R^2\to \R^2$ is the tension map 
\begin{equation}\label{tension.map.def}
    \bT(z) \eqdef \TE(|z|)\hat{z}, \quad z \in \R^2.
\end{equation}
Here the kernel $\KTA=\KXTA$ is given by
\begin{multline}\label{kernel.peskin.noExpand}
\KXTA \eqdef \frac{1}{4\pi} \frac{\BX'(\theta+\alpha) \cdot \DO(\DATX)\BX'(\theta)}{|\DATX|^2} \IO
    \\
    -\frac{1}{4\pi} \frac{\BX'(\theta+\alpha) \cdot \RO(\DATX)\BX'(\theta)}{|\DATX|^2} \RO(\DATX) 
    \\ 
    +\frac{1}{4\pi} \frac{\BX'(\theta+\alpha) \cdot (\DO(\DATX)-\IO)\BX'(\theta)}{|\DATX|^2} \DO(\DATX).
\end{multline} 
Again $\IO$ is the identity matrix on $\R^2$.  Also  the reflection matrices $\RO$ and $\DO$ are defined $\forall z \in \R^2$ by 
\begin{equation}\label{matrix.operators}
  \RO(z) \eqdef \hat{z}\otimes \hat{z}^\perp + \hat{z}^\perp\otimes \hat{z}, \quad  \DO(z) \eqdef \hat{z}\otimes \hat{z} - \hat{z}^\perp \otimes \hat{z}^\perp ,
\end{equation}
where $\hat{z}^\perp\in \R^2 $ is the unit vector perpendicular to $\hat{z}$.  We remark that the three matrices $\IO$, $\RO(z)$, $\DO(z)$ are mutually orthogonal in $\R^4$ and form a basis for the 2 by 2 symmetric matrices for any fixed value of $z\in \R^2\setminus \{0\}$.  This representation of the equation \eqref{peskin.general.tension} for the evolution of $\bX'(t,\theta)$ is fundamental to the analysis in the remainder of this article.  Equation \eqref{peskin.general.tension} is derived in \secref{sec:GenTenDerivation}.

Now, recalling \eqref{distance.X.notation}, to further expand out the additional cancellation in the kernel $\KXTA$ we introduce the notation 
\begin{equation}\label{delta.pm.notation}
   \delta_\alpha^+ \bX'(\theta) \eqdef \bX'(\theta+\alpha) - \DATX, \quad 
   \delta_\alpha^- \bX'(\theta) \eqdef \bX'(\theta) - \DATX.
\end{equation}
Then it is an important observation that the kernel $\KXTA$ from \eqref{kernel.peskin.noExpand} can be expressed as the following matrix valued function 
\begin{equation}\label{kerbel.eqn.deriv}
  \mathcal{K}[\bX](\theta, \alpha)   = \frac{1}{4\pi}\IO + \mathcal{A}[\bX](\theta, \alpha),  
\end{equation}
where
\begin{multline} \label{kerbel.A.eqn.deriv}
 4\pi\mathcal{A}[\BX](\theta, \alpha)  \eqdef \frac{\delta_\alpha^+ \bX'(\theta) \cdot \DO(\DATX) \delta_\alpha^- \bX'(\theta)}{|\DATX|^2}\IO
\\
+  \frac{(\delta_\alpha^+ \bX'(\theta)+\delta_\alpha^- \bX'(\theta)) \cdot \DO(\DATX) \DATX}{|\DATX|^2} \IO
\\ 
-\frac{\delta_\alpha^+ \bX'(\theta) \cdot \RO(\DATX) \delta_\alpha^- \bX'(\theta)}{|\DATX|^2}
\RO(\DATX)
\\
- \frac{(\delta_\alpha^+ \bX'(\theta)+\delta_\alpha^- \bX'(\theta)) \cdot \RO(\DATX) \DATX}{|\DATX|^2} \RO(\DATX)
\\ 
+ \frac{\delta_\alpha^+ \bX'(\theta) \cdot (\DO(\DATX) - \IO) \delta_\alpha^- \bX'(\theta)}{|\DATX|^2} \DO(\DATX).
\end{multline}
This expression follows after taking into account the orthogonality in  \eqref{matrix.operators}.

Then \eqref{peskin.general.tension} can be written as 
\begin{equation}\label{peskin.expand.tension}
    \partial_t \BX'(\theta) -
    \frac{1}{4\pi}\int_{\T} \frac{d\alpha}{\alpha^2}~ \delta_\alpha \bT(\BX'(\theta))= \int_{\T} \frac{d\alpha}{\alpha^2}~ \mathcal{A}[\BX](\theta, \alpha)\delta_\alpha \bT(\BX'(\theta)),
\end{equation}
The expression $\frac{1}{4\pi}\int_{\T} \frac{d\alpha}{\alpha^2}~ \delta_\alpha \bT(\BX'(\theta))$ motivates our definition of $\TLam$ in \eqref{tildeLambda:eq}. 
Then, due to the higher order cancellation of $\mathcal{A}[\BX](\theta, \alpha)$ as in \eqref{kerbel.A.eqn.deriv}, for small $\alpha$ the integrand for the equation \eqref{peskin.general.tension} using \eqref{kerbel.eqn.deriv} is approximately 
\begin{equation}\notag
    \frac{\mathcal{K}[\BX](\theta, \alpha)}{\alpha^2} \delta_\alpha \bT(\BX'(\theta))\approx \frac{\delta_\alpha \bT(\bX')}{4\pi\alpha^2}.  
\end{equation}
Thus a basic model equation for the general tension equation  \eqref{peskin.general.tension} would be a vector version of the fractional porous medium equation: \begin{equation}\notag
    \partial_t \bm{U} = -(-\Delta)^{\frac12} \bT(\bm{U}).
\end{equation}
To the best of our knowledge, this equation has not been studied before, though both the scalar fractional version \cite{FractionalPorousMedium1, FractionalPorousMedium2, MR3656476} and local vector valued \cite{VectorPorous1, VectorPorous2, VectorPorous3} have been studied.  Then the positivity and monotonicity assumptions that we will make on the tension $\TE$ are both physically motivated, as well as the same assumptions that typically appear on the porous media equation in order to ensure ``ellipticity" for the problem such as in \cite{MR3656476}.

\subsection{Scaling}\label{sec:scaling} For the Peskin problem \eqref{peskin.general.tension} in general for any $\lambda>0$ the  rescaling $\BX_{\lambda}(t, \theta) = \lambda^{-1} \BX(\lambda t, \lambda \theta)$ leaves the equation invariant for an arbitrary tension $\TE$ in \eqref{tension.map.def}.  If the tension takes the form of a power law $\TE(r) = r^{1+\gamma}$ for some $\gamma \ge 0$ then the Peskin problem has the additional rescaling $\BX^r(t,\theta) = r \BX(r^\gamma t, \theta)$.  In the case of a simple tension $\TE(r) = k_0 r$, there's a two dimensional family of rescaling $\bX_{\tau, \lambda}(t,\theta) = \tau \bX(\lambda t, \lambda \theta)$, where $\tau\in \R$ and $\lambda>0$ are independent of each other.  To ensure that the arc-chord condition  \eqref{arc.cord.number} also remains invariant then we are limited to the rescaling $\BX_{\lambda}(t, \theta) = \lambda^{-1} \BX(\lambda t, \lambda \theta)$.

Here we give a list of some scaling critical spaces for the Peskin problem \eqref{peskin.general.tension} under the rescaling $\BX_{\lambda}(t, \theta) = \lambda^{-1} \BX(\lambda t, \lambda \theta)$: the Lipshitz space $\dot{W}^{1,\infty}$, the Wiener algebra $\mathcal{A}^1$, $BMO^1$, and the homogeneous Besov spaces $\dot{B}^{1+\frac{1}{p}}_{p, r}$ for all $p, r \in [1,\infty]$.  In particular we emphasize the spaces   $\dot{B}^{\frac32}_{2,r}$ for $1 \le r \le \infty$ and $\dot{H}^{\frac{3}{2}}$ due to their $L^2$ structure.

In this paper we utilize the scaling critical Banach space $\dot{B}^{\frac32}_{2,1}$ since it has a clearly defined $L^2$ based structure, and then hopefully it might be useful also in the further development and study of numerical methods.

\subsection{Notation}\label{sec:normNew}
We use $C>0$ to denote some inessential constant whose value may change from line to line.  We will write $A \lesssim B$ if $A \le C B$.  We also write $A \approx B$ if both $A \lesssim B$ and $B \lesssim A$ hold.    We will use $f:\T \to \R^2$ or $\C$ to denote a generic smooth function throughout this paper, where $f=(f_1, f_2)$ and $|f|^2 \eqdef f_1^2+f^2_2$.    We also define the translation operator $\tau_\beta$ applied to the $\theta \in \T$ variable by
\begin{equation}\label{def.translation}
    \tau_\beta f(\theta) \eqdef f(\theta+\beta).
\end{equation}
We define $\1_{A}$ as the standard indicator function of the set $A$.  We use the notation $\delta_\beta$ for the difference operator \eqref{delta.notation} frequently.

We will use the standard notation for the $L^p(\T)$ spaces as
\begin{equation}\notag
    || f||_{L^p(\T)}=|| f||_{L^p_\theta} \eqdef \left( \int_{\T}  |f(\theta)|^p d\theta\right)^{1/p}, \quad 1 \le p < \infty.
\end{equation}
In this function space, and in all the functional spaces below, we use the standard generalization to $p=\infty$ as
\begin{equation}\notag
    || f||_{L^{\infty}(\T)}\eqdef \esssup_{\theta \in \T}  |f(\theta)|.
\end{equation}
We will also use the temporal spaces 
\begin{equation}\notag
    || f||_{L^p([0,T])}= || f||_{L^p_T} \eqdef \left( \int_{0}^{T} |f(t)|^p dt\right)^{1/p}, \quad 1 \le p < \infty.
\end{equation}
We define the $L^q_TL^p_\theta$ mixed Lebesgue space norms for $1\le p,q \leq \infty$ as follows:
\begin{equation*}
||f||_{L^q_TL^p_\theta}
=
||f||_{L^q_T(L^p_\theta)}
\eqdef
  \big|\big|  || f(\cdot, \cdot)||_{L^p(\T)} \big|\big| _{L^q([0,T])}.
\end{equation*}
Next  we introduce the Besov spaces as follows
\begin{equation}\label{Besov.Space}
    ||f||_{\dot{B}^s_{p,r}} \eqdef \left(\int_{\T}  \frac{d\beta}{|\beta|}  \left(\frac{||\delta_\beta f||_{L^p(\T)}}{|\beta|^{s}} \right)^{r}  \right)^{1/r}.
\end{equation}
Unless otherwise stated, all indicies in the rest of this section are for $0<s<1$ and $p,q,r\in [1,\infty]$. When $r=\infty$ we use  
\begin{equation}\notag
    ||f||_{\dot{B}^s_{p,\infty}} \eqdef  \esssup_{\beta \in \T}\left(\frac{||\delta_\beta f||_{L^p(\T)}}{|\beta|^{s}} \right).
\end{equation}
In the rest of this paper for simplicity when we write $\sup_{\theta \in \T}$ or $\sup_{0 \le t \le T}$ we mean it to be the standard essential supremum.

We will then also use the standard Sobolev spaces that can be defined as 
\begin{equation}\notag
||f||_{\dot{H}^{s}} \eqdef
||f||_{\dot{B}^s_{2,2}}, \quad \forall s \in \R.
\end{equation}    
Technically to define $\dot{B}^s_{2,2}$ in particular $\forall s \in \R$ we use the definition in Remark \ref{rem:besov.define}.

We will also use the Chemin-Lerner \cite{CL1995} mixed regularity spaces as described for example in \cite[Definition 2.67 on page 98]{BCD} that are defined as 
    \begin{equation}\label{Besov.CL.Space}
    || f||_{\widetilde{L}^{q}_T(\dot{B}_{p, r}^{s})}
    \eqdef
       \left( \int_{\T}  \frac{d\beta}{|\beta|} \frac{|| \delta_\beta f||_{L^{q}_T(L^p_\theta)}^r}{|\beta|^{sr}} 
       \right)^{1/r}.
    \end{equation}
Next, motivated by \cite{AlaNgu2021lipsh,2009.08442,2010.06915}, we introduce periodic Besov spaces with additional regularity on the logarithmic scale for $0<s<1$ and $p,r\in [1,\infty]$ as
\begin{equation}\label{Besov.mu.Space}
    ||f||_{\dot{B}^{s,\mu}_{p,r}} \eqdef \left(\int_{\T}  \frac{d\beta}{|\beta|}  \left(\subw(|\beta|^{-1})\frac{||\delta_\beta f||_{L^p(\T)}}{|\beta|^{s}} \right)^{r}  \right)^{1/r}.
\end{equation}
Here the log scale derivative $\subw$ is defined as follows:

\begin{definition}\label{subw.definition}
We consider functions  $\subw\colon[0,\infty) \to [1,\infty)$ which satisfy the following three assumptions: 
\begin{itemize}
\item $\subw(r)$ is increasing and $\lim_{r\to\infty}\subw(r)=\infty$.
    \item There is a $c_0>0$ such that $\subw(2r)\leq c_0\subw(r)$ for any $r\geq 0$. 
    \item The function $r\mapsto \subw(r)/\log(4+r)$ is decreasing on  $[0,\infty)$.
\end{itemize}
\end{definition}    

Then we similarly define 
\begin{equation}\label{Besov.mu.CL.Space}
    || f||_{\widetilde{L}^{q}_T(\dot{B}_{p, r}^{s,\subw})}
    \eqdef
       \left( \int_{\T} \frac{d\beta}{|\beta|}  \left(\subw(|\beta|^{-1})\frac{||\delta_\beta f||_{L^{q}_T(L^p_\theta)}}{|\beta|^{s}} \right)^{r} 
       \right)^{1/r}.
\end{equation}
We introduce streamlined notation for the main norms used in the paper
\begin{equation}\label{C.space.temporal}
    || f||_{\CM}
  \eqdef
       || f||_{\widetilde{L}^{\infty}_T(\dot{B}_{2, 1}^{\frac12,\subw})}
=
\int_{\T} \frac{d\beta}{|\beta|^{3/2}} \subw(|\beta|^{-1}) || \delta_\beta f||_{L^{\infty}_T(L^2_\theta)},
    \end{equation}
    and 
   \begin{equation}\label{D.space.temporal}
    || f||_{\DM}
    \eqdef
     || \TLam^{\frac12} f||_{\widetilde{L}^2_T(\dot{B}^{\frac12,\mu}_{2,1})}
     =
\int_{\T} \frac{d\beta}{|\beta|^{3/2}} \subw(|\beta|^{-1})
 ||\TLam^{\frac12} \delta_\beta f||_{L^2_T(L^2_\theta)}.
    \end{equation}
Above the operator $\TLam$ is a constant multiple of $\Lambda \eqdef (-\Delta)^{\frac12}$ and is defined precisely in \eqref{tildeLambda:eq} in \secref{sec:para}.  
Further from \eqref{tildeLambda:eq} we have
\begin{equation}\notag
     || \TLam^{\frac12} f||_{L^2_\theta}^2
     =
     \int_{\T}  d\theta ~  f(\theta) \cdot \TLam f(\theta)
     =
     \frac{1}{8\pi}\int_{\T}  d\theta \int_{\T}  \frac{d\alpha}{\alpha^{2}}
     ~
    | \delta_\alpha f(\theta)|^2.
\end{equation}
This can be taken as the definition of $|| \TLam^{\frac12} f||_{L^2_\theta}$ in \eqref{D.space.temporal}. 
For the initial data we will use the following norm:
    \begin{equation}\label{initial.B.space}
    || f||_{\BS}
   \eqdef
           || f||_{\dot{B}_{2, 1}^{\frac12,\subw}}
    =
\int_{\T} \frac{d\beta}{|\beta|^{3/2}} \subw(|\beta|^{-1}) || \delta_\beta f||_{L^2_\theta}.
    \end{equation}
Lastly we have 
$||f||_{L^\infty_T(\dot{B}^{1/2, \mu}_{2,1})}
   \leq
   || f||_{\CM},$ and this
 inequality shows that the norm $|| f||_{\CM}$ is stronger than  $||f||_{L^\infty_T(\dot{B}^{1/2, \mu}_{2,1})}$.  We will also use the standard definitions of the H{\"o}lder spaces $C^{k,\gamma}$.

\subsection{Main results}\label{sec:mainResults} Without loss of generality we can suppose initially that $\bX_0'$ has mean zero since the equation \eqref{e:boundaryintegral} and the equation \eqref{peskin.general.tension} both annihilate constants.  Therefore, this mean zero property will be preserved by the solution.  Next we give definitions of our notions of solution.

\begin{definition}\label{def:solution} (Weak solution) Let $\bX_0'\in \dot{B}^{\frac{1}{2}}_{2,1}(\T ; \R^2)$ with $|\bX_0|_*>0$.   We say that $\bX: [0,T]\times \T\to \R^2$ is a weak solution of the Peskin problem \eqref{peskin.general.tension} with tension $\TE$ and initial data $\bX_0$ if $\bX', \bT(\bX')\in L^2_T (L^\infty_\theta \cap \dot{H}^{\frac12}_\theta)$ with $\inf_{0 \le t \le T} |\bX(t)|_*>0$, and for any function $\bY: [0,T]\times \T\to \R^2$ with $\bY'\in L^2_T (L^\infty_\theta \cap \dot{H}^{\frac12}_\theta)$ and $\partial_t Y'\in L^2_T (L^\infty_\theta\cap \dot{H}^{\frac12}_\theta)^*$, we have
\begin{multline}\notag
    \int_\T d\theta ~ \bY'(T,\theta)\cdot \bX'(T,\theta) - \int_\T d\theta~ \bY_0'(\theta)\cdot \bX'_0(\theta) =
    \int_0^T dt \int_\T  d\theta~ \partial_t \bY'(t,\theta)\cdot \bX'(t,\theta) 
    \\
    - \frac12 \int_0^T dt \int_\T d\theta\int_\T   \frac{d\alpha}{\alpha^2} ~ \delta_\alpha \bY'(t) \cdot \mathcal{K}[\BX(t)](\theta, \alpha)\delta_\alpha \bT(\bX'(t)).
\end{multline}
\end{definition}
\begin{remark}
Our definition of a weak solution can be accurately paraphrased as the weakest notion of distributional solution such that $\BX'$ is a valid test function for itself.  This is chosen in order to justify the calculations of our main a priori estimate in \secref{sec:BesovSpace}.
\end{remark}

\begin{definition}\label{def:StrongSolution} (Strong solution) Let $\bX_0'\in \dot{B}^{\frac{1}{2}}_{2,1}(\T ; \R^2)$ with $|\bX_0|_*>0$.
We say that $\bX: [0,T]\times \T\to \R^2$ is a strong solution if $\bX\in C^2((0,T]\times \T\to \R^2)$ solves the equation \eqref{peskin.general.tension} pointwise with $\inf_{0 \le t \le T} |\bX(t)|_*>0$ and 
\begin{equation}\notag
    \lim\limits_{t\to 0} ||\bX'(t)-\bX'_0||_{L^\infty} = 0.
\end{equation}
\end{definition}

\begin{theorem}\label{thm:main}
Let $\bX_0: \T\to \R^2$ with $\bX'_0\in \dot{B}^{\frac{1}{2}}_{2,1}$ and $|\bX_0|_*>0$. 
Let the scalar tension $\TE:(0,\infty)\to (0,\infty)$ be such that $\TE\in C^{1,1}_{loc}(0,\infty)$ with $\TE'(r)>0$ for all $0<r<\infty$. 
 Then there is a time $T>0$ such that there exists a unique weak solution $\bX:[0,T]\times \T\to \R^2$ to the Peskin problem in the sense of Definition \ref{def:solution}, which is also a strong solution to the Peskin problem \eqref{peskin.general.tension} as in Definition \ref{def:StrongSolution}.
Furthermore for any $0<\beta<1$, $\bX\in C^{2,\beta}_{loc}((0, T]\times \T;\R^2)$.
 Additionally, if $\TE\in C^{k,\gamma}_{loc}(0,\infty)$ for some $k\geq 2$ and $0<\gamma<1$ then we have that $\bX\in C_{loc}^{k+1,\gamma}((0, T]\times \T;\R^2)$.
\end{theorem}

Note that due to the structure of equation \eqref{peskin.general.tension}, $\bX\in C^{k+1,\gamma}_{loc}$ is the optimal regularity for $\TE\in C^{k, \gamma}_{loc}$. We prove Theorem \ref{thm:main} by first establishing a quantitative version under more restrictive assumptions on the tension.

\begin{theorem}\label{thm:mainquant}(Quantitative existence)
Consider initial data $\bX_0: \T\to \R^2$ such that $||\bX'_0||_{\dot{B}^{\frac{1}{2}, \mu}_{2,1}}\leq M$ for some $\mu$ satisfying Definition \ref{subw.definition}, for any $M>0$, and $|\bX_0|_*>0$.  Let the tension map $\bT:\R^2\to \R^2$ from \eqref{tension.map.def} be such that $D\bT\in W^{1,\infty}(\R^2; \R^{2\times 2})$ satisfying the ellipticity condition $D\bT(z)\geq \lambda \IO>0.$

 Then there exists a time $T>0$ depending only on $M$, $\mu$, $|\bX_0|_*$, $\lambda$ and $||D\bT||_{W^{1,\infty}}$ such that there exists a strong solution, in the sense of Definition \ref{def:StrongSolution}, $\bX: [0,T]\times \T\to \R^2$ to the Peskin problem \eqref{peskin.general.tension} with tension $\bT$ and initial data $\bX_0$.  This solution satisfies for some universal constant $c>0$ that
 \begin{multline}\label{e:primaryestimate}
    \int_{\T}\frac{d\beta}{|\beta|^{\frac32}}\mu(|\beta|^{-1})\left( ||\delta_\beta \bX'||_{L^\infty_T L^2_\theta} + c\sqrt{\lambda} ||\delta_\beta (-\Delta)^{\frac14} \bX'||_{L^2_T L^2_\theta}\right) 
    \\
    \leq 4 \int_{\T}\frac{d\beta}{|\beta|^{\frac32}} \mu(|\beta|^{-1}) ||\delta_\beta \bX'_0||_{L^2_\theta}.
\end{multline}
 Further for any small time $\tau>0$ and any $0<\beta<1$, $\bX\in C^{2,\beta}([\tau, T]\times \T;\R^2)$, with its norm depending only on $\tau, \beta, $ and the previously mentioned constants.  

If we additionally have that $\bT\in C^{k,\gamma}(\R^2;\R^2)$ for some $k\geq 2$ and $0<\gamma<1$, then for any small time $\tau>0$, $\bX\in C^{k+1,\gamma}([\tau, T]\times \T;\R^2)$ with the $C^{k+1,\gamma}$ norm controlled by $M$, $\mu$, $|\bX_0|_*$, $\lambda$, $\gamma$, $||\bT||_{C^{k,\gamma}}$, and $\tau$.  
\end{theorem}

\begin{remark}\label{pointwise.rk} Since in Theorem \ref{thm:main} and Theorem \ref{thm:mainquant} we have that 
 $\bX\in C^{2,\beta}([\tau, T]\times \T;\R^2)$ for any $\tau>0$ and any $0 < \beta < 1$ then the calculation in \secref{sec:GenTenDerivation} can be reversed, and we have that $\bX(t,\theta)$ solves both \eqref{peskin.general.tension} and \eqref{e:boundaryintegral} pointwise for any $t>0$.
\end{remark}

\begin{theorem}\label{first:unique}
(Uniqueness)
Consider $\bX_0$ and $\bY_0$ such that $\bX'_0, \bY'_0\in \dot{B}^{\frac{1}{2}, \subw}_{2,1}(\T; \R^2)$ with $|\bX_0|_*>0$ and $|\bY_0|_*>0$.    Let the tension map $\bT:\R^2\to \R^2$ satisfy the same conditions as in Theorem \ref{thm:mainquant} and consider the corresponding solutions $\bX, \bY: [0,T]\times \T\to \R^2$.  
Choose any $\omega(r)$ satisfying Definition \ref{subw.definition} such that  there exists $r_* \ge 1$ so that $\frac{\omega(r)}{\subw(r)}$ is decreasing for $r \ge r_*$ and 
$\displaystyle\lim\limits_{r\to \infty}\frac{\omega(r)}{\subw(r)} = 0$.  
For any $\varepsilon>0,$ there exists $\delta_*>0$ such that for any $0<\delta\le \delta_*$ then   $||\bX'_0-\bY'_0||_{L^2_\theta}<\delta$ implies
\begin{equation}\label{stability.est.nu}
    \int_{\T}\frac{d\beta}{|\beta|^{\frac32}}\omega(|\beta|^{-1}) ||\delta_\beta (\bX'-\bY')||_{L^\infty_T L^2_\theta} <\varepsilon.
\end{equation}
In particular if $||\bX'_0-\bY'_0||_{L^2_\theta}=0$ then the solution is unique in $\mathcal{B}^{\omega}_T$.
\end{theorem}

\begin{remark}
In \eqref{stability.est.nu} we can take for example $\omega(r) = \mu(r)^\gamma$ for any $0<\gamma<1$.  
\end{remark}

\begin{remark}
Note that if $\BX_0', \bY_0' \in \dot{B}^{\frac12}_{2,1}$, then there exists some function $\subw$ satisfying the Definition \ref{subw.definition} such that $\BX_0',\bY_0'\in \dot{B}^{\frac12, \subw}_{2,1}$.  To see this, note that by Lemma \ref{lem:VallePoisson} there exist functions $\subw_X, \subw_Y$ such $\bX_0'\in  \dot{B}^{\frac12, \subw_X}_{2,1}$ and $\bY_0'\in  \dot{B}^{\frac12, \subw_Y}_{2,1}$.  Then taking $\subw(r) = \min\{\subw_X(r), \subw_Y(r)\}$ is sufficient.  
\end{remark}

\begin{theorem}\label{main:unique}
(Strong continuity)
We consider the two strong solutions $\bX, \bY: [0,T]\times \T\to \R^2$ to the Peskin problem \eqref{peskin.general.tension} with initial data $\bX_0$, $\bY_0$ as in Theorem \ref{thm:mainquant}.  Suppose the tension map $\bT$ as in \eqref{tension.map.def} satisfies   $D\bT\in W^{2,\infty}(\R^2; \R^{2\times 2})$ and the ellipticity condition $D\bT(z)\geq \lambda \IO>0$.  

Then there exists a time $T_M>0$ depending only on $M$, $|\bX_0|_*$, $|\bY_0|_*$, $\subw$, $\lambda$, and $||D\bT||_{W^{2,\infty}}$ such that for any $0<T\leq T_M,$ 
we have the following strong continuity estimate
\begin{equation*}
     ||\bX' -  \bY'||_{\CN}
    +
  2\lambda^{\frac12} || \bX'-  \bY'||_{\DN}
    \leq   
    8 || \bX'_0 -  \bY'_0||_{\BN}.
\end{equation*}
Above $\nu$, which is defined precisely in \eqref{nu.definition} also satisfies Definition \ref{subw.definition} and defines norms of $\CN$ and $\DN$ that are equivalent to $\CM$ and $\DM$ respectively as seen in \eqref{equivalent.nu.norm}.  
\end{theorem}

\begin{corollary}\label{cor:lipshitz}(Locally Lipschitz)
Suppose the tension map $\bT$ as in \eqref{tension.map.def} satisfies   $D\bT\in W^{2,\infty}(\R^2; \R^{2\times 2})$ and the ellipticity condition $D\bT(z)\geq \lambda \IO>0$, and let $\subw$ satisfy the assumptions of Definition \ref{subw.definition}.  Then for any $M, \rho\in (0,\infty)$, there exists a time $T>0$ such that for all $0<t\leq T$, the map 
\begin{equation}\notag
    \BX_0\longrightarrow \BX(t),
\end{equation}
is Lipschitz continuous from the bounded set $\{\bZ: ||\bZ'||_{\dot{B}^{1/2, \subw}_{2,1}}\leq M, |\bZ|_*\geq \rho\}$ to $\dot{B}^{1/2,\subw}_{2,1}$, with Lipschitz constant depending on $M, \rho, \mu, \lambda, $ and $||D\bT||_{W^{2,\infty}}$.  
\end{corollary}

Corollary \ref{cor:lipshitz} follows directly from Theorem \ref{main:unique}.

\subsection{Discussion of the assumptions on the tension}\label{sec:TensionAssumptions}

In this subsection we will discuss our assumptions on the scalar tension $\TE(r)$ and on the tension map $\bT(z) = \TE(|z|)\hat{z}$ in  \eqref{tension.map.def}.  We separate our assumptions on the tension into two groups: the assumptions needed for the qualitative  Theorem \ref{thm:main} versus the assumptions used to prove the quantitative bounds in Theorems \ref{thm:mainquant}, \ref{first:unique}, and \ref{main:unique}. 

Our qualitative assumptions in Theorem \ref{thm:main} are very weak, only requiring 
\begin{equation}\label{e:QualitativeScalarTension}
    \left\{\begin{array}{l} 
    \TE \in C^{1,1}_{loc}((0,\infty); (0,\infty)), 
    \\ \TE'(r)>0, \  0<r<\infty.
    \end{array}\right. 
\end{equation}
By $\TE \in C^{k,\gamma}_{loc}$ or $C^{k,\gamma}_{loc}(0,\infty)$ for an integer $k\geq 0$ and $0\leq \gamma \leq 1$, we mean for any $0<a<b<\infty$ that $\TE\in C^{k,\gamma}([a,b]; (0,\infty))$.  For qualitative higher regularity, we also assume $\TE\in C^{k, \gamma}_{loc}(0,\infty).$ Thus singularities or degeneracy at $r=0$ or as $r\to \infty$ are allowable, and in particular any positive power law $\TE(r) = C r^p$ for $p>0$ and $C>0$ satisfies \eqref{e:QualitativeScalarTension}.  Note that there is no requirement that $\lim\limits_{r\to \infty}\TE(r) = \infty$, so a bounded function such as $\TE(r)=\arctan(r)$ would also satisfy \eqref{e:QualitativeScalarTension}.  

For our quantitative estimates, we work with tensions that have the following global bounds
\begin{equation}\label{e:QuantitativeScalarTension}
    \left\{\begin{array}{l}
     \TE'\in W^{1,\infty}([0,\infty); [0,\infty)), 
    \\ \inf\limits_{0<r<\infty}\TE'(r)\geq \lambda>0,  
    \\ \TE(0) = 0, 
    \end{array}\right. 
\end{equation}
For quantitative higher regularity and the strong continuity estimate, we also need to assume $\TE\in C^{k, \gamma}_r[0,\infty)$ with $\TE^{'(k)}(0)=0$ for $k\geq 2$.  This would be implied for example if $\TE(r) = cr$ on $0\leq r\leq \epsilon$ for some $c>0$ and any small $\epsilon>0$.  Note that the estimates we prove will depend on bounds for the tension map $\bT(z)$, rather than the scalar tension itself.  The assumption that $\TE$ has higher order derivatives vanish at 0 guarantees that $\TE\in C^{k,\gamma}([0,\infty);[0,\infty))$ implies $\bT\in C^{k,\gamma}(\R^2; \R^2),$ with $||\bT||_{C^{k,\gamma}}$ controlled in terms of $||\TE||_{C^{k,\gamma}}$.  
Also note that the global lower bound $\inf \TE'(r)\geq \lambda>0$ and $\TE(0)=0$ give us a lower bound on the derivative of the tension map $\bT$ as
\begin{equation}\label{e:DTdefn}
    D\bT(z) = \TE'(|z|) \hat{z}\otimes \hat{z} + \frac{\TE(|z|)}{|z|} \hat{z}^\perp\otimes \hat{z}^\perp \geq   \lambda \IO,
    \quad \lambda>0.
\end{equation}
Of course in the case of simple tension \eqref{linearF} where $\TE(r) = k_0r$, it follows that $D\bT(z) = k_0 \IO$.  
For our quantitative estimates, we will typically state our assumptions for the tension map $\bT(z)$ in \eqref{tension.map.def} by assuming $\forall z \in \R^2$ that \eqref{e:DTdefn} holds and further that
\begin{equation}\label{e:QuantitativeTensionMap}
\left\{\begin{array}{l} |D\bT(z)|\leq \DTT, 
\\ |D^2 \bT(z)| \leq \DDTT. 
    \end{array}\right.
\end{equation}
Here $\DTT$ and $\DDTT$ are any fixed positive finite constants that are allowed to be large.  For our strong continuity estimate in Theorem \ref{main:unique}, for a fixed positive finite constant $\DDDTT$, we additionally assume $\forall z \in \R^2$ that
\begin{equation}\label{tension.derivatives.continuity}
    |D^3\bT(z)|\leq \DDDTT.
\end{equation}
Our quantitative estimates on higher regularity additionally depend on $||\bT||_{C^{k,\gamma}}$.  

Lastly, we note the apparent mismatch between our qualitative \eqref{e:QualitativeScalarTension} and quantitative \eqref{e:QuantitativeScalarTension} assumptions.  That is, not every scalar tension $\TE$ satisfying the qualitative assumptions will also satisfy the quantitative version.  In particular, all positive power laws satisfy the former, but only the linear case satisfies the latter.

We are able to deal with these different assumptions for the following reason.  Suppose that we have a tension $\TE_1$ satisfying the quantitative assumptions \eqref{e:QuantitativeScalarTension}, and we use our quantitative estimates to construct a solution $\bX:[0,T]\times \T\to \R^2$ to the Peskin problem with tension $\TE_1$.  Then for any time $t$ and any $\theta\in \T$ we have $0<|\bX(t)|_*\leq |\bX'(t,\theta)|\leq ||\bX'(t)||_{L^\infty_\theta}$.  Taking $a = \inf_{0\leq t\leq T} |\bX(t)|_*$ and $b = ||\bX'||_{L^\infty_T (L^\infty_\theta)}$, we then have that $\bX(t)$ is also a solution to the Peskin problem \eqref{peskin.general.tension} for any tension $\TE_2$ such that $\TE_2\big|_{[a,b]} = \TE_1\big|_{[a,b]} $.  

Now suppose that our tension $\TE$ only satisfies the  qualitative assumptions \eqref{e:QualitativeScalarTension}.  These are still enough to guarantee that for any $0<a<b<\infty$, there exists a tension $\tilde{\TE}$ such that $\tilde{\TE}\big|_{[a,b]} = \TE\big|_{[a,b]},$ and $\tilde{\TE} $ satisfies the quantitative assumptions \eqref{e:QuantitativeScalarTension}.  Thus, for any fixed initial data $\bX_0$ with $\BX'_0\in \dot{B}^{\frac12}_{2,1}\subseteq L^\infty$ and $|\bX_0|_*>0$, we take some interval $(a,b)$ which compactly contains $\{|\bX'_0(\theta)|: \theta\in \T\}$, and then we construct a solution $\bX:[0,T]\times \T\to \R^2$ to the Peskin problem with tension $\tilde{\TE}$.  Taking $T>0$ small enough that $\{|\bX'_0(t,\theta)|: (t,\theta)\in [0,T]\times \T\}\subset (a,b)$, we then have that $\bX(t)$ is also a solution to the Peskin problem with our original tension $\TE$.  We go through this argument again in more detail in the proof of our main theorem in \secref{sec:mainThmProof}.

\begin{remark}\label{remark:trick}
We note that the trick explained above and in the proof of our main theorem in \secref{sec:mainThmProof}  always works for the kinds of solutions we consider with Definitions \ref{def:solution} and \ref{def:StrongSolution}.    For the assumptions needed in order to apply this trick (to replace one tension with another) to fail, the solution would have to satisfy one of two conditions.  Either (1) the solution $\bX(t)$ violates the arc-chord condition \eqref{arc.cord.number} after an infinitesimal amount of time $\liminf\limits_{t\to 0+} |\bX(t)|_* = 0$, or (2) the $L^\infty$ norm of the solution $\BX'(t)$ blows up after an infinitesimal amount of time: $\limsup\limits_{t\to 0+} ||\bX'(t)||_{L^\infty} = \infty.$  It's not clear whether a notion of solution which obey's either of these two conditions starting from initial data with $\BX_0'\in \dot{B}^{1/2}_{2,1}\subseteq L^\infty_\theta$ and $|\bX_0|>0$ would represent a physical solution.
\end{remark}

\begin{remark}
At the same time we remark that Theorem's \ref{thm:mainquant}, \ref{first:unique} and \ref{main:unique}  also hold if 
instead we replaced  \eqref{e:QuantitativeTensionMap} and \eqref{tension.derivatives.continuity} with
\begin{equation}\notag
    \bigg| D^{(k)}\bT(z)\big|_{z=\bX'}\bigg|\leq \DKTT(|| \BX' ||_{L^\infty_\theta}, |\bX|_*^{-1}),
\end{equation}
where for $k\in\{1,2,3\}$ we have $\DKTT=\DKTT(|| \BX' ||_{L^\infty_\theta}, |\bX|_*^{-1})$ are any increasing functions of both variables.  Then under these conditions the proofs of those theorems in this paper continue to hold without any essential modifications.   And further the solutions constructed under the assumptions in this remark would prevent the occurrence of (1) or (2) in the previous Remark \ref{remark:trick}. 
\end{remark}

\subsection{A de la Valle-Poisson lemma}
Motivated by the work in \cite{AlaNgu2021lipsh,2009.08442,2010.06915}, we now prove the following de la Valle-Poisson type lemma.

\begin{lemma}\label{lem:VallePoisson} Fix any $p\in [1,\infty]$, $r\in [1,\infty)$, and $s\in (0,1)$.  Given any function $f$ satisfying  $|| f||_{\dot{B}^s_{p,r}(\T)} <\infty$, then there exists a function $\subw$, depending upon $f$, satisfying the assumptions of Definition \ref{subw.definition} such that $|| f||_{\dot{B}^{s,\subw}_{p,r}(\T)} <\infty$.
\end{lemma}

The proof builds upon the related lemma from \cite[Lemma 3.8 on page 35]{AlaNgu2021lipsh}.

\begin{proof} Since $|| f||_{\dot{B}^s_{p,r}(\T)} <\infty$ then after a simple change of variables we have that 
\begin{equation}\notag
|| f||_{\dot{B}^s_{p,r}(\T)}^r
=
    \int_0^\pi \frac{d\beta}{|\beta|^{1+sr}} \left( ||\delta_\beta f||_{L^p}^r +||\delta_{-\beta} f||_{L^p}^r\right) < \infty.
\end{equation}
We now define 
 \begin{equation*}
    h_{p,r}(\beta) \eqdef ||\delta_\beta f||_{L^p}^r+||\delta_{-\beta} f||_{L^p}^r, \quad  \omega(\alpha) \eqdef  \pi^{-sr}\alpha^{sr-1} h_{p,r}(\pi \alpha^{-1}).
 \end{equation*}
 Then we will use the change of variables $\alpha = \pi \beta^{-1}$ to obtain 
\begin{equation}\notag
    \int_1^\infty d\alpha ~\omega(\alpha)
    =
\int_0^\pi \frac{d\beta}{|\beta|^{1+sr}} \left( ||\delta_\beta f||_{L^p}^r +||\delta_{-\beta} f||_{L^p}^r\right) < \infty.
\end{equation}
By \cite[Lemma 3.8 on page 35]{AlaNgu2021lipsh} there then exists some function $\nu:[1,\infty)\to [1,\infty) \ $ satisfying the conditions of Definition \ref{subw.definition} such that 
\begin{equation}\notag
   \int_1^\infty d\alpha ~\omega(\alpha)~\nu(\alpha)
   =
\int_0^\pi \frac{d\beta}{|\beta|^{1+sr}} \nu(\pi |\beta|^{-1}) \left( ||\delta_\beta f||_{L^p}^r +||\delta_{-\beta} f||_{L^p}^r\right) < \infty.
\end{equation}
Taking $\mu(|\beta|^{-1}) = \nu(\pi |\beta|^{-1})^{1/r}$, we have that $\mu$ satisfies the conditions of Definition \ref{subw.definition} as well, and further $|| f||_{\dot{B}^{s,\subw}_{p,r}(\T)} <\infty$.\end{proof}

We point out that using this Lemma \ref{lem:VallePoisson} then Theorem \ref{thm:main} follows immediately from Theorem's \ref{thm:mainquant} and \ref{first:unique}.

\subsection{The $\Lambda$ operator}\label{sec:para}
For a function $f:\R\to \C$ the $\Lambda^s=(-\Delta)^{\frac{s}{2}}$ operator is widely defined for any $s\in (0,2)$ as
\begin{equation*}
        -\Lambda^s f(x) \eqdef C_{s} \pv\int_{\R} \frac{(\delta_y+\delta_{-y})f(x)}{|y|^{1+s}} dy, \quad C_s \eqdef \frac{2^s \bm{\Gamma}(\frac12 (1+s))}{2\pi^{\frac12}|\bm{\Gamma}(-\frac{s}{2})|}. 
\end{equation*}
Here we use the principal value integral when it is needed, and $\bm{\Gamma}$ is the standard Gamma function.  
Then for $f:\T \to \C$, identified as a periodic function on $\R$, this can readily be reduced to 
\begin{equation}\label{lambda.s.def}
        -\Lambda^s f(\theta) = C_{s}\int_{\T} (\delta_\alpha+\delta_{-\alpha})f(\theta)  \sum_{k \in \Z}\frac{1}{|\alpha + 2\pi k|^{1+s}}   d\alpha.    
\end{equation}
The above can be taken as the definition of $\Lambda^s$ on $\T$.  Now for simplicity we  define the notation $\SL(\alpha)$ as 
\begin{equation}\label{distance.alpha}
    \SL(\alpha) \eqdef 2 \sin(\alpha/2).
\end{equation}
Then we have the following known expansion formula
\begin{equation}\notag
    \frac{1}{\SL(\alpha)^2} 
    =
    \sum_{n=-\infty}^{+\infty} \frac{1}{(\alpha+2\pi  n)^2}, \quad 0 < |\alpha| \leq \pi .
\end{equation}
Thus for $s=1$ the $\Lambda$ operator on $\T$ has the following succinct formula
\begin{equation}\label{lambda.definition}
   - \Lambda f(\theta) = \frac{1}{\pi} \int_{\T}   \frac{ f(\alpha) -f(\theta) }{\SL(\theta - \alpha)^2}d\alpha
    = \frac{1}{\pi} \int_{\T}   \frac{\delta_\alpha f(\theta)}{\SL(\alpha)^2} d\alpha,
\end{equation}
Notice further that we have 
\begin{equation}\label{sine.bound}
\frac{2}{\pi} \leq \frac{ \SL(\alpha)}{\alpha} \leq 1, \quad \forall \alpha \in {\T}.
\end{equation}
In particular in the $L^p(\T)$ sense then \eqref{lambda.definition} is equivalent to the operator containing $\alpha^2$ in the denominator instead of $\SL(\alpha)^2$.  This discussion motivates the following simplifed notation that we will use in the rest of this article
\begin{equation}\label{tildeLambda:eq}
   - \TLam f(\theta) \eqdef \frac{1}{4\pi} \int_{\T} \frac{\delta_\alpha f(\theta)}{\alpha^{2}} d\alpha. 
\end{equation}
By \eqref{sine.bound} the operator $\TLam$ is equivalent to $\Lambda$ from \eqref{lambda.definition} in the $L^p(\T)$ norms.  

More generally, for $s\in (0,2)$ we write the previous sum as 
\begin{equation*}
            \sum_{k \in \Z}\frac{1}{|\alpha + 2\pi k|^{1+s}}
    =
    \frac{1}{|\alpha|^{1+s}}\left(1
    +
    \sum_{k \neq 0}\frac{|\alpha|^{1+s}}{|\alpha + 2\pi k|^{1+s}}\right)
    =\frac{1}{|\alpha|^{1+s}} \left(1+ \mathfrak{U}(\alpha)\right).
\end{equation*}
Notice that for $\alpha \in \T$ the series $\mathfrak{U}(\alpha)$ converges uniformly.  Also $\mathfrak{U}(\alpha)$ is non-negative and is uniformly bounded for $\alpha \in \T$.  We conclude that 
\begin{equation*}
            \sum_{k \in \Z}\frac{1}{|\alpha + 2\pi k|^{1+s}}
            \approx
            \frac{1}{|\alpha|^{1+s}},
            \quad \forall \alpha \in \T.
\end{equation*}
Thus again in the $L^p(\T)$ sense $\Lambda^s$ is equivalent to the operator containing $\frac{1}{|\alpha|^{1+s}}$ instead of $\sum_{k \in \Z} \frac{C_{s}}{|\alpha + 2\pi k|^{1+s}}$ in \eqref{lambda.s.def} for any $s\in(0,2)$.

\subsection{Overview of the proof}\label{subsec:proofOverview} One very important point in the proof is the derivation of the equation \eqref{peskin.general.tension} with the kernel \eqref{kerbel.eqn.deriv}.  It is crucial that the equation \eqref{peskin.general.tension} cancels the second order derivatives that are present in \eqref{e:boundaryintegral}.     Let $\nabla_{\BX'}$ denote the directional derivative in the direction $\BX'$ as in \eqref{directional.def}, then with \eqref{stokeslet.def} the main idea can be seen as in
\begin{equation}\notag
[\nabla_{\BX'(\theta+\alpha)}G_1](\delta_\alpha \BX)\delta_\alpha \BX + G_2(\delta_\alpha \BX) \BX'(\theta+\alpha)=0.
\end{equation}
Fortunately this type of cancellation is preserved when we take higher order derivatives of the equation \eqref{e:boundaryintegral}.  This more general cancellation structure is observed via a sequence of integrations by parts performed in \secref{sec:GenTenDerivation}.

Then the heart of our argument is the initial a priori estimate \eqref{e:primaryestimate}.  In order to prove this, we make use of our new formulation of the equation for $\partial_t \bX'$ in \eqref{peskin.general.tension}.  Because $\mathcal{K}(\theta, \alpha)$ is symmetric in $\theta, \theta+\alpha$, our equation \eqref{peskin.general.tension} has divergence form symmetry making $L^2$ based energy estimates a useful choice.  By making use of Besov spaces, we're interested then in keeping careful track of the time evolution of differences $||\delta_\beta \bX'||_{L_\theta^2}(t)$ where $\beta\in \T$ is arbitrary.  Taking into account \eqref{kerbel.eqn.deriv} and \eqref{peskin.expand.tension} with \eqref{tildeLambda:eq}, we have that $\delta_\beta \bX'$  solves the equation
\begin{equation}\notag
    \partial_t \delta_\beta \bX' + \TLam \delta_\beta \bT(\bX') = \int_\T \frac{d\alpha}{\alpha^2} \delta_\beta \left(\mathcal{A}[\BX](\theta, \alpha)\delta_\alpha \bT(\bX')\right).
\end{equation}
When we calculate $\frac{d}{dt} ||\delta_\beta \bX'||_{L^2_\theta}^2$, we then get one good diffusive term $-\lambda ||\delta_\beta \bX'||_{\dot{H}^{1/2}}^2$ from the $\TLam \delta_\beta \bT(\bX')$ (along with additional error terms if our tension isn't simple).  We treat the remaining terms as error, and then we are left to bound integrals (for $q = 1, 2$) of the form   
\begin{equation}\notag
    \int_\T d\theta\int_\T  \frac{d\alpha}{\alpha^2}~ |\delta_\beta \delta_\alpha \bX'(\theta)|^2 ~ |\delta_\alpha \bX'(\theta)|^q,
\end{equation}
and
\begin{equation}\notag
  \int_\T d\theta\int_\T  \frac{d\alpha}{\alpha^2}~  |\delta_\beta \delta_\alpha \bX'(\theta)| \ |\delta_\beta \bX'(\theta)| \ |\delta_\alpha \bX'(\theta)|^{q+1}.
\end{equation}
If we were to bound the first term naively, we would get
\begin{equation}\notag
   C ||\delta_\beta \bX'||_{\dot{H}^{1/2}}^2 ||\bX'||_{L^\infty_\theta}^q,
\end{equation}
which would make it impossible to close the estimate, as this is of the same order as our good diffusive term but with a possibly large coefficient in front for large data.  However, the norm for $\dot{B}^{\frac12, \mu}_{2,1}$ both controls the size of the norm $\dot{B}^{\frac12}_{2,1}$ and the rate of decay for 
\begin{equation}\label{rate.decay}
    r \to \int_{|\alpha|<r}d\alpha  \frac{||\delta_\alpha f||_{L^2_\theta}}{|\alpha|^{3/2}}  \lesssim \frac{||f||_{\dot{B}^{1/2, \mu}_{2,1}}}{\mu(r^{-1})}.
\end{equation}
Thus splitting the integral in our error term between $|\alpha|<\eta$ and $|\alpha|>\eta$ for some $\eta$ sufficiently small depending on $\mu$,  $||\bX'_0||_{\dot{B}^{1/2, \mu}_{2,1}}$, and other relevant constants, we are able to bound this error term for any small $\epsilon>0$ as 
\begin{equation}\notag
    \epsilon ||\delta_\beta \bX'||_{\dot{H}^{1/2}}^2 + C_{\epsilon} ||\delta_\beta \bX'||_{L^2_\theta}^2, 
\end{equation}
which we can handle.  For the second type of error term, the story is similar except that we are forced to bound the $|\delta_\beta \bX'|$ in $L^\infty$, as it has no decay as $\alpha\to 0$.  Thus we end up with an error term of the form
\begin{equation}\notag
    \epsilon||\delta_\beta \bX'||_{\dot{H}^{1/2}}^2 + C ||\delta_\beta \bX'||_{L^2_\theta}^2 + \epsilon ||\delta_\beta \bX'||_{L^\infty_\theta}^2.
\end{equation}
This $L^\infty$ error term at first seems very bad, as notably the Sobolev embedding fails in $L^\infty$ and $||\delta_\beta\bX'||_{L^\infty_\theta}^2$ is not controlled by our good diffusive piece  $-\lambda ||\delta_\beta \bX'||_{\dot{H}^{1/2}}^2$.  However, once we integrate in $\beta$ against $\mu(|\beta|^{-1}) |\beta|^{-3/2}$ the Sobolev embedding is again true, and we can control this error term at the end of the estimate.

It is also vitally important to get a positive bound from below on the arc-chord condition $|\bX(t)|_*$.  In order to do this, we make use of the estimate
\begin{equation}\notag
||\bX(t)|_* - |\bX_0|_*| \leq ||\bX'(t)-\bX'_0||_{L^\infty}.
\end{equation}
Thus in \secref{sec:ArcChord} we prove continuity of the map $t\to \bX'(t)$ in $L^\infty_\theta$ for small times.  Our main a priori estimate \eqref{e:primaryestimate} grants us uniform bounds on the $\CM$ and $\DM$ norms.  Using our $\DM$ bound, we then control $\partial_t \bX'$ in $L_{t,\theta}^2$ and use this to prove continuity of $\bX'(t)$ in $L_\theta^2$.  Continuity in time in $L_\theta^2$ and our bound in $\CM$ then gives us continuity in time in $\dot{B}^{1/2}_{2,1}$, which controls $L^\infty_\theta$.

The strong continuity estimate given in Theorem \ref{main:unique} is for the most part similar to our main a priori estimate \eqref{e:primaryestimate}.  However to obtain this estimate requires subtracting two solutions to the equation \eqref{peskin.general.tension} which in turn requires using the higher order bound \eqref{tension.derivatives.continuity}.  Additionally when taking the difference of two solutions $\bX'$ and $\bY'$ to  \eqref{peskin.general.tension} we encounter a new term of the form 
\begin{equation}\label{e:StrongContNewBad}
\DDTT  \int_{\T} d\theta \int_{\T} \frac{d\alpha}{\alpha^2} ~
   | \delta_\beta\delta_\alpha  (\bX'-  \bY')(\theta)|  |\tau_\beta \mathcal{K}[\bX](\theta,\alpha)|
   | (\bX'-  \bY')(\theta)| | \delta_\beta \delta_\alpha  \bY'(\theta)|.
\end{equation}
The structure of this term does not have the ability to obtain extra smallness using the rate of decay in \eqref{rate.decay} in the energy estimate.  This major difficulty prevents closing the strong continuity estimate in the norm of $\CM$ in \eqref{C.space.temporal}.    Instead we simply bound this term by
\begin{equation}\notag
\LDS ||   \delta_\beta \TLam^{\frac12}(\bX'-  \bY')||_{L^2_\theta}^2
+
C \lambda^{-1}  \DDTT^2   || \bX'-  \bY'||_{L^\infty_\theta}^2 
||   \delta_\beta \TLam^{\frac12}\bY'||_{L^2_\theta}^2.
\end{equation}
The term $||   \delta_\beta \TLam^{\frac12}(\bX'-  \bY')||_{L^2_\theta}^2$ can be controlled by the dissipation.  But we also require a small constant in front of $|| \bX'-  \bY'||_{L^\infty_\theta}^2 $ to close the continuity estimate.  For this reason instead of using the norm $\CM$ with $\subw$ satisfying Definition \ref{subw.definition}, we need to introduce an equivalent norm as in \eqref{nu.definition} for a small constant $\varepsilon=\varepsilon(\lambda, \DDTT,||\bX'_0||_{\dot{B}^{\frac12,\subw}_{2,1}},||\bY'_0||_{\dot{B}^{\frac12,\subw}_{2,1}})>0$ as
\begin{equation}\notag
    \nu(r) \eqdef 1 + \varepsilon \subw(r).
\end{equation}
And then $\nu$ also satisfies Definition \ref{subw.definition}.  
Then the norm of $\CN$ is equivalent to the norm of $\CM$ and we are able to close the continuity estimate in $\CN$.

We also prove continuity for $\bX'(t)-\bY'(t)$ in $L^2_\theta.$  This estimate is much simpler than the strong continuity estimate, and only requires $D\bT \in W^{1,\infty}$ rather than $D\bT \in W^{2,\infty}$.  In particular, by making use of our a priori estimate in the higher order $\CM$ and $\DM$ norms, we are able to bound a term like \eqref{e:StrongContNewBad} directly without changing to some equivalent norm.  Continuity in $L^2_\theta$ and a bound in $\CM$ then implies that we have control over $\bX'-\bY'$ in the $\mathcal{B}^\omega_T$ norm, for any function $\omega$ satisfying that $\frac{\omega(r)}{\mu(r)}$ is eventually decreasing with $\displaystyle\lim\limits_{r\to \infty} \frac{\omega(r)}{\mu(r)} = 0.$

Our higher regularity proofs are contained in \secref{sec:smoothing}.  We begin by proving an $L^\infty_t \dot{H}^1$ estimate for $\bX'$ and then establish regularity of the remainder from \eqref{peskin.expand.tension}:
\begin{equation}\notag
\mathcal{V}(\theta) \eqdef \int_{\T} \frac{d\alpha}{\alpha^2}~ \mathcal{A}[\BX](\theta, \alpha)\delta_\alpha \bT(\BX'(\theta)),
\end{equation}
in terms of the regularity of $\bX'$.  Following the proof in \cite{MR3656476} for the scalar fractional porous medium equation, we then establish higher regularity for the fully nonlinear Peskin problem with a bootstrapping argument.

\subsection{Outline} In the next \secref{sec:GenTenDerivation} we will derive the equation \eqref{peskin.general.tension} that we will study in the rest of this work.   Then in \secref{sec:BesovSpace} we will prove our main a priori estimate.  After that in \secref{sec:ArcChord} we will explain a priori how we control the arc-chord condition \eqref{arc.cord.number} along the time evolution of \eqref{peskin.general.tension}.  And then in \secref{sec:DifferenceBesovSpace} we prove the a priori continuity estimates for solutions to \eqref{peskin.general.tension} that enable us to establish the strong continuity and uniqueness. 
Next in \secref{sec:smoothing} we prove the higher order smoothing effects.
Finally in \secref{sec:mainThmProof} we collect the previous results to explain the proof of our main theorems.  Afterwards in \secref{sec:LPtorus} we explain some of the inequalities that we use in the previous sections of this text using the Littlewood-Paley decomposition on the torus. Lastly in \secref{sec:kernelDIFF} we give the difference estimates for the kernel \eqref{kerbel.eqn.deriv} and \eqref{kerbel.A.eqn.deriv} of the equation \eqref{peskin.general.tension}.

\section{Derivation of the general tension equation}\label{sec:GenTenDerivation}

In this section we will derive our alternative formulation of the equation for $\BX'(\theta)$ as in \eqref{peskin.general.tension} with \eqref{tension.map.def} and \eqref{kernel.peskin.noExpand}.  It is important for our main theorems in this paper that the equation \eqref{peskin.general.tension} does not contain any terms with $\BX''(\theta)$ or higher derivatives.  This is not obvious because the equation \eqref{e:boundaryintegral} does in fact contain terms with $\BX''(\theta)$.  Then in this section we explain the cancellation necessary to show that the higher derivative terms do not occur.  We will first derive an alternative form of the equation for $\partial_t \BX(\theta)$ in \eqref{peskin.equation.first.order.final}.  Then afterwards we will derive in \eqref{general.tension} the equation for $\partial_t \BX'(\theta)$ that we have written previously in \eqref{peskin.general.tension}.

To this end, with a general tension $\TE$ as in \eqref{tension.map.def}, the Peskin problem \eqref{e:boundaryintegral} takes the form of an equation for the parametrization 
\begin{equation}\notag
    \partial_t \BX(\theta) = \int G(\BX(\eta)-\BX(\theta)) \partial_\eta \left(  \bT(\BX')(\eta)\right)d\eta,
\end{equation}
where $G(z)= G_1(z) + G_2(z)$ is the matrix valued function from \eqref{stokeslet.def} and $\bT(z)$ is the tension map from \eqref{tension.map.def}.   In this section we will write the integral, $\int$, without a domain such as $\T$ to emphasize that our calculations in this section are independent of the parametrization.

Next making the change of variables $\eta = \theta+\alpha$ and using \eqref{delta.notation}, we write 
\begin{multline*}
            \partial_t \BX(\theta) = \int G_1(\delta_\alpha \BX) \partial_\alpha \left( \TE(|\BX'|(\theta+\alpha))\widehat{\BX'}(\theta+\alpha)\right)d\alpha
            \\
            + \int G_2(\delta_\alpha \BX) \partial_\alpha \left( \TE(|\BX'|(\theta+\alpha))\widehat{\BX'}(\theta+\alpha)\right)d\alpha.
\end{multline*}
First we will focus on the term involving $G_1(z)$.

We use integration by parts and $\BX'(\theta+\alpha) = \partial_\alpha (\delta_\alpha \BX'(\theta))$ to obtain
\begin{multline*}
        \int G_1(\delta_\alpha \BX) \partial_\alpha \left( \TE(|\BX'|(\theta+\alpha))\widehat{\BX'}(\theta+\alpha)\right)d\alpha 
        \\
        = -\int \partial_\alpha [G_1(\delta_\alpha \BX)] \TE(|\BX'|(\theta+\alpha)) \widehat{\BX'}(\theta+\alpha) d\alpha 
    \\ = -\int \partial_\alpha [G_1(\delta_\alpha \BX)] \frac{\TE(|\BX'|(\theta+\alpha))}{|\BX'|(\theta+\alpha)} \partial_\alpha (\delta_\alpha \BX(\theta)) d\alpha 
    \\ = \int \partial_\alpha \left( \frac{\TE(|\BX'|(\theta+\alpha))}{|\BX'|(\theta+\alpha)} \partial_\alpha [G_1(\delta_\alpha \BX)]\right)\delta_\alpha \BX(\theta) d\alpha. 
\end{multline*}
We plug this calculation back into the equation to obtain
\begin{multline}\label{equation.intermediate}
    \partial_t \BX(\theta) = \int G(\delta_\alpha \BX) \partial_\alpha\left( \frac{\TE(|\BX'|)}{|\BX'|} \BX'\right)(\theta+\alpha) d\alpha 
    \\= \int \partial_\alpha \left( \frac{\TE(|\BX'|)(\theta+\alpha)}{|\BX'|(\theta+\alpha)} \partial_\alpha [G_1(\delta_\alpha \BX)]\right)\delta_\alpha \BX(\theta)  d\alpha
    \\
    + \int G_2(\delta_\alpha \BX) \partial_\alpha \left( \frac{\TE(|\BX'|)(\theta+\alpha)}{|\BX'|(\theta+\alpha)} \BX'(\theta+\alpha)\right) d\alpha.
\end{multline}
Next let $\nabla_u$ denote the directional derivative in the direction $u\in \R^2$, i.e.
\begin{equation}\label{directional.def}
    \nabla_uf(z) \eqdef \lim\limits_{\epsilon\to 0+} \frac{f(z+\epsilon u)-f(z)}{\epsilon}.
\end{equation}
For the matrix valued functions $G_1(z)$ and $G_2(z)$ from \eqref{stokeslet.def} direct calculation gives  
\begin{equation}\notag
4\pi\left\{\begin{array}{ll}
    \nabla_u G_1(z) & = - \frac{u\cdot \hat{z}}{|z|} \IO,
    \\  \nabla_u G_2(z) & = \frac{u\cdot \hat{z}^\perp}{|z|} \RO(z), 
    \\ \nabla_u \nabla_v G_1(z) & = \frac{u \cdot \DO(z) v}{|z|^2} \IO, 
    \\  \nabla_u \nabla_v G_2(z) &= -\frac{u \cdot \RO(z)v}{|z|^2} \RO(z) + \frac{u \cdot (\DO(z)-\IO)v}{|z|^2} \DO(z), \end{array}\right. 
\end{equation}
where $\RO(z)$ and  $\DO(z)$ are the reflection matrices from  \eqref{matrix.operators}. Thus
\begin{equation}\notag
    \partial_\alpha \left[G_1(\delta_\alpha \BX)\right] = \left[\nabla_{\BX'(\theta+\alpha)} G_1\right](\delta_\alpha 
    \BX), 
\end{equation}
and 
\begin{equation}\notag
    \partial_\alpha^2 [G_1(\delta_\alpha \BX)] = \left[\nabla^2_{\BX'(\theta+\alpha)} G_1\right](\delta_\alpha \BX) + \left[\nabla_{\BX''(\theta+\alpha)} G_1\right] (\delta_\alpha \BX).
\end{equation}
We now {\it claim} that 
\begin{equation}\label{peskin.equation.first.order}
    \partial_t \BX(\theta) = \int \left[\nabla^2_{\BX'(\theta+\alpha)} G_1\right](\delta_\alpha \BX) \frac{\TE(|\BX'|(\theta+\alpha))}{|\BX'|(\theta+\alpha)} \delta_\alpha \BX(\theta)d\alpha.
\end{equation}
Then \eqref{peskin.equation.first.order} directly implies that 
\begin{equation}\label{peskin.equation.first.order.final}
    \partial_t \BX(\theta) = \frac{1}{4\pi}\int \frac{\BX'(\theta+\alpha) \cdot \DO(\delta_\alpha \BX)\BX'(\theta+\alpha)}{|\delta_\alpha \BX|^2}  \frac{\TE(|\BX'|(\theta+\alpha))}{|\BX'|(\theta+\alpha)} \delta_\alpha \BX(\theta) d\alpha.
\end{equation}
Then \eqref{peskin.equation.first.order.final} will be our main expression for the Peskin equation for $\BX(\theta)$.

Now \eqref{equation.intermediate} and the previous calculations imply that
\begin{equation}\notag
\begin{split}
    \partial_t \BX(\theta) &=  \int \left[\nabla_{\BX'(\theta+\alpha)}^2 G_1\right](\delta_\alpha \BX) \frac{\TE(|\BX'|(\theta+\alpha))}{|\BX'|(\theta+\alpha)} \delta_\alpha \BX(\theta)d\alpha
    \\& \quad + \int \frac{\TE(|\BX'|)}{|\BX'|} \left([\nabla_{\BX''(\theta+\alpha)}G_1](\delta_\alpha \BX)\delta_\alpha \BX + G_2(\delta_\alpha \BX) \BX''(\theta+\alpha)\right)d\alpha 
    \\& \quad + \int \partial_\alpha \left(\frac{\TE(|\BX'|)}{|\BX'|}\right) \left([\nabla_{\BX'(\theta+\alpha)}G_1](\delta_\alpha \BX)\delta_\alpha \BX + G_2(\delta_\alpha \BX) \BX'(\theta+\alpha)\right)d\alpha.
\end{split}
\end{equation}
Now to prove the {\it claim} \eqref{peskin.equation.first.order}, with \eqref{stokeslet.def} we use 
\begin{equation}\label{directional.cancellation}
    [\nabla_{u} G_1(z)]z = -\frac{1}{4\pi} \frac{u\cdot \widehat{z} }{|z|} z 
    = -\frac{u\cdot \widehat{z}}{4\pi}\widehat{z} 
    = - G_2(z) u.
\end{equation}
This exact calculation \eqref{directional.cancellation} is crucial to cancel the second two terms above, and in particular to cancel the second order derivatives.   Using this cancellation, since the last two terms in the equation above are zero, then we obtain the {\it claim} in \eqref{peskin.equation.first.order}.  And the equation \eqref{peskin.equation.first.order.final} is our alternative representation of the Peskin equation for $\BX(t,\theta)$.

To obtain an equation for $\partial_t \BX'(\theta)$, we could of course just differentiate \eqref{peskin.equation.first.order.final} in $\theta$.  However, that equation contains $\bX''(\theta)$ and ends up being more difficult to work with.  Luckily though, there is another form for $\partial_t \bX'$ which can be written in terms of only $\bX'$.   To begin our derivation for $\partial_t \BX'$, we note that integrating by parts and using \eqref{e:boundaryintegral} with \eqref{tension.map.def} we have
\begin{equation}\notag
\begin{split}
    \partial_t \BX(\theta) &= \int G(\delta_\alpha \BX) \partial_\alpha  \bT(\BX')(\theta+\alpha) d\alpha 
    \\& = - \int \partial_\alpha G(\delta_\alpha \BX) \delta_\alpha\bT(\BX')(\theta)d\alpha. 
\end{split}
\end{equation}
Differentiating this equation with respect to $\theta,$ we see that
\begin{equation*}
            \partial_t \BX'(\theta)  = -\int \partial_\theta \partial_\alpha  G(\delta_\alpha \BX) \delta_\alpha\bT(\BX')(\theta)
        -\int\partial_\alpha G(\delta_\alpha \BX) \partial_\theta\delta_\alpha \bT(\BX') (\theta)d\alpha.
\end{equation*}
As $\partial_\alpha G(\delta_\alpha \BX)$ is a derivative, it follows that
\begin{equation*}
            \int \partial_\alpha  G(\delta_\alpha \BX) \partial_\theta \bT(\BX')(\theta)d\alpha  
        = \partial_\theta \bT(\BX')(\theta)
        \int \partial_\alpha  G(\delta_\alpha \BX)d\alpha  = 0.
\end{equation*}
Notice that the zero integral above removes a highest order derivative.  

We also have that 
$\partial_\theta \bT(\BX')(\theta+\alpha) = \partial_\alpha \bT(\BX')(\theta+\alpha)$. So we can make this exchange and integrate by parts to obtain
\begin{equation}\notag
\begin{split}
\int \partial_\alpha G(\delta_\alpha \BX) \partial_\theta \delta_\alpha \bT(\BX') (\theta) d\alpha &=   \int \partial_\alpha G(\delta_\alpha \BX) \partial_\alpha \delta_\alpha \bT(\BX') (\theta) d\alpha  
\\&= -\int \partial_\alpha^2 G(\delta_\alpha \BX) \delta_\alpha \bT(\BX')(\theta)d\alpha.
\end{split}
\end{equation}
Hence, we have that
\begin{equation}\notag
    \partial_t \BX'(\theta) = \int (\partial_\alpha^2 - \partial_\alpha \partial_\theta)G(\delta_\alpha \BX)\delta_\alpha \bT(\BX')(\theta)d\alpha.
\end{equation}
Its a straight forward calculation to see that 
\begin{equation}\notag
    (\partial_\alpha -\partial_\theta)[G(\delta_\alpha \BX)] = [(\nabla_{\BX'(\theta+\alpha)} - \nabla_{\delta_\alpha \BX'(\theta)})G](\delta_\alpha \BX) = [\nabla_{\BX'(\theta)} G](\delta_\alpha \BX),
\end{equation}
and 
\begin{equation}\notag
\partial_\alpha [\nabla_{\BX'(\theta)}G (\delta_\alpha \BX)] = [\nabla_{\BX'(\theta+\alpha)}\nabla_{\BX'(\theta)} G](\delta_\alpha \BX).
\end{equation}
Thus using our previous calculations of the derivatives of $G(z)$, we have that 
the Peskin problem for a general tension can be written as an evolution equation for $\BX'(\theta)$ as
\begin{equation} \label{general.tension}
    \partial_t \BX'(\theta) = \int \HXTA \delta_\alpha  \bT(\BX')(\theta)  d\alpha.
\end{equation}
Here the kernel $\HXTA$ is given by
\begin{multline}\label{kernel.peskin.nosing}
\HXTA \eqdef \frac{1}{4\pi} \frac{\BX'(\theta+\alpha) \cdot \DO(\delta_\alpha \BX)\BX'(\theta)}{|\delta_\alpha \BX|^2} \IO
    \\
    -\frac{1}{4\pi} \frac{\BX'(\theta+\alpha) \cdot \RO(\delta_\alpha \BX)\BX'(\theta)}{|\delta_\alpha \BX|^2} \RO(\delta_\alpha \BX) 
    \\ 
    +\frac{1}{4\pi} \frac{\BX'(\theta+\alpha) \cdot (\DO(\delta_\alpha \BX)-\IO)\BX'(\theta)}{|\delta_\alpha \BX|^2} \DO(\delta_\alpha \BX).
\end{multline} 
Note that nothing we have done so far has implied periodicity of the solution $\BX(\theta)$, and that these forms of the equations work for any parametrization.  

Then further we can write $|\delta_\alpha \BX|^2 = \alpha^2 |\DATX|^2$ using \eqref{delta.notation} and \eqref{distance.X.notation}.  Thus we can write that $\HXTA  = \alpha^{-2}\KXTA$ where $\KXTA$ is given by \eqref{kernel.peskin.noExpand}.  This establishes equation \eqref{peskin.general.tension} from \eqref{general.tension} and \eqref{kernel.peskin.nosing}.

\section{Main estimate}\label{sec:BesovSpace}

In this section we will prove our main a priori estimate for the Peskin problem \eqref{peskin.general.tension} with a general tension \eqref{tension.map.def} in Proposition \ref{prop:general.apriori.final.local}.  To this end we let $\bX'(t,\theta)$ be the solution of the Peskin problem \eqref{peskin.general.tension} with the general tension map $\TE$ given in \eqref{tension.map.def} satisfying the assumptions from \secref{sec:TensionAssumptions} 
and the kernel given by \eqref{kerbel.eqn.deriv} with \eqref{kerbel.A.eqn.deriv}.   We consider intitial data for \eqref{peskin.general.tension},  $\bX_0$, satisfying
\begin{equation}\label{initial.assumption}
    ||\bX'_0||_{\BS}=||\bX'_0||_{\dot{B}^{\frac12,\subw}_{2,1}} \leq M, \qquad |\bX_0|_* = \inf_{\alpha\neq \theta} |\DAL \bX_0(\theta)| >0.
\end{equation}
Here $0 < M < \infty$ is allowed to be large.  We then suppose in this section that over a short time interval $T>0$ for some fixed $\rho>0$ we have
\begin{equation}\label{apriori.bd}
 |\BX(t)|_*\geq \rho, \quad 0 \le t \leq T.
\end{equation}
For some $\CTS >0$ we further suppose for $T>0$ that  we have
\begin{equation}\label{apriori.bd.CM}
    ||\bX'||_{\CM} \le \CTS M.
\end{equation}
We recall the notation \eqref{C.space.temporal}, \eqref{D.space.temporal}, and \eqref{initial.B.space}.  Then the main result in this section is the following proposition.

\begin{proposition}\label{prop:general.apriori.final.local}
Let $\bX: [0,T]\times \T \to \R^2$ be a weak solution to the Peskin problem with tension $\TE$ in the sense of Definition \ref{def:solution}.  Assume that $\BX_0,$ and $\bT$ satisfy the assumptions of Theorem \ref{thm:mainquant} including \eqref{initial.assumption}.  Additionally, assume that \eqref{apriori.bd} and \eqref{apriori.bd.CM} hold.  Then there are uniform constants $c, C>0$ such that the solution $\bX'(t,\theta)$ satisfies the following inequality 
\begin{equation*}
     ||\bX'||_{\CM}  
 + c \lambda^{1/2} ||\bX'||_{\DM} 
\leq 2 || \BX'_0||_{\BS}
+ 
CT^{1/2}
\BPU^{1/2}
 ||\bX'||_{\CM}.
\end{equation*}
Above $\BPU=\BPU[M,\rho,\lambda, \DTT, \DDTT]$ is defined in \eqref{BU.def}.   In particular there exists $T_M = T_M(M, \rho, \mu, \lambda, \DTT, \DDTT)>0$ such that if $T\leq T_M$ then we have
\begin{equation}\notag
    ||\bX'||_{\CM} + 2c\lambda^{1/2}||\bX'||_{\DM}\leq 4||\bX'_0||_{\BS}\leq 4M.
\end{equation}
\end{proposition}

In the rest of this section, we will prove Proposition \ref{prop:general.apriori.final.local}.  To that end, we first fix some arbitrary $|\beta|>0$.  Then direct calculation using \eqref{peskin.general.tension} gives 
\begin{equation}\notag
\begin{split}
    \frac{d}{dt} ||\delta_\beta \BX'(t)||_{L^2}^2 
    &= 2\int_{\T} d\theta ~\delta_\beta \BX'(\theta) \cdot \delta_\beta \partial_t \BX'(\theta)  
    \\
    &= 2\int_{\T}\int_{\T} d\theta d\alpha ~\frac{\delta_\beta \BX'(\theta)\cdot \delta_\beta \left(\mathcal{K}(\theta, \alpha) \delta_\alpha \bT(\BX')\right)}{\alpha^2}
    \\&= - \int_{\T}\int_{\T} d\theta d\alpha ~\frac{\delta_\beta \delta_\alpha \BX'(\theta)\cdot \delta_\beta \left(\mathcal{K}(\theta, \alpha) \delta_\alpha \bT(\BX')\right)}{\alpha^2}.
\end{split}
\end{equation}
We thus conclude that 
\begin{equation}\label{beta.L2.difference.1}
\begin{split}
    \frac{d}{dt} ||\delta_\beta \BX'(t)||_{L^2}^2 
    &= -\int_{\T}\int_{\T} d\theta d\alpha ~\frac{\delta_\beta \delta_\alpha \BX'(\theta)\cdot \tau_{\beta}\mathcal{K}(\theta, \alpha) \delta_\beta\delta_\alpha \bT(\BX')}{\alpha^2} 
    \\
    & \quad -\int_{\T}\int_{\T} d\theta d\alpha  ~\frac{\delta_\beta \delta_\alpha \BX'(\theta)\cdot \delta_\beta\mathcal{K}(\theta, \alpha) \delta_\alpha \bT(\BX')}{\alpha^2}.
\end{split}
\end{equation}
We will deal with these two integrals on the right side in order.

We next study the differences of the tension map.  
First we give the following useful lemma which tells us in particular that the operators $\delta_\alpha^\pm$ from \eqref{delta.pm.notation} and the kernel \eqref{kerbel.A.eqn.deriv} are bounded above by the same Besov space with the operator $\delta_\alpha$.

\begin{lemma}\label{continunity.delta.pm}
Let $\T = \R/2\pi \Z = [-\pi, \pi]$.  Recall the operators $\DAL$ from \eqref{distance.X.notation} and $\delta_\alpha^\pm$ from \eqref{delta.pm.notation}.  Then  for any $p\in [1,\infty]$ we have
\begin{equation}\label{operator.bd.first}
\begin{split}
    ||\delta_\alpha^- f'||_{L^p_\theta} \leq 2 ||f'||_{L^p_\theta}, \quad 
    ||\delta_\alpha^+ f'||_{L^p_\theta} \leq 4 ||f'||_{L^p_\theta}, \quad  ||D_\alpha f||_{L^p_\theta} \leq   ||f'||_{L^p_\theta}.
\end{split}
\end{equation}
Furthermore, fix $0<s<1$ and $p,q \in [1,\infty]$.  Then we have the uniform estimate
\begin{equation}\label{delta.operator.ineq}
\left(\int_{\T}  \frac{d\beta}{|\beta|^{1+sq}}  ||\delta_\beta^\pm f||^{q}_{L^p_\theta} \right)^{1/q}
    \lesssim
    || f ||_{\dot{B}^s_{p,q}}.
\end{equation}
We use the standard modification of the lower bound when $q=\infty$.
\end{lemma}

\begin{proof}
From \eqref{distance.X.notation} we have 
\begin{equation}\notag
\DAL f(\theta) = \int_0^1 d\tau~ f'(\theta+\tau\alpha). 
\end{equation}
Then also using \eqref{delta.pm.notation} we have 
\begin{equation}\notag
\delta_\alpha^- f'(\theta) 
=f'(\theta) - \int_0^1 d\tau~ f'(\theta+\tau\alpha)
=
-\int_0^1 d\tau~ \delta_{\tau\alpha} f'(\theta),
\end{equation}
and then
\begin{equation}\label{delta.plus.formula}
\delta_\alpha^+ f'(\theta) = \delta_\alpha f'(\theta) +\delta_\alpha^- f'(\theta). 
\end{equation}
Then from Minkowski's integral inequality we have for example
\begin{equation}\notag
    ||\delta_\alpha^- f' ||_{L^p_\theta} \leq \int_0^1 d\tau~ ||\delta_{\tau\alpha} f' ||_{L^p_\theta},
\end{equation}
Thus the inequalities in \eqref{operator.bd.first} follow from Minkowski's inequality, translation invariance, and the triangle inequality.

It remains to prove \eqref{delta.operator.ineq}.  We use Minkowski's integral inequality twice as
\begin{equation}\notag
   \left( \int_{\T} \frac{d\beta}{|\beta|^{1+sq}}||\delta_\beta^- f' ||_{L^p_\theta}^q     \right)^{1/q}
    \leq
    \int_0^1 d\tau~ \left(\int_{\T} \frac{d\beta}{|\beta|^{1+sq}}||\delta_{\tau\beta} f' ||_{L^p_\theta}^q \right)^{1/q}. 
\end{equation}
Now applying the change of variable $\alpha = \tau\beta$ we obtain \eqref{delta.operator.ineq} for $\delta_\beta^-$.  The inequality \eqref{delta.operator.ineq} for $\delta_\beta^+$ then follows from the formula \eqref{delta.plus.formula} and the triangle inequality.
\end{proof}

\begin{remark}\label{ignore.remark}
    In the rest of this article, when we use the estimates of \eqref{kerbel.A.eqn.deriv} in Lemma \ref{A.bound.lem}, Lemma \ref{lemm.A.diff} and Lemma \ref{lem:Abeta.upper} that are proven in \secref{sec:kernelDIFF}, we will only write the upper bounds with $\delta_\alpha$ in place of $\delta_\alpha^+$ and $\delta_\alpha^-$ from \eqref{delta.pm.notation}.   We use this simplification to ease the notation, but more importantly because these operators have no effect on our final estimates due to Lemma \ref{continunity.delta.pm} and the inequality in \eqref{delta.operator.ineq}.  We will also ignore the translation operator $\tau_\beta$ from \eqref{def.translation} when we use the estimates of \eqref{kerbel.A.eqn.deriv} as in \eqref{e:Abounds}, which is justified because all of the functional spaces that we are using in this article are translation invariant.  
\end{remark}

Now to begin studying the differences of the tension map in \eqref{beta.L2.difference.1} we write 
\begin{equation}\label{e:deltaalphaT}
    \delta_\alpha \bT(\BX'(\theta)) 
    = \int_0^1 ds_1 \frac{d}{ds_1} \bT( s_1 \delta_\alpha \BX'(\theta)+\BX'(\theta)) 
    = \DBTX \delta_\alpha \BX'(\theta), 
\end{equation}
where letting $D\bT(z)$ denote the derivative in \eqref{e:DTdefn} of the tension map $\BT(z)$ in \eqref{tension.map.def} we have
\begin{equation}\label{DBTX.def}
  \DBTX (\theta)\eqdef 
        \int_0^1 ds_1~  D\bT(g_1[\BX'](s_1, \alpha, \theta)),
\end{equation}
where
\begin{equation*}
    g_1[\BX'](s_1, \alpha, \theta) \eqdef s_1 \tau_\alpha \BX'(\theta) +
    (1-s_1)\BX'(\theta).
\end{equation*}
We therefore obtain from \eqref{e:QuantitativeTensionMap} that 
\begin{equation}\label{delta.alpha.BX.bound}
    |\delta_\alpha \bT(\BX'(\theta))| \leq \DTT |\delta_\alpha \BX'(\theta)|. 
\end{equation}
We will use this estimate on the second term in \eqref{beta.L2.difference.1}.

To study the first term in \eqref{beta.L2.difference.1}, we apply $\delta_\beta$ to $\delta_\alpha \bT(\BX'(\theta))$ to obtain
\begin{equation}\label{e:albe.T.XY}
    \delta_\beta \delta_\alpha \bT(\BX'(\theta)) = \tau_\beta \DBTX(\theta) \delta_\beta \delta_\alpha \BX'(\theta)
    +\delta_\beta\DBTX(\theta)  \delta_\alpha \BX'(\theta), 
\end{equation}
where
\begin{equation}\label{delta.beta.DBTX}
\delta_\beta\DBTX(\theta)  
=
\int_0^1 ds_1~ \DDBT[\BX'](\theta) (g_1[\delta_\beta\BX'](s_1, \alpha, \theta) ),
\end{equation}
and
\begin{equation}\notag
\DDBT[\BX'](\theta)
\eqdef
\int_0^1 ds_2~  D^2\bT(g_2[\BX'](s_2,s_1,\alpha,\theta,\beta)),
\end{equation}
with $g_2[\BX'](s_2,s_1,\alpha,\theta,\beta)$ given by
\begin{equation}\label{g2.s2.def}
    g_2[\BX'](,\alpha,\theta,\beta)s_2,s_1
\eqdef
s_2 g_1[\tau_\beta\BX'](s_1, \alpha, \theta)
   +(1-s_2) g_1[\BX'](s_1, \alpha, \theta).
\end{equation}
 We conclude using Remark \ref{ignore.remark} and \eqref{e:QuantitativeTensionMap} that 
\begin{equation}\label{delta.beta.DBTX.bound}
    |\delta_\beta\DBTX(\theta)  \delta_\alpha \BX'(\theta)| \leq 
    \DDTT |\delta_\beta  \BX'(\theta)|| \delta_\alpha \BX'(\theta)|,
\end{equation}
and
\begin{equation}\label{beta.alpha.TXP}
    |\delta_\beta \delta_\alpha \bT(\bX')| \leq \DTT  |\delta_\beta\delta_\alpha \bX'(\theta)| + \DDTT  |\delta_\beta \BX'(\theta)|  |\delta_\alpha \bX'(\theta)|.
\end{equation}
We will use this estimate on the first term in \eqref{beta.L2.difference.1}.  Also notice that we have  the matrix inequality in \eqref{e:DTdefn} for $\overline{D\bT}$, since 
the pointwise lowerbound for $D\bT(z)$ automatically applies to $\overline{D\bT}$ from \eqref{DBTX.def}.

Plugging all of this into \eqref{beta.L2.difference.1}, and using  \eqref{kerbel.eqn.deriv} with \eqref{kerbel.A.eqn.deriv}, \eqref{tildeLambda:eq} and the bounds from \eqref{e:DTdefn} we obtain 
\begin{equation}\label{beta.L2.difference}
        \frac{d}{dt} || \delta_\beta \bX' ||_{L^2_\theta}^2
    + \lambda
    || \delta_\beta \TLam^{\frac12} \bX' ||_{L^2_\theta}^2 
    \leq  \LLT_1+ \LLT_2 
    +   \LLT_3.
\end{equation}
Then with \eqref{kerbel.A.eqn.deriv} we have 
\begin{equation}\notag
    \LLT_1 \eqdef   \DTT \int_{\T} d\theta \int_{\T} \frac{d\alpha}{\alpha^2} ~
   | \delta_\beta\delta_\alpha  \bX'(\theta)|^2  |\tau_\beta \mathcal{A}[\bX](\theta, \alpha)|, 
\end{equation}
\begin{equation}\notag
    \LLT_2 \eqdef   \DTT \int_{\T} d\theta \int_{\T} \frac{d\alpha}{\alpha^2} ~
   | \delta_\beta\delta_\alpha  \bX'(\theta)| | \delta_\alpha  \bX'(\theta)|  |\delta_\beta \mathcal{A}[\bX](\theta, \alpha)|, 
\end{equation}
\begin{equation}\notag
    \LLT_3 \eqdef   \DDTT \int_{\T} d\theta \int_{\T} \frac{d\alpha}{\alpha^2} ~
   | \delta_\beta\delta_\alpha  \bX'(\theta)|  | \delta_\beta  \bX'(\theta)| | \delta_\alpha  \bX'(\theta)| |\tau_\beta \mathcal{K}[\bX](\theta, \alpha)|. 
\end{equation}
We will estimate each of the terms above.    

\begin{remark}
    Note that in the simple tension case, $D^2\bT \equiv 0,$ and hence as in \eqref{delta.beta.DBTX} in this case $\LLT_{3} \equiv 0.$  
\end{remark}

For $\LLT_{1}$ we split the kernel \eqref{kerbel.A.eqn.deriv}  as 
\begin{equation}\notag
 \mathcal{A}(\theta, \alpha)  
 =
 \mathcal{A}^S(\theta, \alpha)  
+
 \mathcal{A}^L(\theta, \alpha),  
\end{equation}
where for a fixed small $\eta>0$ to be chosen
\begin{equation}\label{split.K}
 \mathcal{A}^S(\theta, \alpha)  
 =
 \mathcal{A}(\theta, \alpha) \1_{|\alpha| < \eta},
\quad
 \mathcal{A}^L(\theta, \alpha)  
 =
 \mathcal{A}(\theta, \alpha) \1_{|\alpha| \ge \eta}.
\end{equation}
Thus we have 
\begin{equation}\notag
    \LLT_1^S \eqdef  \DTT \int_{\T} d\theta \int_{\T} \frac{d\alpha}{\alpha^2} ~
   | \delta_\beta\delta_\alpha  \bX'(\theta)|^2  |\tau_\beta \mathcal{A}[\bX](\theta, \alpha)| \1_{|\alpha|< \eta}.
\end{equation}
And we define $\LLT_{1}^L = \LLT_{1} - \LLT_{1}^S$.  Next, from \eqref{kerbel.A.eqn.deriv} and Remark \ref{ignore.remark} we have the general estimate
\begin{equation}\label{e:Abounds}
    |\mathcal{A}(\theta, \alpha)| \lesssim |\BX|_*^{-2} |\delta_\alpha \BX'(\theta)|^2 + |\BX|_*^{-1} | \delta_\alpha \BX'(\theta)|.
\end{equation}
Then we apply H{\"o}lder's inequality using \eqref{apriori.bd} to obtain
\begin{multline}\notag
   \LLT_1^S
\lesssim 
\DTT
\left( \int_{|\alpha|< \eta}   \frac{d\alpha}{\alpha^2} 
||\delta_\beta \delta_\alpha \BX'||_{L^4_\theta}^4  \right)^{\frac{1}{2}}
\left( \int_{|\alpha|< \eta}   \frac{d\alpha}{\alpha^2} 
|| \delta_\alpha \BX'||_{L^4_\theta}^4 \right)^{\frac{1}{2}}
\frac{1}{\rho^2}
\\
+
\DTT ||\delta_\beta \TLam^{\frac12} \BX'||_{L^2_\theta}
\left( \int_{|\alpha|< \eta}  \frac{d\alpha}{\alpha^2} 
||\delta_\beta \delta_\alpha \BX'||_{L^4_\theta}^4  \right)^{\frac{1}{4}}
\left( \int_{|\alpha|< \eta}  \frac{d\alpha}{\alpha^2} 
|| \delta_\alpha \BX'||_{L^4_\theta}^4 \right)^{\frac{1}{4}}\frac{1}{\rho}.
\end{multline}
In order to deal with this error term, we need to absorb it by the elliptic term.  A priori though, $\DTT$ and $\rho^{-1}$ could both be very large, and this might seem like we need to restrict our choice of tensions in \secref{sec:TensionAssumptions}.  

However, the function $\subw$ from Definition \ref{subw.definition} allows us to control the decay rate of the integral on the right hand side.  Specifically, for any $p> 1$ we have 
\begin{multline}\label{extra.smallness.besov}
 \int_{|\alpha|< \eta}   \frac{d\alpha}{\alpha^2} 
|| \delta_\alpha \BX'||_{L^p_\theta}^p 
\leq
  \frac{1}{\subw(\eta^{-1})^p}   \int_{\T}   \frac{d\alpha}{\alpha^2} 
|| \delta_\alpha \BX'||_{L^p_\theta}^p ~\subw(|\alpha|^{-1})^p 
=
\frac{ ||  \BX' ||_{\dot{B}_{p, p}^{\frac{1}{p},\mu}}^p}{\subw(\eta^{-1})^p}.
\end{multline}
Note that $\subw(\eta^{-1})^{-1}$ can be made arbitrarily small for $\eta>0$ small.  We next use the following embeddings from Proposition \ref{besov.ineq.prop} as 
\begin{equation}\label{embed.la.besov}
|| f ||_{\dot{B}^{\frac1p}_{p,p}}
\lesssim
|| f ||_{\dot{B}^{\frac12}_{2,p}}, \quad 
|| f ||_{\dot{B}^{\frac1p,\mu}_{p,p}}
\lesssim
|| f ||_{\dot{B}^{\frac12,\mu}_{2,p}}, \quad p\ge 2.
\end{equation}
We remark that we will also use the following inequality frequently in the rest of this paper, 
$||  f ||_{\dot{B}_{p, r_1}^{s}} \lesssim ||  f ||_{\dot{B}_{p, r_2}^{s}}$ which holds for any $r_1 \ge r_2 \ge 1$ and any $s\in \R$.

We will now use these inequalities in the form 
\begin{equation}\notag
    || \delta_\beta \BX'||_{\dot{B}^{\frac14}_{4,4}} \lesssim || \delta_\beta \BX'||_{\dot{B}^{\frac12}_{2,4}}
\lesssim 
|| \delta_\beta \BX'||_{\dot{B}^{\frac12}_{2,2}}
\approx
||\delta_\beta \TLam^{\frac12} \BX'||_{L^2_\theta}, 
\end{equation}
and also using \eqref{apriori.bd.CM} we have,
\begin{equation}\notag
|| \BX'||_{\dot{B}^{\frac{1}{4},\mu}_{4,4}} \leq C || \BX'||_{\dot{B}^{\frac{1}{2},\mu}_{2,4}}
\leq C \CTS M.
\end{equation}
We therefore conclude that
\begin{equation}\label{LAS.estimate}
   \LLT_1^S
     \leq
      C \kappa_1   \DTT
||\delta_\beta \bX'||_{\dot{B}^{\frac12}_{2,2}}^2 ,
\end{equation}
where we define $\kappa_1 = \kappa_1(\eta)$ by 
\begin{equation}\label{kappa1.def}
    \kappa_1 \eqdef \frac{ 1 }{\rho^2}  \frac{ M^2}{\subw(\eta^{-1})^2} +\frac{ 1}{\rho}  \frac{ M}{\subw(\eta^{-1})}.
\end{equation}
 This will be  our main estimate for $\LLT_1^S$.  We will later choose $\eta>0$ small enough so that $C \kappa_1  \DTT \ll \lambda$.

 Next we will estimate $\LLT_1^L$ containing $\mathcal{A}^L(\theta, \alpha)$ from \eqref{split.K} and   \eqref{kerbel.A.eqn.deriv} on the region $|\alpha| \geq \eta$. Noting that $||\delta_\alpha f||_{L^p} \leq 2 ||f||_{L^p}$ for all $p \in [1,\infty]$, we can neglect the $\delta_\alpha$'s and apply H{\"o}lder's inequality in $\theta$ to obtain
\begin{equation*}
 \LLT_1^L
\lesssim 
\DTT\left( \int_{|\alpha|\geq \eta}   \frac{d\alpha}{\alpha^2} \right)\left(\frac{||\delta_\beta \bX'||_{L_\theta^4}^2  ||\bX'||_{L_\theta^4}^2}{\rho^2} +  \frac{ ||\delta_\beta \bX'||_{L_\theta^2} ||\delta_\beta \bX'||_{L_\theta^4}  ||\bX'||_{L_\theta^4}}{\rho}\right).
\end{equation*}
Next from Proposition \ref{besov.ineq.prop}, Lemma \ref{Besov.increase}, and \eqref{apriori.bd.CM} we use the following  inequalities
 \begin{equation}\notag
 ||\bX'||_{L_\theta^4}
     \lesssim
|| \bX' ||_{\dot{B}^{0}_{4,2}}
     \lesssim
|| \bX' ||_{\dot{B}^{\frac14}_{2,2}}
     \lesssim
|| \bX' ||_{\dot{B}^{\frac12}_{2,\infty}} \lesssim M.
 \end{equation}
For the other term we use also Lemma \ref{Besov.interpolation} to obtain
 \begin{equation}\notag
 ||\delta_\beta \bX'||_{L_\theta^4}
     \lesssim
|| \delta_\beta \bX' ||_{\dot{B}^{0}_{4,2}}
     \lesssim
|| \delta_\beta \bX' ||_{\dot{B}^{\frac14}_{2,2}}
     \lesssim
||\delta_\beta \TLam^{\frac12}\bX'||_{L^2_\theta}^{1/2} ||\delta_\beta \bX'||_{L^2_\theta}^{1/2}. 
 \end{equation}
We also use that $\left( \int_{|\alpha|\geq \eta}   \frac{d\alpha}{\alpha^2} \right)  = \eta^{-1}$.  Then we obtain 
\begin{equation*}
     \LLT_1^L
\lesssim 
 \frac{\DTT M^2}{\eta\rho^2}
||\delta_\beta \TLam^{\frac12} \bX'||_{L^2_\theta} ||\delta_\beta \bX'||_{L^2_\theta} 
+  
 \frac{\DTT M}{\eta\rho}
||\delta_\beta \TLam^{\frac12}\bX'||_{L^2_\theta}^{1/2} ||\delta_\beta \bX'||_{L^2_\theta}^{3/2}.
\end{equation*}
Using Young's inequality, we can separate out the higher order terms 
and get 
\begin{equation}\label{estimate.LLTAL}
 \LLT_1^L
\leq \frac{\lambda}{8} ||\delta_\beta \bX'||_{\dot{B}^{\frac12}_{2,2}}^2 
+ C||\delta_\beta \bX'||_{L^2_\theta}^2
\BB,
\end{equation}
where
\begin{equation}\label{BB.const.def}
  \BB= \BB [M, \lambda, \rho, \DTT, \eta] \eqdef \frac{\DTT^2 M^4}{\lambda \eta^2 \rho^4} + \frac{\DTT^{4/3} M^{4/3}}{\lambda^{1/3} \eta^{4/3} \rho^{4/3}}.
\end{equation}
This is our main estimate for $\LLT_1^L$.

Now we can collect the estimates for $ \LLT_1^S$ in \eqref{LAS.estimate} and $ \LLT_1^L$ in \eqref{estimate.LLTAL} to obtain
\begin{equation}\label{LA.estimate}
 \LLT_1
     \leq
       \left(\frac{\lambda}{8}+     C \kappa_1   \DTT \right)||\delta_\beta \bX'||_{\dot{B}^{\frac12}_{2,2}}^2 + C ||\delta_\beta \bX'||_{L^2_\theta}^2 \BB.
\end{equation}
This is our main estimate for the term $\LLT_1$.

This estimate above motivates the following lemma.  First, for some $A>0$, we consider a typical term of the following form
\begin{equation}\label{typical.term}
    \LLT = \LLT[\bX'_1, \bX'_2, \bX'_3] \eqdef  A \int_{\T} d\theta \int_{\T} \frac{d\alpha}{\alpha^2} ~
   | \delta_\beta\delta_\alpha  \bX'_1|       | \delta_\beta  \bX'_2|
|\delta_\alpha \bX'_3|.
\end{equation}
Here $\BX'_i$ are given functions for $i=1,2,3$.  Then we have

\begin{lemma}\label{lem:typical.term} For any small constant $0<c<1$, and for $\lambda>0$ from \eqref{e:DTdefn}, for any small $\eta>0$ we have the following uniform estimate for \eqref{typical.term} as 
\begin{multline}\label{TTG.y.estimate}
      \LLT
\leq
c \lambda ||   \delta_\beta \TLam^{\frac12}\BX'_1||_{L^2_\theta}^2
+
\frac{C A^2}{\lambda \mu(\eta^{-1})^2} ||   \delta_\beta \BX'_2||_{L^\infty_\theta}^2
|| \BX'_3||_{\BS}^2
\\
+
c || \delta_\beta\BX'_1||_{L^2_\theta}^2 
+ 
CA^2 \eta^{-2}|| \delta_\beta\BX'_2||_{L^2_\theta}^2 ||\BX'_3||_{L^{\infty}_\theta}^2   .
\end{multline}
In particular if $\BX'_1 = \BX'_2$ then we also have
\begin{equation*}
\LLT
\leq
c \lambda ||   \delta_\beta \TLam^{\frac12}\BX'_1||_{L^2_\theta}^2
+
\frac{C A^2}{\lambda \mu(\eta^{-1})^2} ||   \delta_\beta \BX'_1||_{L^\infty_\theta}^2
|| \BX'_3||_{\BS}^2
+ 
C\frac{A}{\eta} || \delta_\beta\BX'_1||_{L^2_\theta}^2 ||\BX'_3||_{L^{\infty}_\theta}.
\end{equation*}
\end{lemma}

\begin{proof}[Proof of Lemma \ref{lem:typical.term}]
We split this term into $\LLT=\LLT^S + \LLT^L$ where $\LLT^S$ is restricted to the integration domain $|\alpha| < \eta$ and $\LLT^L$ to the domain $|\alpha| \ge \eta$ similar to \eqref{split.K}.  For $\LLT^S$ we use  H{\"o}lder's inequality to obtain
\begin{equation}\notag
      \LLT^S
\lesssim 
||   \delta_\beta \BX'_2||_{L^\infty_\theta}
\left( \int_{|\alpha| < \eta}   \frac{d\alpha}{\alpha^2} 
|| \delta_\alpha \delta_\beta\BX'_1||_{L^2_\theta}^2  \right)^{\frac{1}{2}}
\left( \int_{|\alpha| < \eta}   \frac{d\alpha}{\alpha^2} 
|| \delta_\alpha \BX'_3||_{L^{2}_\theta}^{2} \right)^{1/2}.
\end{equation}
We further use the embeddings \eqref{extra.smallness.besov} and \eqref{embed.la.besov} to obtain
\begin{equation}\notag
      \LLT^S
\lesssim 
||   \delta_\beta \TLam^{\frac12}\BX'_1||_{L^2_\theta}
||   \delta_\beta \BX'_2||_{L^\infty_\theta}
\frac{|| \BX'_3||_{\dot{B}_{2, 2}^{\frac{1}{2},\mu}}}{\mu(\eta^{-1})}.
\end{equation}
Next we estimate $\LLT^L$.  Again using H{\"o}lder's inequality we have
\begin{equation}\notag
\begin{split}
\LLT^L
&\lesssim 
|| \delta_\beta\BX'_1||_{L^2_\theta}  || \delta_\beta\BX'_2||_{L^2_\theta} ||\BX'_3||_{L^{\infty}_\theta}
\left( \int_{|\alpha| \geq \eta}   \frac{d\alpha}{\alpha^2} 
 \right)
 \\& \lesssim 
 \eta^{-1}
|| \delta_\beta\BX'_1||_{L^2_\theta}  || \delta_\beta\BX'_2||_{L^2_\theta} ||\BX'_3||_{L^{\infty}_\theta}.
\end{split}
\end{equation}
  We combine the estimates above and use the embedding $\BS \subset \dot{B}_{2, 2}^{\frac{1}{2},\mu}$ to obtain the following general estimate 
  \begin{multline}\label{TTG.est}
      \LLT
\leq
\frac{C A}{\mu(\eta^{-1})}
||   \delta_\beta \TLam^{\frac12}\BX'_1||_{L^2_\theta}
||   \delta_\beta \BX'_2||_{L^\infty_\theta}
|| \BX'_3||_{\BS}
\\
+
 C A \eta^{-1}
|| \delta_\beta\BX'_1||_{L^2_\theta}  || \delta_\beta\BX'_2||_{L^2_\theta} ||\BX'_3||_{L^{\infty}_\theta}.
\end{multline}
  Then \eqref{TTG.y.estimate} follows after applying Young's inequality.
\end{proof}

Next, we turn our attention towards bounding the term $\LLT_{3}$ in \eqref{beta.L2.difference}.  Recalling \eqref{kerbel.eqn.deriv}, \eqref{apriori.bd} and \eqref{e:Abounds}, we can bound $|\tau_\beta \mathcal{K}[\bX]|$ in general  as
\begin{equation}\notag
    |\tau_\beta \mathcal{K}[\bX](\theta, \alpha)| \leq C\left( 1 + \rho^{-1} || \BX' ||_{L^\infty_\theta}
+ \rho^{-2} || \BX' ||_{L^\infty_\theta}^2\right).
\end{equation}
Now we state the following useful embedding as
\begin{equation}\label{embed.infty}
    || \BX' ||_{L^\infty_\theta}
    \lesssim 
    || \BX' ||_{\dot{B}_{2,1}^{\frac12}}.
\end{equation}
This embedding follows as in Proposition \ref{besov.ineq.prop}.  
Then further using $|| \BX' ||_{\dot{B}_{2,1}^{\frac12}} \lesssim || \BX' ||_{\dot{B}_{2,1}^{\frac12,\subw}}$ and \eqref{apriori.bd.CM} we have
\begin{equation}\label{K.bound.infty}
 C\left( 1 + \rho^{-1} || \BX' ||_{L^\infty_\theta}
+ \rho^{-2} || \BX' ||_{L^\infty_\theta}^2\right) 
\leq C (1+ \rho^{-2} M^2)\eqdef \WK_1=\WK_1[\BX'].
\end{equation}
Thus we notice that the remaining part of $\LLT_{3}$ is in the form of \eqref{typical.term} with $\BX_1'=\BX_2'=\BX_3'=\BX'$.  Thus applying \eqref{TTG.est} we obtain
\begin{multline}\notag
      \LLT_{3}
\leq
C||   \delta_\beta \TLam^{\frac12}\BX'||_{L^2_\theta}
||   \delta_\beta \BX'||_{L^\infty_\theta}
\frac{|| \BX'||_{\BS}}{\mu(\eta^{-1})} \WK_1[\BX'] \DDTT
\\
+
 C\eta^{-1}
|| \delta_\beta\BX'||_{L^2_\theta}^2 ||\BX'||_{L^{\infty}_\theta}\WK_1[\BX'] \DDTT.
\end{multline}
We further apply Young's inequality to the first term above, and use \eqref{apriori.bd.CM},  to obtain 
\begin{equation}\label{LTT12.estimate}
      \LLT_{3}
\leq
\frac{\lambda}{8}||\delta_\beta \TLam^{\frac12}\bX'||_{L^2_\theta}^2
+     
C \kappa_2   \DDTT^2 
||   \delta_\beta \BX'||_{L^\infty_\theta}^2
+
 C
|| \delta_\beta\BX'||_{L^2_\theta}^2  \BK,
\end{equation}
where recalling \eqref{K.bound.infty} we have 
\begin{equation}\label{BK.const.def}
  \BK = \BK [M, \rho^{-1}, \DDTT, \eta] \eqdef 
 \eta^{-1}   M(1+ \rho^{-2} M^2) \DDTT,
\end{equation}
and $\kappa_2 = \kappa_2(\eta)$ is
\begin{equation}\label{kappa2.def}
    \kappa_2 \eqdef \lambda^{-1} \frac{M^2}{\subw(\eta^{-1})^2} (1+ \rho^{-2} M^2)^2.
\end{equation}
 This will be  our main estimate for $\LLT_{3}$.  We will later choose $\eta>0$ small enough so that under our assumptions $C\kappa_2  \DDTT^2 \ll \lambda$.

To prove further estimates we will now state the following lemma which gives the pointwise estimates of $\delta_\beta \mathcal{A}(\theta, \alpha)$.  The proof of Lemma \ref{A.bound.lem} is given in \secref{sec:kernelDIFF}.

\begin{lemma}\label{A.bound.lem} 
Considering $\mathcal{A}(\theta, \alpha)$ from \eqref{kerbel.A.eqn.deriv} we can split $\delta_\beta \mathcal{A}(\theta, \alpha)$ as 
\begin{equation}\label{Abeta.split}
     \delta_\beta \mathcal{A}(\theta, \alpha)  =   \mathcal{A}_{1\beta}(\theta, \alpha)  +  \mathcal{A}_{2\beta}(\theta, \alpha), 
\end{equation}
where $\mathcal{A}_{1\beta}(\theta, \alpha)$ satisfies the following uniform upper bound
\begin{multline}\label{A1X.bound}
\left| \mathcal{A}_{1\beta}(\theta, \alpha) \right|
\lesssim
\frac{| \delta_\beta \DAP(\theta)| | \tau_\beta \DAM(\theta)| +|  \DAP(\theta)| | \delta_\beta \DAM(\theta)| }{|\BX|_*^2}
\\
+
\frac{| \delta_\beta \DAP(\theta)|+| \delta_\beta \DAM(\theta)| }{|\BX|_*}.
\end{multline}
Further $\mathcal{A}_{2\beta}(\theta, \alpha)$ satisfies the  uniform upper bound
\begin{multline}\label{A2X.bound}
\left| \mathcal{A}_{2\beta}(\theta, \alpha)  \right| 
\lesssim  
|  \DAP(\theta)| 
\left( | \tau_\beta \DAM(\theta)| + | \DAM(\theta)| \right)
\frac{| \delta_\beta \DAX(\theta)| }{|\BX|_*^3}
\\
+
\left( 
|  \DAP(\theta)| 
+
| \DAM(\theta)| 
\right)
\frac{| \delta_\beta \DAX(\theta)| }{|\BX|_*^2}.
\end{multline}
\end{lemma}

Next we will estimate the term $\LLT_{2}$ from \eqref{beta.L2.difference}. For future use we will estimate the following more general term with a constant $A>0$ as
\begin{equation}\label{LLT2.generalform}
    \LLT_2[\bX'_1, \bX'_2, \bX'_3] \eqdef   A  \int_{\T} d\theta \int_{\T} \frac{d\alpha}{\alpha^2} ~
   | \delta_\beta\delta_\alpha  \bX'_1(\theta)|  |\delta_\beta \mathcal{A}[\bX_2](\theta, \alpha)| | \delta_\alpha  \bX'_3(\theta)| . 
\end{equation}
Here $\bX_1$, $\bX_2$ and $\bX_3$ are given functions.  Then  in Lemma \ref{A.bound.lem} we split the kernel from \eqref{kerbel.A.eqn.deriv} as $\delta_\beta \mathcal{A}[\bX_2](\theta, \alpha)=   \mathcal{A}_{1\beta}[\bX_2](\theta, \alpha)  +  \mathcal{A}_{2\beta}[\bX_2](\theta, \alpha)$.  Taking into account Remark \ref{ignore.remark}, from \eqref{A1X.bound} and \eqref{A2X.bound} we have
\begin{equation}\label{A1betaRemark}
    \left| \mathcal{A}_{1\beta}[\bX_2](\theta, \alpha) \right|
\lesssim
\frac{| \delta_\beta \delta_\alpha \BX'_2(\theta)| |\delta_\alpha \BX'_2(\theta)|  }{|\BX_2|_*^2}
+
\frac{| \delta_\beta \delta_\alpha \BX'_2(\theta)|}{|\BX_2|_*},
\end{equation}
and
\begin{equation}\label{A2betaRemark}
\left| \mathcal{A}_{2\beta}[\bX_2](\theta, \alpha)  \right| 
\lesssim  
\frac{|  \delta_\alpha \BX'_2(\theta)|^2  | \delta_\beta \DAL \BX_2(\theta)| }{|\BX_2|_*^3}
+
\frac{|  \delta_\alpha \BX'_2(\theta)|  | \delta_\beta \DAL \BX_2(\theta)| }{|\BX_2|_*^2}.
\end{equation}
Now we split $\LLT_{2}=\LLT_{21}+\LLT_{22}$ according to \eqref{Abeta.split}.  In particular $\LLT_{21}$ is the term $\LLT_{2}$ with $\delta_\beta \mathcal{A}(\theta, \alpha)$ replaced by $\mathcal{A}_{1\beta}(\theta, \alpha)$.  Now notice that $\LLT_{21}[\bX'_1, \bX'_2, \bX'_3]=\LLT_{21}$  satisfies the upper bound 
\begin{equation} \notag 
 \LLT_{21}
\lesssim
A  \sum_{j=1}^2 \rho^{-j} || \BX'_2||_{L^\infty_\theta}^{j-1}\int_{\T}  d\theta \int_{\T}  \frac{d\alpha}{\alpha^2}  
   | \delta_\beta\delta_\alpha  \bX'_1(\theta)|  
  | \delta_\beta\bX'_2(\theta)| 
   | \delta_\alpha  \bX'_3(\theta)|. 
\end{equation}
Therefore as in \eqref{typical.term} we have from \eqref{TTG.y.estimate} the estimate 
\begin{multline}\label{LLT21.estimate}
 \LLT_{21}
\leq
c \lambda ||   \delta_\beta \TLam^{\frac12}\BX'_1||_{L^2_\theta}^2
+
  \frac{C A^2}{\lambda\subw(\eta^{-1})^2} \left( \sum_{j=1}^2 \rho^{-j} || \BX'_2||_{L^\infty_\theta}^{j-1}\right)^2 ||   \delta_\beta \BX'_2||_{L^\infty_\theta}^2
|| \BX'_3||_{\BS}^2
\\
+
c || \delta_\beta\BX'_1||_{L^2_\theta}^2 
+ 
C|| \delta_\beta\BX'_2||_{L^2_\theta}^2 ||\BX'_3||_{L^{\infty}_\theta}^2   A^2  \eta^{-2}    
\left( \sum_{j=1}^2 \rho^{-j} || \BX'_2||_{L^\infty_\theta}^{j-1}\right)^2.
\end{multline}
This is our main estimate for the term $\LLT_{21}[\bX'_1, \bX'_2, \bX'_3]$.

Lastly we will estimate $\LLT_{22}=\LLT_{22}[\bX'_1, \bX'_2, \bX'_3]$.  For this term we have the upper bound 
\begin{equation} \notag 
  \left|  \LLT_{22} \right|
\lesssim
A \sum_{j=2}^3 \rho^{-j} || \BX'_2||_{L^\infty_\theta}^{j-1}\int_{\T}  d\theta \int_{\T}  \frac{d\alpha}{\alpha^2}  
 | \delta_\beta \delta_\alpha \BX'_1(\theta) |
 | \delta_\beta \DAL \BX_2(\theta)|
  | \delta_\alpha \BX'_3(\theta)|.
\end{equation}
We can estimate this term the same way that we estimated $\LLT_{3}$ in \eqref{LTT12.estimate} using Lemma \ref{lem:typical.term}.  This follows because the term $  | \delta_\beta \DAL \BX_2(\theta)|$ in $\LLT_{22}$ is treated exactly as the term $| \delta_\beta \BX_2'(\theta)|$ in \eqref{typical.term} and \eqref{TTG.y.estimate}.  We can do that as in \eqref{operator.bd.first} because 
\begin{equation}\label{averaging.infinity.est}
    ||   \delta_\beta \DAL \BX_2||_{L^p_\theta}
    \lesssim
        || \delta_\beta \BX'_2 ||_{L^p_\theta},
        \quad p \in[1,\infty].
\end{equation}
Thus as in \eqref{TTG.y.estimate} we have
\begin{multline}\label{LLT22.estimate}
 \LLT_{22}
\leq
c \lambda ||   \delta_\beta \TLam^{\frac12}\BX'_1||_{L^2_\theta}^2
+
 \frac{C A^2}{\lambda\subw(\eta^{-1})^2} \left( \sum_{j=2}^3 \rho^{-j} || \BX'_2||_{L^\infty_\theta}^{j-1}\right)^2 ||   \delta_\beta \BX'_2||_{L^\infty_\theta}^2
|| \BX'_3||_{\BS}^2
\\
+
c || \delta_\beta\BX'_1||_{L^2_\theta}^2 
+ 
C|| \delta_\beta\BX'_2||_{L^2_\theta}^2 ||\BX'_3||_{L^{\infty}_\theta}^2 A^2 \eta^{-2}    
\left( \sum_{j=2}^3 \rho^{-j} || \BX'_2||_{L^\infty_\theta}^{j-1}\right)^2.
\end{multline}
This is our main estimate for the term $\LLT_{22}$.

Then for the term $\LLT_{2}=\LLT_{2}[\bX'_1, \bX'_2, \bX'_3]$ from \eqref{LLT21.estimate} and \eqref{LLT22.estimate} we have for any small constant $0<c<1$ that
\begin{multline}\label{LLT2.bound.general}
 \LLT_{2}
\leq 
c \lambda ||   \delta_\beta \TLam^{\frac12}\BX'_1||_{L^2_\theta}^2
+     
C\frac{A^2 \BG [\bX'_2 ]}{\lambda\subw(\eta^{-1})^2}
 ||   \delta_\beta \BX'_2||_{L^\infty_\theta}^2 || \BX'_3||_{\BS}^2
\\
+
c || \delta_\beta\BX'_1||_{L^2_\theta}^2 
+
C|| \delta_\beta\BX'_2||_{L^2_\theta}^2 ||\BX'_3||_{L^{\infty}_\theta}^2 
A^2 \eta^{-2}  \BG [\bX'_2 ]   .
\end{multline} 
where 
\begin{equation}\label{BG.const.def}
\BG =\BG [||\bX'_2||_{L^\infty_\theta}, \rho^{-1} ] \eqdef 
\left( \sum_{j=1}^2 \rho^{-j} || \BX'_2||_{L^\infty_\theta}^{j-1}\right)^2+\left( \sum_{j=2}^3 \rho^{-j} || \BX'_2||_{L^\infty_\theta}^{j-1}\right)^2.
\end{equation}
The above general estimate will be used in \secref{sec:strongCont}.

Specifically for $\LLT_{2}=\LLT_{2}[\bX', \bX', \bX']$ from \eqref{beta.L2.difference} following a similar procedure we  obtain
\begin{equation}\label{LLT2.bound}
 \LLT_{2}
\leq \frac{\lambda}{4}||\delta_\beta \bX'||_{\dot{B}^{\frac12}_{2,2}}^2 + 
C  \kappa_3  \DTT^2
||\delta_\beta \bX'||_{L^\infty_\theta}^2 
+ C||\delta_\beta\bX'||_{L^2_\theta}^2 \DTT \BC, 
\end{equation}
where recalling \eqref{apriori.bd.CM} then 
$\kappa_3 = \kappa_3[\bX' ](\eta)$ is
\begin{equation}\label{kappa3.def}
    \kappa_3 \eqdef  
    \frac{ M^2}{\lambda\subw(\eta^{-1})^2}
 \left( \sum_{j=1}^2 \rho^{-j} M^{j-1}\right)^2(1+\rho^{-2} M^2),
\end{equation}
and
\begin{equation}\label{BC.const.def}
   \BC =\BC [M, \rho^{-1}, \eta^{-1} ] \eqdef 
M \eta^{-1}    
 \sum_{j=1}^3 \rho^{-j} M^{j-1}.
\end{equation}
This is our main estimate for the term $\LLT_{2}$.  We will later choose $\eta>0$ small enough so that under our assumptions $C\kappa_3  \DTT^2 \ll \lambda$.

Putting together our estimates for $\mathcal{L}_1$ \eqref{LA.estimate}, $\mathcal{L}_2$  \eqref{LLT2.bound} and $\mathcal{L}_3$  \eqref{LTT12.estimate} into \eqref{beta.L2.difference} we arrive at 
\begin{multline*}
        \frac{d}{dt}||\delta_\beta \bX'||_{L^2_\theta}^2 + \frac{\lambda}{2} ||\delta_\beta \bX'||_{\dot{B}^{\frac12}_{2,2}}^2 
        \leq 
        C\kappa_1   \DTT ||\delta_\beta \bX'||_{\dot{B}^{\frac12}_{2,2}}^2 
        \\
        +C\left( \DDTT^2 \kappa_2+\DTT^2 \kappa_3  \right)||\delta_\beta \bX'||_{L^\infty_\theta}^2 
+ C||\delta_\beta \bX'||_{L^2_\theta}^2 \left(\BB+\BK+\DTT\BC \right),
\end{multline*}
where we recall \eqref{kappa1.def}, \eqref{kappa2.def}, \eqref{kappa3.def},  \eqref{BB.const.def}, \eqref{BK.const.def} and \eqref{BC.const.def} respectively.  

For convenience from \eqref{BB.const.def}, \eqref{K.bound.infty}, \eqref{BK.const.def} and \eqref{BC.const.def} we now define $\BPU$ by 
\begin{equation}\label{BU.def}
\BPU =\BPU [M,\rho,\lambda, \DDTT, \DTT] \eqdef \BB+\BK+\DTT\BC.
\end{equation}
From \eqref{kappa2.def} and \eqref{kappa3.def}  we also define 
\begin{equation}\notag
    \kappa_0 \eqdef  \DDTT^2 \kappa_2+\DTT^2 \kappa_3.
\end{equation}
Now we can choose $\eta>0$ small enough so that $C\kappa_1   \DTT < \lambda/4$.  Thus we obtain
\begin{equation*}
            \frac{d}{dt}||\delta_\beta \bX'||_{L^2_\theta}^2 + \frac{\lambda}{4} ||\delta_\beta \bX'||_{\dot{B}^{\frac12}_{2,2}}^2 
        \leq 
        C \kappa_0 ||\delta_\beta \bX'||_{L^\infty_\theta}^2 
+ C\BPU ||\delta_\beta \bX'||_{L^2_\theta}^2.
\end{equation*}
Further integrating in time, we get that 
\begin{multline*}
           ||\delta_\beta \bX'||_{L^2_\theta}^2(t) + \frac{\lambda}{4}  ||\delta_\beta \bX'||_{L^2_t(\dot{B}^{\frac12}_{2,2})}^2 
       \leq ||\delta_\beta \bX_0 '||_{L^2_\theta}^2
       +
 C \kappa_0 ||\delta_\beta \bX'||_{L^2_t(L^\infty_\theta)}^2 
 \\
+ C\BPU ||\delta_\beta \bX'||_{L^2_t(L^2_\theta)}^2.
\end{multline*}
Note that trivially, we have the bound
\begin{equation}\label{betaX.L2.Linfinity.bound}
||\delta_\beta \bX'||_{L^2_t(L^2_\theta)}^2\leq t ||\delta_\beta \bX'||_{L^\infty_t(L^2_\theta)}^2.
\end{equation}
Now we take the essential supremum over $0\le t \le T$, at the cost of an extra factor of 2 on the RHS, to obtain
\begin{multline}\label{ineq.refer.later}
           ||\delta_\beta \bX'||_{L^\infty_T(L^2_\theta)}^2 + \frac{\lambda}{4}  ||\delta_\beta \bX'||_{L^2_T(\dot{B}^{\frac12}_{2,2})}^2 
       \leq 2 ||\delta_\beta \bX_0 '||_{L^2_\theta}^2
       +
 C \kappa_0 ||\delta_\beta \bX'||_{L^2_T(L^\infty_\theta)}^2 
 \\
+ CT
\BPU ||\delta_\beta \bX'||_{L^\infty_T(L^2_\theta)}^2.
\end{multline}
Next, note that for any constants $A$ and $B$ we have
\begin{equation}\label{ineq.gen}
    \frac{1}{\sqrt{2}} (|A|+|B|) \le (A^2+B^2)^{1/2} \le  |A|+|B|.
\end{equation}
So taking the previous inequality and raising it to the $1/2$ power, we obtain 
\begin{multline*}
           ||\delta_\beta \bX'||_{L^\infty_T(L^2_\theta)} + \left(\frac{\lambda}{4}\right)^{1/2}  ||\delta_\beta \bX'||_{L^2_T(\dot{B}^{\frac12}_{2,2})}
       \leq 2 ||\delta_\beta \bX_0 '||_{L^2_\theta}
       +
 C \kappa_0^{1/2} ||\delta_\beta \bX'||_{L^2_T(L^\infty_\theta)} 
 \\
+ CT^{1/2}
\BPU^{1/2}||\delta_\beta \bX'||_{L^\infty_T(L^2_\theta)}.
\end{multline*}
Then further integrating the above in $d\beta$ against $|\beta|^{-3/2}\mu(|\beta|^{-1})$ thus gives us 
\begin{multline*}
 ||\bX'||_{\CM}  
 + \left(\frac{\lambda}{4}\right)^{1/2}  ||\bX'||_{\DM} 
\leq 2 || \BX'_0||_{\BS}
       \\
 +C \kappa_0^{1/2}
 \int_{\T} \frac{d\beta}{|\beta|^{3/2}}\mu(|\beta|^{-1}) ||\delta_\beta \bX'||_{L^2_T(L^\infty_\theta)} 
+ 
CT^{1/2}
\BPU^{1/2}
 ||\bX'||_{\CM}.
\end{multline*}
To handle the term containing $\kappa_0^{1/2}$ we will use the following lemma.

\begin{lemma}\label{L.infinity.embedding}
There exists a constant $C_{\mu}>0$ such that 
\begin{equation}\label{embedding.infty.use}
||f||_{\widetilde{L}^2_T(\dot{B}^{\frac12,\mu}_{\infty,1})} 
 \leq 
    C_{\mu}    \int_{\T} \frac{d\beta}{|\beta|^{3/2}} \subw(|\beta|^{-1}) ||\delta_\beta \TLam^{\frac12} f||_{L^2_T(L^{2}_{\theta})} 
    =
    C_{\mu}    ||f||_{\DM}. 
\end{equation}
\end{lemma}
The proof of Lemma \ref{L.infinity.embedding} is a direct combination of Proposition \ref{besov.ineq.prop} with \eqref{bernstein.2}.   Then after using Lemma \ref{L.infinity.embedding} we can further choose $\eta>0$ small enough so that $\left(\frac{\lambda}{4}\right)^{1/2} - C \kappa_0^{1/2} C_{\mu} \geq c \lambda^{1/2}>0$ for some small positive constant $c \ll 1$.  We thus obtain 
\begin{equation*}
     ||\bX'||_{\CM}  
 + c \lambda^{1/2} ||\bX'||_{\DM} 
\leq 2 || \BX'_0||_{\BS}
+ 
CT^{1/2}
\BPU^{1/2}
 ||\bX'||_{\CM}.
\end{equation*}
This completes the proof of Proposition \ref{prop:general.apriori.final.local}.

\section{Control of the arc-chord condition}\label{sec:ArcChord}

In this section we will establish the a priori control over the arc-chord condition defined with \eqref{arc.cord.number} for a solution to the Peskin problem \eqref{peskin.general.tension} with a general tension \eqref{tension.map.def} satisfying the a priori estimates \eqref{apriori.bd.norm}.  Recall from \eqref{peskin.expand.tension} with  \eqref{kerbel.A.eqn.deriv} and \eqref{tildeLambda:eq} that $\BX'(t)$ solves the equation 
\begin{equation}\label{v.theta.def}
\partial_t \bX' + \TLam \bT(\bX') = \mathcal{V}(\theta),
\quad     \mathcal{V}(\theta) \eqdef \int_{\T} \frac{d\alpha}{\alpha^2} \mathcal{A}(\theta, \alpha) \delta_\alpha \bT(\bX'(\theta)).
\end{equation}
We suppose that we are given initial data satisfying \eqref{initial.assumption} for equation \eqref{v.theta.def}.
For some $\CTS >0$ we will further suppose for $T>0$ that for some $c>0$ and for $\lambda>0$ as in \eqref{e:DTdefn} that we have
\begin{equation}\label{apriori.bd.norm}
    ||\bX'||_{\CM}+    c \lambda^{\frac12}||\bX'||_{\DM} \le \CTS M.
\end{equation}
Next we have the following estimate on the $L^2_T(\dot{H}^1)$ norm of a solution.

\begin{lemma}\label{H1.small.time}
Given a solution to \eqref{v.theta.def} satisfying \eqref{apriori.bd} and \eqref{apriori.bd.norm}.  For any $\varepsilon>0,$ there exists $T_{\varepsilon} = T(\varepsilon,M,\mu,\rho , \lambda )>0$ such that 
\begin{equation}\notag
    \int_0^{T_{\varepsilon}}ds~ ||\bX'(s)||_{\dot{H}^1}^2   \leq \varepsilon.
\end{equation}
\end{lemma}

\begin{proof}
 We split into $|\beta| < \eta$ and $|\beta| \ge \eta$ for some small $\eta>0$ to be chosen:
\begin{multline}\notag
    \int_0^T ||\bX'(s)||_{\dot{H}^1}^2ds \lesssim  
    \int_0^T ds \int_{\T} d\theta \int_{\T} \frac{d\beta}{\beta^2}  \int_{\T} \frac{d\alpha}{\alpha^2} |\delta_\beta \delta_\alpha \BX'(s,\theta)|^2
\\ 
\lesssim \frac{1}{\mu(\eta^{-1})^2}  \int_{|\beta| < \eta} \frac{d\beta}{\beta^2} \subw(|\beta|^{-1})^2
   \int_0^T ds \int_{\T} d\theta  \int_{\T} \frac{d\alpha}{\alpha^2} |\delta_\beta \delta_\alpha \BX'(s,\theta)|^2
\\
+ \frac{1}{\eta} \int_{\T} d\theta \int_0^T ds\int_{\T} d\alpha \frac{| \delta_\alpha \BX'(s,\theta)|^2}{\alpha^2} 
\\
\lesssim \frac{||\TLam^{\frac12}\BX'||_{\widetilde{L}^{2}_T(\dot{B}_{2, 2}^{\frac12, \mu})}^2}{\mu(\eta^{-1})^2} + \frac{T ||\BX'||_{\widetilde{L}^{\infty}_T(\dot{B}_{2, 2}^{\frac12})}^2}{\eta}
\lesssim \frac{\CTS^2 M^2}{\mu(\eta^{-1})^2\lambda} + \frac{\CTS^2 M^2T}{\eta}.
\end{multline}
Above we use the spaces from \eqref{Besov.CL.Space} with \eqref{Besov.mu.Space} as in \eqref{C.space.temporal} and \eqref{D.space.temporal}, and the last line follows from \eqref{apriori.bd.norm}.  We can choose $\eta>0$ small enough so that 
$C \frac{\CTS^2 M^2}{\mu(\eta^{-1})^2\lambda} < \frac12 \varepsilon$, and then we can choose $T=T_{\varepsilon}$ small enough so that $C\frac{\CTS^2 M^2 T}{\eta}< \frac12 \varepsilon$.
\end{proof}

Next, we prove the following lemma, which controls the $L^2(\T)$ norm of the time derivative of a solution by the $\dot{H}^1(\T)$ norm.

\begin{lemma}\label{time.space.equivalent}
A solution to \eqref{v.theta.def} satisfying \eqref{apriori.bd} and \eqref{apriori.bd.CM} has the estimate
\begin{equation}\notag
    ||\partial_t \bX'(t)||_{
    L^2(\T)} \leq C_1||\bX'(t)||_{\dot{H}^1(\T)},
\end{equation}
for some constant $C_1 = C_1(M, \rho, \DTT)>0$ and for any time $0<t<T$.
\end{lemma}

\begin{proof}
We use the equation \eqref{v.theta.def} to obtain that
\begin{equation}\notag
    ||\partial_t \BX'||_{L^2}\leq ||\TLam \bT(\BX')||_{L^2} + ||\mathcal{V}||_{L^2}.
\end{equation}
We will therefore estimate each of the two terms in the upper bound.
For the first term, we have from \eqref{e:QuantitativeTensionMap} that 
\begin{equation}\notag
\begin{split}
    ||\TLam \bT(\BX')||_{L^2} & \approx ||\bT(\BX')||_{\dot{H}^1} \approx ||D\bT(\BX') \bX''||_{L^2} \\&\lesssim \DTT ||\bX''||_{L^2} \lesssim \DTT ||\bX'||_{\dot{H}^1}.
\end{split}
\end{equation}
For the term $||\mathcal{V}||_{L^2}$, by the structure of $\mathcal{A}$ from \eqref{e:Abounds} with \eqref{apriori.bd} and \eqref{delta.alpha.BX.bound}, it is straightforward to get that 
\begin{equation}\label{Vbound.g}
|\mathcal{V}(\theta)|\lesssim  \DTT\int_{\T} d\alpha ~ \left(   \frac{ |\delta_\alpha \BX'(\theta)|^2}{\rho \alpha^2}+\frac{|\delta_\alpha \BX'(\theta)|^3}{\rho^2 \alpha^2}\right).
\end{equation}
Applying Minkowski's inequality, we then get that 
\begin{equation}\notag
 ||\mathcal{V}||_{L^2} \lesssim \DTT  \int_{\T} 
   \frac{d\alpha}{\alpha^2} \left(  \rho^{-1}  ||\delta_\alpha \BX'||_{L^4}^2 + \rho^{-2}    ||\delta_\alpha \BX'||_{L^6}^3 \right) .
\end{equation}
In terms of the Besov spaces the upper bound above is 
\begin{equation}\notag
||\mathcal{V}||_{L^2} \lesssim 
  \DTT \rho^{-1} ||\BX'||_{\dot{B}^{1/2}_{4,2}}^2 + \DTT\rho^{-2}   ||\BX'||_{\dot{B}^{1/3}_{6,3}}^3.
\end{equation}
From Proposition \ref{besov.ineq.prop} and Lemma \ref{Besov.interpolation}, we have the  embedding inequalities
\begin{equation}\notag
    ||\bX'||_{\dot{B}_{4,2}^{1/2}}^2 \lesssim
    ||\bX'||_{\dot{B}_{2,2}^{3/4}}^2 \lesssim
    ||\BX'||_{\dot{B}^{1/2}_{2,1}} \ ||\BX'||_{\dot{H}^1}, 
\end{equation}
and
\begin{equation}\notag
    ||\bX'||_{\dot{B}_{6,3}^{1/3}}^3 \lesssim ||\bX'||_{\dot{B}_{2,2}^{2/3}}^3 \lesssim ||\BX'||_{\dot{B}^{1/2}_{2,1}}^2 \ ||\BX'||_{\dot{H}^1}.
\end{equation}
Plugging in these inequalities and using \eqref{apriori.bd.norm} we have
\begin{equation}\notag
||\mathcal{V}||_{L^2} \lesssim 
  \DTT \left(  \rho^{-1}  M + \rho^{-2}   M^2 \right)  ||\BX'||_{\dot{H}^1}.
\end{equation}
This completes the proof.
\end{proof}

\begin{corollary}\label{prop:time.estimate}
Given a solution to \eqref{peskin.general.tension} satisfying both \eqref{apriori.bd} and \eqref{apriori.bd.norm}.  Then for any $\varepsilon>0$ there exists a time $T_\varepsilon = T(\varepsilon, M, \mu, \rho, \lambda, \DTT)>0$ such that 
\begin{equation*}
    \int_0^{T_\varepsilon}dt~ ||\partial_t \bX'(t)||_{L^2_\theta}^2 < \varepsilon.
\end{equation*}
\end{corollary}

The proof of Corollary \ref{prop:time.estimate} follows from Lemma \ref{time.space.equivalent} and Lemma \ref{H1.small.time}.

\begin{proposition}\label{prop:arc.chord.small}
Given a solution to \eqref{peskin.general.tension} satisfying both \eqref{apriori.bd} and \eqref{apriori.bd.norm}. Then for any small $\varepsilon>0$ there exists a time $T_\epsilon = T(\epsilon, M, \mu, \rho, \lambda, \DTT)>0$ such that
\begin{equation}\notag
    ||\bX'(t)-\bX'_0||_{L^\infty_\theta} < \varepsilon,
\end{equation}
for all $0<t<T_{\epsilon}$.  
\end{proposition}

\begin{proof}
We use the embedding \eqref{embed.infty} and \eqref{apriori.bd.norm}, and then we have for any $\eta>0$:
\begin{multline}\notag
    ||\bX'(t)-\bX'_0||_{L^\infty_\theta} \lesssim \int_\T \frac{d\beta}{|\beta|^{3/2}} ||\delta_\beta (\bX'(t)-\bX'_0)||_{L^2_\theta} 
    \lesssim \int_{|\beta| < \eta } + \int_{|\beta| \geq \eta }
    \\ \lesssim \frac{||\bX'||_{\CM}+||\bX'_0||_{\BS}}{\mu(\eta^{-1})} + \frac{||\bX'(t) - \bX'_0||_{L^2_\theta}}{\eta^{1/2}}\lesssim \frac{\CTS M}{\mu(\eta^{-1})} + \frac{||\bX'(t) - \bX'_0||_{L^2_\theta}}{\eta^{1/2}}.
\end{multline}
Now fix $\varepsilon>0$ small.  Then 
we take $\eta$ sufficiently small, and we can guarantee
\begin{equation}\notag
    ||\bX'(t)-\bX'_0||_{L^\infty_\theta} \leq \frac{\varepsilon}{2} + C\frac{||\bX'(t) - \bX'_0||_{L^2_\theta}}{\eta^{1/2}}.
\end{equation}
Next we apply the Minkowski and H{\"o}lder inequalities so that we can bound the latter term as follows
\begin{equation}\notag
||\bX'(t) - \bX'_0||_{L^2_\theta} = \bigg|\bigg| \int_0^t ds~ \partial_t \bX'(s) \bigg|\bigg|_{L^2_\theta} \leq t^{\frac12}\left(\int_0^t ds~ ||\partial_t \bX'(s)||_{L^2_\theta}^2\right)^{\frac12}.  
\end{equation}
Lastly we apply Corollary \ref{prop:time.estimate}, and then the result then follows so long as $T_{\epsilon}>0$ is taken sufficiently small.  
\end{proof}

We now point out that the argument in \cite[Prop 8.7 on page 337]{MR1867882} shows that for any two vectors $\BX_1$ and $\BX_2$ from \eqref{distance.X.notation} and \eqref{arc.cord.number} we have 
\begin{multline*}
       \left||\BX_1|_* - |\BX_2|_* \right|
       =
           \left|\inf_{\theta \ne \alpha} |\DAL \BX_1(\theta)| - \inf_{\theta \ne \alpha} |\DAL \BX_2(\theta)| \right|
           \\
           \leq
\sup_{\theta \ne \alpha}\left| \frac{|\delta_\alpha \BX_1(\theta)|}{|\alpha|} -  \frac{|\delta_\alpha \BX_2(\theta)|}{|\alpha|} \right|
           \leq
\sup_{\theta \ne \alpha} \frac{|\delta_\alpha (\BX_1-\BX_2)(\theta)|}{|\alpha|}. 
\end{multline*}
We thus conclude that
\begin{equation}\label{chord.arc.upper}
    \left||\BX_1|_* - |\BX_2|_* \right|
    \leq
    || \BX_1' - \BX_2'||_{L^\infty_\theta}.
\end{equation}
We can now deduce from Proposition \ref{prop:arc.chord.small} and \eqref{chord.arc.upper} that if initially $|\BX_0|_*>0$, then for a solution to \eqref{peskin.general.tension} satisfying \eqref{apriori.bd.norm} for any fixed $\rho$ satisfying $0<\rho<|\BX_0|_*$ there exists a small-time $T_\rho>0$ such that \eqref{apriori.bd} holds over $0 \le t \le T_\rho$.

\section{Strong continuity estimate}\label{sec:DifferenceBesovSpace}

In this section we will prove two a priori continuity estimates that will imply the uniqueness of solutions.  In \secref{sec:strongCont} we will prove the estimate that will establish the strong continuity result in Theorem \ref{main:unique}.   Then \secref{sec:l2continuity} we prove the estimates that will give the uniqueness in Theorem \ref{first:unique}.

\subsection{Strong continuity estimate}\label{sec:strongCont}
We consider two different solutions to \eqref{peskin.general.tension}, $\BX'(t,\theta)$ and $\BY'(t,\theta)$, with corresponding initial data $\bX_0$ and $\bY_0$ respectively. In this section we will sometimes use the notation $\bZ$ to denote either $\BX$ or $\BY$. When we use $\BZ$ in the estimates below it will not matter whether it is $\BX$ or $\BY$.  We consider initial data for \eqref{peskin.general.tension},  $\bZ_0$, satisfying
\begin{equation}\label{initial.assumption.Z}
    ||\bZ'_0||_{\BS}=||\bZ'_0||_{\dot{B}^{\frac12,\subw}_{2,1}} \leq M, \qquad |\bZ_0|_* = \inf_{\alpha\neq \theta} |\DAL \bZ_0(\theta)| >0.
\end{equation}
Here $0 < M < \infty$ is allowed to be large.  
  Then for some $\CTS >0$ we suppose for $T>0$ that for some $c>0$ and $\lambda>0$ as in \eqref{e:DTdefn} we have
\begin{equation}\label{apriori.bd.norm.Z}
    ||\bZ'||_{\CM}+    c \lambda^{\frac12}||\bZ'||_{\DM} \le \CTS M.
\end{equation}
We also prove our estimate in this section, for some $\rho>0$ that is allowed to be small, under the following condition
\begin{equation}\label{apriori.bd.Z}
 |\BZ(t)|_*\geq \rho, \quad 0 \le t \le T.
\end{equation}
Given $\subw$ from Definition \ref{subw.definition} we will use the equivalent semi-norm defined with $\nu$ instead of $\subw$ where $\nu$ is given by
\begin{equation}\label{nu.definition}
    \nu(r) \eqdef 1 + \frac{\subw(r)}{\KC} , \quad \KCC \ge 1.
\end{equation}
We will choose $\KCC = \KCC(\lambda^{-1}, \DTT)$ to be a possibly large constant at the end of the proof of Proposition \ref{prop:continuity}.   Notice that $\nu$ defines equivalent norms $\CN$ and $\DN$ to the norms $\CM$ and $\DM$ defined in \eqref{C.space.temporal} and \eqref{D.space.temporal} respectively.  In particular, from \eqref{nu.definition} we have 
\begin{equation}\label{equivalent.nu.norm}
|| f ||_{\CN} \leq    2 || f ||_{\CM}, \quad 
    || f ||_{\DN} \leq   2  || f ||_{\DM},
\end{equation}
and
\begin{equation}\notag
|| f ||_{\CM} \leq     \KC || f ||_{\CN},  \quad
        || f ||_{\DM} \leq     \KC || f ||_{\DN}. 
\end{equation}
Then with this equivalent norm we will prove the following continuity estimate.

\begin{proposition}\label{prop:continuity}
Let $\bX, \bY: [0,T]\times \T \to \R^2$ be two weak solutions to the Peskin problem with tension $\TE$ in the sense of Definition \ref{def:solution} with initial data $\BX_0,$ $\BY_0$ respectively.  Assume that $\BX_0,$ $\BY_0,$ and $\bT$ satisfy the assumptions of Theorem \ref{main:unique} including \eqref{initial.assumption.Z}.  Additionally, assume that \eqref{apriori.bd.norm.Z} and \eqref{apriori.bd.Z} hold.  For the tension map \eqref{tension.map.def} we assume that \eqref{e:QuantitativeTensionMap} and \eqref{tension.derivatives.continuity} hold.
Then for the two solutions $\BX'$ and $\BY'$ to \eqref{peskin.general.tension} over $0\le t \le T$ with $T>0$ we have 
\begin{equation*}
     ||\bX' -  \bY'||_{\CN}
    +
  \lambda^{\frac12} || \bX'-  \bY'||_{\DN}
    \leq   
    4 || \bX'_0 -  \bY'_0||_{\BN}
+C || \bX'-\bY'||_{\CN} T^{\frac12} \WK, 
\end{equation*}
where $\WK=\WK[\rho, M]$  is defined in \eqref{WK.constant.def}.

In particular, there exists $T_M = T_M(M, \rho, \mu, \lambda, \DTT, \DDTT, \DDDTT)>0$ such that for any $0<T\leq T_M,$ 
we have the following estimate
\begin{equation*}
     ||\bX' -  \bY'||_{\CN}
    +
  2\lambda^{\frac12} || \bX'-  \bY'||_{\DN}
    \leq   
    8 || \bX'_0 -  \bY'_0||_{\BN}.
\end{equation*}
\end{proposition}

For use below we define the following notation using \eqref{arc.cord.number}:  
\begin{equation}\label{arc.chord.XY2}
    |\BX,\BY|_* \eqdef \min\{|\BX|_*,|\BY|_*\}.
\end{equation}
Then the next two lemmas will be used in the proof of Proposition \ref{prop:continuity}.

\begin{lemma}\label{lemm.A.diff}
We have the following uniform estimate
\begin{multline} \notag
\left| \mathcal{A}[\bX]-\mathcal{A}[\bY] \right|
\lesssim
|\BX|_*^{-1}\left(\left| \DAPXY(\theta)\right| 
+
\left| \DAMXY(\theta)\right| 
\right)
\\
+
|\BX|_*^{-2}\left(\left| \DAPXY(\theta)\right| \left| \DAM(\theta)\right|
+
\left| \DAMXY(\theta)\right| \left| \DAYP(\theta)\right|
\right)
\\
+ 
|\BX,\BY|_*^{-2}\left| \DAL(\bX'-\bY')(\theta)\right|\left( \left| \DAYM(\theta)\right|
+\left| \DAYP(\theta)\right|
\right)
\\
+
|\BX,\BY|_*^{-3}\left| \DAL(\bX'-\bY')(\theta)\right| \left| \DAYM(\theta)\right|
\left| \DAYP(\theta)\right|.
\end{multline}
\end{lemma}

We also use the decomposition in \eqref{Abeta.split} as
\begin{equation}\notag
     \delta_\beta \mathcal{A}[\bX] =   \mathcal{A}_{1\beta}[\bX]  +  \mathcal{A}_{2\beta}[\bX], 
    \quad     
    \delta_\beta \mathcal{A}[\bY]  =   \mathcal{A}_{1\beta}[\bY]  +  \mathcal{A}_{2\beta}[\bY].
\end{equation}
We further introduce the following notation
\begin{equation}\label{notation.two.beta}
|\delta_\beta\DAX, \delta_\beta\DAY|
\eqdef
\max\{|\delta_\beta\DAX(\theta)|, | \delta_\beta\DAY(\theta)|\}.
\end{equation}

\begin{lemma}\label{lem:Abeta.upper} We have the uniform estimate for the difference
\begin{multline}\label{A1beta.upper}
\left| \mathcal{A}_{1\beta}[\bX] - \mathcal{A}_{1\beta}[\bY] \right|
\lesssim
\left( |\delta_\beta \DAPXY(\theta)|
+
|\delta_\beta \DAMXY(\theta)|
\right)|\BX|_{*}^{-1}
\\
+
\left(
|\delta_\beta \DAPXY(\theta)||\tau_\beta \DAM(\theta)|+ |\delta_\beta \delta_\alpha^- (\BX'-\BY')(\theta)| |\DAYP(\theta)| \right) |\BX|_{*}^{-2}
\\
+
|\delta_\beta \DAYP(\theta)| 
|\tau_\beta \DAMXY(\theta)|
|\BX|_{*}^{-2}
+
|\DAPXY(\theta)||\delta_\beta \DAM(\theta)| |\BX|_{*}^{-2}
\\
+
\left( |\delta_\beta \DAYP(\theta)|
+
|\delta_\beta \DAYM(\theta)|
\right)
|\tau_\beta \DAL(\bX'-\bY')(\theta)||\BX,\BY|_{*}^{-2}
\\
+
|\delta_\beta \DAYP(\theta)||\tau_\beta \DAYM(\theta)|
|\tau_\beta \DAL(\bX'-\bY')(\theta)||\BX,\BY|_{*}^{-3}
\\
+
|\DAYP(\theta)||\delta_\beta \DAYM(\theta)| 
|\DAL(\bX'-\bY')(\theta)| |\BX,\BY|_{*}^{-3}.
\end{multline}
And we have 
\begin{multline}\label{A2beta.upper}
\left| \mathcal{A}_{2\beta}[\bX] - \mathcal{A}_{2\beta}[\bY] \right|
\lesssim
|\DAPXY(\theta)|(1+|\tau_\beta \DAM(\theta)|+| \DAM(\theta)| ) \frac{|\delta_\beta \DAX(\theta)|}{|\BX|_{*}^{3}}
\\
+
\left(
|\tau_\beta \DAMXY(\theta)|
| \DAYP(\theta)| 
+
| \DAMXY(\theta)|(1+|\DAYP(\theta)|)
\right)\frac{|\delta_\beta \DAX(\theta)|}{|\BX|_{*}^{3}}
\\
+
(| \DAYP(\theta)|(1+| \DAYM(\theta)| +|\tau_\beta \DAYM(\theta)|)+| \DAYM(\theta)|)
\frac{|\delta_\beta \DAL(\bX'-\bY')(\theta)|}{|\BX,\BY|_{*}^{3}}
\\
+
| \DAYP(\theta)||\tau_\beta \DAYM(\theta)|
|\tau_\beta\DAL(\bX'-\bY')(\theta)|
\frac{|\delta_\beta\DAX, \delta_\beta\DAY|}{|\BX,\BY|_{*}^{4}}
\\
+
\left(
| \DAYP(\theta)||\tau_\beta \DAYM(\theta)|
+
|\DAYP(\theta)|| \DAYM(\theta)| 
\right)
|\DAL(\bX'-\bY')(\theta)| \frac{|\delta_\beta\DAX, \delta_\beta\DAY|}{|\BX,\BY|_{*}^{4}}
\\
+
\left(| \DAYP(\theta)|
+
| \DAYM(\theta)|
\right)
|\DAL(\bX'-\bY')(\theta)| \frac{|\delta_\beta\DAX, \delta_\beta\DAY|}{|\BX,\BY|_{*}^{4}}.
\end{multline}
\end{lemma}

The proofs of Lemmas \ref{lemm.A.diff} and \ref{lem:Abeta.upper} are contained in \secref{sec:kernelDIFF}.

\begin{proof}[Proof of Proposition \ref{prop:continuity}] For now we consider \eqref{peskin.general.tension}, and we take the difference of two solutions as
\begin{multline}\notag
\partial_t \BX'(\theta) -\partial_t \BY'(\theta) = \int_{\T} d\alpha \frac{\mathcal{K}[\BX](\theta, \alpha)}{\alpha^2}  \delta_\alpha \left(\bT(\BX'(\theta))-\bT(\BY'(\theta))\right)
\\
+
\int_{\T} d\alpha \frac{\mathcal{A}[\BX](\theta, \alpha)-\mathcal{A}[\BY](\theta, \alpha)}{\alpha^2}  \delta_\alpha \bT(\BY'(\theta)). 
\end{multline}
We now take $\delta_\beta$ of the equation above to obtain
\begin{multline}\notag
\partial_t \delta_\beta(\BX' - \BY')(\theta) = \int_{\T} d\alpha \frac{\tau_\beta \mathcal{K}[\BX](\theta, \alpha)}{\alpha^2}  \delta_\beta\delta_\alpha \left(\bT(\BX'(\theta))-\bT(\BY'(\theta))\right)
\\
+
\int_{\T} d\alpha \frac{\delta_\beta \mathcal{A}[\BX](\theta, \alpha)}{\alpha^2} 
\delta_\alpha(\bT(\BX'(\theta))-\bT(\BY'(\theta)))
\\
+
\int_{\T} d\alpha \frac{\tau_\beta(\mathcal{A}[\BX]-\mathcal{A}[\BY])(\theta, \alpha)}{\alpha^2}  \delta_\beta\delta_\alpha \bT(\BY'(\theta))
\\
+
\int_{\T} d\alpha \frac{\delta_\beta(\mathcal{A}[\BX]-\mathcal{A}[\BY])(\theta, \alpha)}{\alpha^2}  \delta_\alpha \bT(\BY'(\theta)).
\end{multline}
Now we consider this expression in $L^2$ similar to \eqref{beta.L2.difference.1} as
\begin{multline}\label{energy.diff}
\frac{d}{dt} || \delta_\beta \bX' - \delta_\beta \bY'||_{L^2}^2
    \\
    =
    - \int_{\T} d\theta \int_{\T} d\alpha ~
    \delta_\beta\delta_\alpha  (\bX'-  \bY')(\theta)  \cdot
    \frac{\tau_\beta \mathcal{K}[\bX](\theta, \alpha)}{\alpha^2}  
  \delta_\beta\delta_\alpha  (\bT(\BX'(\theta))-\bT(\BY'(\theta)))
    \\
    - \int_{\T} d\theta \int_{\T} d\alpha ~
    \delta_\beta \delta_\alpha  (\bX'-  \bY')(\theta) 
     \cdot\frac{\delta_\beta\mathcal{A}[\bX](\theta, \alpha)}{\alpha^2}  
    \delta_\alpha  (\bT(\BX'(\theta))-\bT(\BY'(\theta)))
        \\
    - \int_{\T} d\theta \int_{\T} d\alpha ~
       \delta_\beta\delta_\alpha  (\bX'-  \bY')(\theta) 
    \cdot \frac{ \tau_\beta (\mathcal{A}[\bX]-\mathcal{A}[\bY])(\theta, \alpha)}{\alpha^2}  
    \delta_\beta \delta_\alpha  \bT(\BY'(\theta))
    \\
    - \int_{\T} d\theta \int_{\T} d\alpha ~
           \delta_\beta\delta_\alpha  (\bX'-  \bY')(\theta) 
           \cdot 
    \frac{\delta_\beta(\mathcal{A}[\BX]-\mathcal{A}[\BY])(\theta, \alpha)}{\alpha^2}  \delta_\alpha \bT(\BY'(\theta)).
\end{multline}
To prove our estimate, we will first expand out each of the terms above.

To this end we recall  \eqref{e:deltaalphaT} and \eqref{e:albe.T.XY}.  As in \eqref{e:deltaalphaT} we expand out
\begin{multline}\label{difference.alpha.T}
    \delta_\alpha  (\bT(\BX'(\theta))-\bT(\BY'(\theta)))
=
  \DBTX  \delta_\alpha( \BX'-  \BY' )(\theta) \\ + (\DBTX - \DBTY)\delta_\alpha\bY'(\theta). 
\end{multline}
Further as in \eqref{DBTX.def} we calculate that 
\begin{equation}\label{difference.DBTXY}
    \DBTX - \DBTY
 =
\int_0^1 ds_1~ \DDTB(s_1,\theta,\alpha) ~
(g_1[\BX']-g_1[\BY'] ), 
\end{equation}
where $g_1[\BX']=g_1[\BX'](s_1, \alpha, \theta)$ is defined below \eqref{DBTX.def} and
\begin{equation}\label{DDBT.term}
\DDTB(s_1,\theta,\alpha) 
\eqdef
\int_0^1 ds_2~  D^2\bT
(f_2[\BX',\BY'](s_1,s_2,\theta,\alpha)).
\end{equation}
Here we also use the definition
\begin{multline}\notag
f_2[\BX',\BY'](s_1,s_2,\theta,\alpha)
\eqdef
s_2(s_1 \tau_\alpha ( \BX'-\BY')(\theta) + 
    (1-s_1)( \BX'-\BY')(\theta) )
    \\
    +s_1  \tau_\alpha \BY'(\theta) + 
     (1-s_1)\BY'(\theta).
\end{multline}
Thus recalling Remark \ref{ignore.remark} and using \eqref{e:QuantitativeTensionMap} we have 
\begin{multline*}
        |    \delta_\alpha  (\bT(\BX'(\theta))-\bT(\BY'(\theta)))  |
\leq
\DTT |\delta_\alpha  (\BX'-\BY')(\theta)|
\\
+\DDTT
 | (\BX'-\BY')(\theta)|
|\delta_\alpha  \BY'(\theta)|.
\end{multline*}
We will use this estimate for the second term in \eqref{energy.diff}.

 For the first term in \eqref{energy.diff}, we expand \eqref{e:albe.T.XY} out as
\begin{multline*}
\delta_\beta\delta_\alpha  (\bT(\BX'(\theta))-\bT(\BY'(\theta)))
=
  \tau_\beta \DBTX  \delta_\beta\delta_\alpha( \BX'-  \BY' )(\theta) 
  \\ 
  + \tau_\beta (\DBTX - \DBTY)\delta_\beta\delta_\alpha\bY'(\theta) 
      + 
    \delta_\beta \DBTX \delta_\alpha(\BX'-\BY')(\theta)
    \\
        + 
    (\delta_\beta \DBTX-\delta_\beta \DBTY )\delta_\alpha \BY'(\theta).
\end{multline*}
Notice that $\delta_\beta \DBTX$ is calculated in \eqref{delta.beta.DBTX} and it has the bound \eqref{delta.beta.DBTX.bound}.  We further calculate using $g_1$ from \eqref{DBTX.def} and  \eqref{delta.beta.DBTX} that 
\begin{multline}\label{delta.beta.DBTX.minus}
\delta_\beta\DBTX - \delta_\beta\DBTY
=
\int_0^1 ds_1 ~
\DDBT[\BX'](\theta) g_1[\delta_\beta(\BX'-\BY')](s_1, \alpha, \theta)
\\    
    +
    \int_0^1 ds_1 ~
    (\DDBT[\BX']- \DDBT[\BY']) g_1[\delta_\beta\BY'](s_1, \alpha, \theta),
\end{multline}
where 
\begin{equation*}
    \DDBT[\BX']- \DDBT[\BY']
    =
   \int_0^1 ds_2~ \DDDBT(\theta,\alpha,\beta) g_2[\BX'-\BY'],
\end{equation*}
where $g_2[\BX'-\BY']$ is defined in \eqref{g2.s2.def}.  We further use
\begin{equation*}
\DDDBT(\theta,\alpha,\beta) 
\eqdef
 \int_0^1 ds_3~  D^3\bT
(g_3(s_3,\theta,\alpha,\beta)),
\end{equation*}
and with $g_2[\BX']$ defined in \eqref{g2.s2.def} we use
\begin{equation*}
g_3(s_3,\theta,\alpha,\beta)
\eqdef
s_3 g_2[\BX'](s_2,s_1,\alpha,\theta,\beta) + (1-s_3)g_2[\BY'](s_2,s_1,\alpha,\theta,\beta).
\end{equation*}
Notice that $  \tau_\beta \DBTX  \delta_\beta\delta_\alpha( \BX'-  \BY' )(\theta)$ with \eqref{DBTX.def} will give rise to the crucial elliptic term in \eqref{energy.diff} using \eqref{e:DTdefn}.  Putting all of this together including Remark \eqref{ignore.remark} using \eqref{delta.beta.DBTX.bound}, \eqref{e:QuantitativeTensionMap} and \eqref{tension.derivatives.continuity} we conclude the following bound
\begin{multline*}
|\delta_\beta\delta_\alpha  (\bT(\BX'(\theta))-\bT(\BY'(\theta)))
-
  \tau_\beta \DBTX  \delta_\beta\delta_\alpha( \BX'-  \BY' )(\theta) |
  \\ 
  \leq
  \DDTT(
| \delta_\beta\delta_\alpha\bY'(\theta) |
 |(\BX'-\BY')(\theta)|
      + 
    |\delta_\beta\BX'(\theta)|
    |\delta_\alpha(\BX'-\BY')(\theta)|)
    \\
        + 
  \DDTT  
    |\delta_\beta( \BX' - \BY')(\theta)| |\delta_\alpha \BY'(\theta)|
        + 
  \DDDTT
    |(\BX'-\BY')(\theta)|
    |\delta_\beta \BY'(\theta)| |\delta_\alpha \BY'(\theta)|.
\end{multline*}
 We will use this to estimate the first term in \eqref{energy.diff}. 
 To bound the third term in \eqref{energy.diff}, we recall  \eqref{e:albe.T.XY} and \eqref{beta.alpha.TXP}.
Lastly, to bound fourth term in \eqref{energy.diff} we recall \eqref{e:deltaalphaT} and \eqref{delta.alpha.BX.bound}.

Plugging all of these calculations into \eqref{energy.diff}, and using  \eqref{kerbel.eqn.deriv} with \eqref{kerbel.A.eqn.deriv} and \eqref{e:DTdefn} and Remark \eqref{ignore.remark} we obtain 
\begin{multline}\label{energy.difference.first1}
    \frac{d}{dt} || \delta_\beta (\bX' -  \bY')||_{L^2}^2
    +
 \lambda\int_{\T} d\theta \int_{\T} \frac{d\alpha}{\alpha^2} ~
       |\delta_\beta\delta_\alpha  (\bX'-  \bY')(\theta)|^2 
    \\
    \leq   \sum_{j=0}^{3} \TT_j
    +  \sum_{j=4}^{7} \TT_j
        +  \sum_{j=8}^{9} \TT_j.
\end{multline}
To ease the notation, when we list the terms below we will drop the $(\theta)$ and $(\theta, \alpha)$ notation from each term.  For example we will write $\tau_\beta \mathcal{A}[\bX](\theta, \alpha) = \tau_\beta \mathcal{A}[\bX]$.  Then with \eqref{kerbel.eqn.deriv} and  \eqref{kerbel.A.eqn.deriv} we have 
\begin{equation}\notag
    \TT_0 \eqdef  \DTT \int_{\T} d\theta \int_{\T} \frac{d\alpha}{\alpha^2} ~
   | \delta_\beta\delta_\alpha  (\bX'-  \bY')|^2  |\tau_\beta \mathcal{A}[\bX]|, 
\end{equation}
\begin{equation}\notag
    \TT_1 \eqdef  \DTT \int_{\T} d\theta \int_{\T} \frac{d\alpha}{\alpha^2} ~
   | \delta_\beta\delta_\alpha  (\bX'-  \bY')| | \delta_\alpha  (\bX'-  \bY')|  |\delta_\beta \mathcal{A}[\bX]|, 
\end{equation}
\begin{equation}\notag
    \TT_2 \eqdef  \DTT \int_{\T} d\theta \int_{\T} \frac{d\alpha}{\alpha^2} ~
   | \delta_\beta\delta_\alpha  (\bX'-  \bY')|  | \delta_\beta\delta_\alpha  \bY'|  |\tau_\beta (\mathcal{A}[\BX]-\mathcal{A}[\BY])|, 
\end{equation}
\begin{equation}\notag
    \TT_3 \eqdef  \DTT \int_{\T} d\theta \int_{\T} \frac{d\alpha}{\alpha^2} ~
   | \delta_\beta\delta_\alpha  (\bX'-  \bY')|  | \delta_\alpha  \bY'|  |\delta_\beta(\mathcal{A}[\BX]-\mathcal{A}[\BY])|, 
\end{equation}
\begin{equation}\notag
    \TT_4 \eqdef  \DDTT  \int_{\T} d\theta \int_{\T} \frac{d\alpha}{\alpha^2} ~
   | \delta_\beta\delta_\alpha  (\bX'-  \bY')|  |\tau_\beta \mathcal{K}[\bX]|
   | \bX'-  \bY'| | \delta_\beta \delta_\alpha  \bY'|, 
\end{equation}
\begin{equation}\notag
    \TT_5 \eqdef \DDTT \int_{\T} d\theta \int_{\T} \frac{d\alpha}{\alpha^2} ~
   | \delta_\beta\delta_\alpha  (\bX'-  \bY')|  |\tau_\beta \mathcal{K}[\bX]|
   |  \delta_\beta  \bX'| | \delta_\alpha  (\bX'-  \bY')|, 
\end{equation}
\begin{equation}\notag
    \TT_{6} \eqdef  \DDTT  \int_{\T} d\theta \int_{\T} \frac{d\alpha}{\alpha^2} ~
   | \delta_\beta\delta_\alpha  (\bX'-  \bY')|  |\tau_\beta \mathcal{K}[\bX]|
      | \delta_\beta  (\bX'-  \bY')| 
|  \delta_\alpha  \bY'|,
\end{equation}
\begin{equation}\notag
    \TT_{7} \eqdef  \DDTT  \int_{\T} d\theta \int_{\T} \frac{d\alpha}{\alpha^2} ~
   | \delta_\beta\delta_\alpha  (\bX'-  \bY')|  |\delta_\beta \mathcal{A}[\bX]|
  |\bX'-  \bY'| |\delta_\alpha \bY'|, 
\end{equation}
\begin{equation}\notag
    \TT_{8} \eqdef  \DDTT \int_{\T} d\theta \int_{\T} \frac{d\alpha}{\alpha^2} ~
   | \delta_\beta\delta_\alpha  (\bX'-  \bY')|   |\tau_\beta (\mathcal{A}[\BX]-\mathcal{A}[\BY])|
    |\delta_\beta \bY'| |\delta_\alpha \bY'|, 
\end{equation}
\begin{equation}\notag
    \TT_{9} \eqdef  \DDDTT  \int_{\T} d\theta \int_{\T} \frac{d\alpha}{\alpha^2} ~
   | \delta_\beta\delta_\alpha  (\bX'-  \bY')| 
   |\tau_\beta \mathcal{K}[\bX]|
      |\bX'-  \bY'|
    |\delta_\beta \BY'| |\delta_\alpha \BY'|. 
\end{equation}
We will estimate each of the terms above individually.  In the all of the following estimates we will use a small $\eta>0$ to be chosen at the end of the proof.

First notice that $\TT_{0}$ is analogous to $\LLT_{1}$ from \eqref{beta.L2.difference} and \eqref{e:Abounds}.  Thus similar to \eqref{LA.estimate} we have 
\begin{equation}\label{TT1.estimate}
     \TT_{0}
     \leq
\left( \LDS +     C \kappa_1  \DTT  \right) ||\delta_\beta \TLam^{\frac12}(\bX'-\bY')||_{L^2_\theta}^2 + C ||\delta_\beta (\bX'-\bY')||_{L^2_\theta}^2 \BB,
\end{equation}
where $\kappa_1 = \kappa_1(\eta)$ is given by \eqref{kappa1.def} and $\BB$ is defined in \eqref{BB.const.def}.   Later we will be able to choose $\eta>0$ small enough in  so that  we have  $C \kappa_1   \DTT \leq \frac{1}{2}\LDS$.  This is our main estimate for the term $\TT_{0}$.

Then the term $\TT_{1}$ is exactly $\LLT_2[\bX'-\bY', \bX', \bX'-\bY']$ from \eqref{LLT2.generalform}.  Thus as in \eqref{LLT2.bound.general}, $\TT_{1}$ satisfies the estimate 
\begin{multline}\label{TT2.bound}
\TT_{1}
\leq 
\LDS ||\delta_\beta \TLam^{\frac12}(\bX'-\bY')||_{L^2_\theta}^2
+     
C\kappa_3\DTT^2 
||   \delta_\beta \BX'||_{L^\infty_\theta}^2 ||  \bX'-\bY'||_{\BA}^2
\\
+
 \CDS
|| \delta_\beta(\bX'-\bY')||_{L^2_\theta}^2
+C || \delta_\beta\BX'||_{L^2_\theta}^2
||  \bX'-\bY'||_{L^\infty_\theta}^2 \DTT^2 
\WK_2,
\end{multline}
where $\kappa_3 = \kappa_3(\eta)$ is given by \eqref{kappa3.def} 
and as in  \eqref{BG.const.def} we have  
\begin{equation}\label{WK2.constant.def}
\WK_2 = \WK_2[M,\rho^{-1},\eta^{-1}]
    \eqdef 
 \eta^{-2}    
\left( \sum_{j=1}^2 \rho^{-j} M^{j-1}\right)^2(1+\rho^{-2} M^2).
\end{equation}
This is our main estimate for the term $\TT_{1}$.

Next we consider $\TT_{7}$.  After bounding $|\bX'-  \bY'|  \lesssim ||\bX'-  \bY'||_{L^\infty_\theta}$, then as in \eqref{LLT2.bound.general} with \eqref{apriori.bd.norm.Z} the term $\TT_{7}$ satisfies the bounds 
\begin{multline}\label{TT11.bound}
\TT_{7}
\leq 
\LDS ||\delta_\beta \TLam^{\frac12}(\bX'-\bY')||_{L^2_\theta}^2
+     
C\kappa_3\DDTT^2 M^2 || \bX'-\bY'||_{L^\infty_\theta}^2 
||   \delta_\beta \BX'||_{L^\infty_\theta}^2
\\
+
 \CDS
|| \delta_\beta(\bX'-\bY')||_{L^2_\theta}^2 
+ 
C
|| \delta_\beta\BX'||_{L^2_\theta}^2 
|| \bX'-\bY'||_{L^\infty_\theta}^2 M^2\DDTT^2 \WK_2.
\end{multline}
This is our main estimate for $\TT_{7}$.

Next, we apply the estimate \eqref{TTG.y.estimate} to $\TT_{5}$ to obtain
\begin{multline}\label{TT7.bound}
        \TT_{5} 
        \leq
      \LDS  ||   \delta_\beta \TLam^{\frac12}(\bX'-  \bY')||_{L^2_\theta}^2
      +
C\kappa_5 \DDTT^2||   \delta_\beta \bX'||_{L^\infty_\theta}^2 ||\bX'-  \bY'||_{\BA}^2
\\
+
 \CDS
|| \delta_\beta(\bX'-  \bY')||_{L^2_\theta}^2+  C || \delta_\beta\bX'||_{L^2_\theta}^2 
||\bX'-  \bY'||_{L^\infty_\theta}^2 \DDTT^2
\WK_5.
\end{multline}
where recalling $\WK_1=\WK_1[\rho, M]=C(1+ \rho^{-2} M^2)$ from \eqref{K.bound.infty} we define
\begin{equation}\label{kappa5.def}
    \kappa_5 \eqdef \frac{(1+ \rho^{-2} M^2)^2}{\lambda \MA(\eta^{-1})^2},
\end{equation}
and 
\begin{equation}\label{WK5.constant.def}
    \WK_5 \eqdef  \eta^{-2}
  (1+ \rho^{-2} M^2)^2.
\end{equation}
Next, we apply the estimate \eqref{TTG.y.estimate} to $\TT_{6}$ to obtain
\begin{multline}\label{TT9.bound}
        \TT_{6}
        \leq
      \LDS  ||   \delta_\beta \TLam^{\frac12}(\bX'-  \bY')||_{L^2_\theta}^2
      +
C\kappa_6 \DDTT^2 ||   \delta_\beta (\bX'-  \bY')||_{L^\infty_\theta}^2
\\
+
 C
|| \delta_\beta(\bX'-  \bY')||_{L^2_\theta}^2
\DDTT 
 M (1+ \rho^{-2} M^2) \eta^{-1},
\end{multline}
where
\begin{equation}\label{kappa6.def}
    \kappa_6 = \frac{M^2}{\lambda \MA(\eta^{-1})^2} (1+ \rho^{-2} M^2)^2. 
\end{equation}
We also apply the estimate \eqref{TTG.y.estimate} to $\TT_{9}$ to obtain
\begin{multline}\label{TT18.bound}
        \TT_{9} 
        \leq
      \LDS  ||   \delta_\beta \TLam^{\frac12}(\bX'-  \bY')||_{L^2_\theta}^2
      +
C\kappa_6 \DDDTT^2 ||\bX'-  \bY'||_{L^\infty_\theta}^2 ||   \delta_\beta \bY'||_{L^\infty_\theta}^2
\\
+
\CDS
|| \delta_\beta(\bX'-  \bY')||_{L^2_\theta}^2
\\
+C || \delta_\beta\bY'||_{L^2_\theta}^2 
||\bX'-  \bY'||_{L^\infty_\theta}^2
\DDDTT^2 
 M^2 (1+ \rho^{-2} M^2)^2 \eta^{-2}.
\end{multline}
This is our main estimate for $\TT_{9}$.

We will now estimate the terms $\TT_{2}$ and $\TT_{8}$.  From Lemma \ref{lemm.A.diff}, Remark \ref{ignore.remark} and \eqref{apriori.bd.Z} we have 
\begin{multline} \label{A.diff.est.noPM}
\left| \mathcal{A}[\bX]-\mathcal{A}[\bY] \right|
\lesssim
| \delta_\alpha (\BX'-\BY')(\theta)|\left( \rho^{-1}+
\rho^{-2}(| \delta_\alpha \BX'(\theta)|+| \delta_\alpha \BY'(\theta)|)
\right)
\\
+ 
\left| \DAL (\BX - \BY)\right|  | \delta_\alpha \BY'(\theta)|
\left(\rho^{-2}+ \rho^{-3}|\delta_\alpha \BY'(\theta)|\right).
\end{multline}
We thus define
\begin{equation}\label{WK21.constant.def}
    \WK_{21}= \WK_{21}[\rho,M] \eqdef  
    \rho^{-1}+
\rho^{-2}M, \quad   \WK_{22}= \rho^{-1}\WK_{21}.
\end{equation}
Then for $\TT_{2}$ using \eqref{operator.bd.first} we split it up as
\begin{multline*}
        \TT_{2} \leq  C    \WK_{21}\DTT \int_{\T} d\theta \int_{\T} \frac{d\alpha}{\alpha^2} ~
   | \delta_\beta\delta_\alpha  (\bX'-  \bY')(\theta)|  | \delta_\beta  \bY'(\theta)|  | \delta_\alpha (\BX'-\BY')(\theta)|
   \\
   +
   C    \WK_{22}\DTT || \BX'-\BY'||_{L^\infty_{\theta}}\int_{\T} d\theta \int_{\T} \frac{d\alpha}{\alpha^2} ~
   | \delta_\beta\delta_\alpha  (\bX'-  \bY')(\theta)|  | \delta_\beta  \bY'(\theta)|  | \delta_\alpha \BY'(\theta)|
   \\
   = \TT_{21}+\TT_{22}.
\end{multline*}
Then for $\TT_{21}$ we use \eqref{TTG.y.estimate} to obtain
\begin{multline}\label{TT31.estimate}
\TT_{21}
\leq
\LDS ||   \delta_\beta \TLam^{\frac12}(\bX'-  \bY')||_{L^2_\theta}^2
+
\frac{C \WK_{21}^2\DTT^2}{\lambda \MA(\eta^{-1})^2} ||   \delta_\beta \BY'||_{L^\infty_\theta}^2
|| \bX'-  \bY'||_{\BA}^2
\\
+
\CDS || \delta_\beta(\bX'-  \bY')||_{L^2_\theta}^2 
+
C|| \delta_\beta\BY'||_{L^2_\theta}^2
||\bX'-  \bY'||_{L^{\infty}_\theta}^2 
\WK_{21}^2\DTT^2 \eta^{-2}.
\end{multline}
And for $\TT_{22}$ we similarly obtain
\begin{multline}\label{TT32.estimate}
\TT_{22}
\leq
\LDS ||   \delta_\beta \TLam^{\frac12}(\bX'-  \bY')||_{L^2_\theta}^2
+
\frac{C \WK_{22}^2\DTT^2 M^2}{\lambda \MA(\eta^{-1})^2} ||   \delta_\beta \BY'||_{L^\infty_\theta}^2 || \BX'-\BY'||_{L^\infty_{\theta}}^2
\\
+
\CDS
|| \delta_\beta(\bX'-  \bY')||_{L^2_\theta}^2 
+
C || \delta_\beta\BY'||_{L^2_\theta}^2|| \BX'-\BY'||_{L^\infty_{\theta}}^2
M^2   \WK_{22}^2\DTT^2 \eta^{-2}.
\end{multline}
These are our main estimates for $\TT_{2}$.

For the term $\TT_{8}$, also using \eqref{operator.bd.first}, with \eqref{apriori.bd.norm.Z} and Remark \ref{ignore.remark} we will use the following estimate
\begin{equation}\notag
    \left| \mathcal{A}[\bX]-\mathcal{A}[\bY] \right|
\leq
C\WK_{8}
|| \BX'-\BY'||_{L^\infty_{\theta}},
\end{equation}
where
\begin{equation}\label{WK8.constant.def}
    \WK_{8}= \WK_{8}[\rho,M] \eqdef  
    \rho^{-1}+
\rho^{-2}M
+ 
\rho^{-3}M^2.
\end{equation}
Then for $\TT_{8}$ we further use \eqref{TTG.y.estimate} to find
\begin{multline}\label{TT13.estimate}
\TT_{8}
\leq
\LDS ||   \delta_\beta \TLam^{\frac12}(\bX'-  \bY')||_{L^2_\theta}^2
+
C \frac{ \WK_{8}^2 M^2 \DDTT^2
}{\lambda \MA(\eta^{-1})^2} ||   \delta_\beta \BY'||_{L^\infty_\theta}^2
|| \BX'-\BY'||_{L^\infty_{\theta}}^2
\\
+
\CDS
|| \delta_\beta(\bX'-  \bY')||_{L^2_\theta}^2 
+ 
C
|| \delta_\beta\BY'||_{L^2_\theta}^2 || \BX'-\BY'||_{L^\infty_{\theta}}^2 M^2   \WK_{8}^2
 \eta^{-2} \DDTT^2.
\end{multline}
This is our main estimate for $\TT_{8}$.

We will now estimate $\TT_{3}$.  To this end we use the decomposition in \eqref{Abeta.split} as
\begin{equation}\notag
     \delta_\beta \mathcal{A}[\bX] =   \mathcal{A}_{1\beta}[\bX]  +  \mathcal{A}_{2\beta}[\bX], 
    \quad     
    \delta_\beta \mathcal{A}[\bY]  =   \mathcal{A}_{1\beta}[\bY]  +  \mathcal{A}_{2\beta}[\bY]. 
\end{equation}
Then from \eqref{A1beta.upper} with Remark \ref{ignore.remark} and \eqref{apriori.bd.Z} we have the following bound
\begin{multline}\notag
\left| \mathcal{A}_{1\beta}[\bX] - \mathcal{A}_{1\beta}[\bY] \right|
\lesssim
 |\delta_\beta \delta_\alpha (\BX'-\BY')(\theta)|\rho^{-1}
\\
+
|\delta_\alpha (\BX'-\BY')(\theta)|
\left(
|\delta_\beta \delta_\alpha \BY'(\theta)| 
+
|\delta_\beta \delta_\alpha \BX'(\theta)| 
\right)
\rho^{-2}
\\
+
|\delta_\beta \delta_\alpha (\BX'-\BY')(\theta)| 
\left(
| \delta_\alpha \BY'(\theta)| 
+
|\delta_\alpha \BX'(\theta)| 
\right)
\rho^{-2}
\\
+
|\delta_\beta \delta_\alpha \BY'(\theta)|
|\DAL(\BX - \BY)(\theta)|
\left(1+ \rho^{-1}|\delta_\alpha \BY'(\theta)|\right)\rho^{-2}.
\end{multline}
And from \eqref{A2beta.upper} we have
\begin{multline}\notag
\left| \mathcal{A}_{2\beta}[\bX] - \mathcal{A}_{2\beta}[\bY] \right|
\lesssim
\rho^{-3}
|\delta_\alpha (\BX'-\BY')(\theta)|
|\delta_\beta \DAX (\theta)|
\\
+
\rho^{-3}\left(
|\delta_\alpha \BY'(\theta)|
+
|\delta_\alpha \BX'(\theta)|
\right)
|\delta_\alpha (\BX'-\BY')(\theta)|
|\delta_\beta \DAX (\theta)|
\\
+
\rho^{-3}\left(|\delta_\alpha \BY'(\theta)|+|\delta_\alpha \BY'(\theta)|^2 \right)
|\delta_\beta \DAL(\BX - \BY)(\theta)|
\\
+
\rho^{-4}\left(|\delta_\alpha \BY'(\theta)|+|\delta_\alpha \BY'(\theta)|^2 \right)
|\DAL(\BX - \BY)(\theta)| (|\delta_\beta\DAX (\theta)|+|\delta_\beta\DAY (\theta)|).
\end{multline}
We thus define $\WK_{3}= \WK_{3}[\rho,M]$ by
\begin{equation}\label{WK3.constant.def}
        \WK_{3} \eqdef  
    \rho^{-1}+
\rho^{-2}\left(1+M\right)\left(1+\rho^{-1}(1+M)\right)
+
\rho^{-4}M\left(1+M \right).
\end{equation}
And then we plug these estimates in, using also \eqref{operator.bd.first}, to observe 
\begin{multline}\notag
 \TT_{3}
 \leq
     C  \DTT \WK_{3} \int_{\T} d\theta \int_{\T} \frac{d\alpha}{\alpha^2} ~
   | \delta_\beta\delta_\alpha  (\bX'-  \bY')(\theta)| |\delta_\alpha (\BX'-\BY')(\theta)| |\delta_\beta \BZ'(\theta)|
      \\
      +
  C  \DTT \WK_{3} \int_{\T} d\theta \int_{\T} \frac{d\alpha}{\alpha^2} ~
   | \delta_\beta\delta_\alpha  (\bX'-  \bY')(\theta)| |\delta_\alpha (\BX'-\BY')(\theta)| |\delta_\beta \DAL\BZ(\theta)|
      \\
   +
  C  \DTT \WK_{3} \int_{\T} d\theta \int_{\T} \frac{d\alpha}{\alpha^2} ~
   | \delta_\beta\delta_\alpha  (\bX'-  \bY')(\theta)|| \delta_\beta  (\bX'-  \bY')(\theta)| | \delta_\alpha  \bZ'(\theta)| 
      \\
   +
  C  \DTT \WK_{3} \int_{\T} d\theta \int_{\T} \frac{d\alpha}{\alpha^2} ~
   | \delta_\beta\delta_\alpha  (\bX'-  \bY')(\theta)| |\delta_\beta \DAL(\BX - \BY)(\theta)| | \delta_\alpha  \bZ'(\theta)| 
               \\
   +
  C  \DTT \WK_{3}  || \BX'- \BY'||_{L^\infty_\theta }\int_{\T} d\theta \int_{\T} \frac{d\alpha}{\alpha^2} ~
   | \delta_\beta\delta_\alpha  (\bX'-  \bY')(\theta)|  | \delta_\alpha  \bY'(\theta)|     |\delta_\beta \BZ'(\theta)|
         \\
   +
  C  \DTT \WK_{3}|| \BX'- \BY'||_{L^\infty_\theta }\int_{\T} d\theta \int_{\T} \frac{d\alpha}{\alpha^2} ~
   | \delta_\beta\delta_\alpha  (\bX'-  \bY')(\theta)|  | \delta_\alpha  \bY'(\theta)|     |\delta_\beta \DAL\BZ(\theta)|
   \\
   =\sum_{j=1}^6  \TT_{3j}.
\end{multline}
Here we recall the notation $\BZ'$ defined above \eqref{initial.assumption.Z}.

Now for the term $\TT_{31}$ we use  \eqref{TTG.y.estimate} to get
\begin{multline}\label{TT41.estimate}
     \TT_{31}
\leq
    \LDS  ||   \delta_\beta \TLam^{\frac12}(\bX'-  \bY')||_{L^2_\theta}^2
    +
\frac{C \DTT^2 \WK_{3}^2}{\lambda \MA(\eta^{-1})^2}
|| \BX'-\BY'||_{\BA}^2
||\delta_\beta \BZ'||_{L^\infty_\theta }^2
\\
+
    \CDS || \delta_\beta(\bX'-  \bY')||_{L^2_\theta}^2
    +
 C||\delta_\beta \BZ'||_{L^2_\theta }^2 
 ||\bX'-  \bY'||_{L^\infty_\theta}^2
 \DTT^2 \WK_{3}^2 \eta^{-2}.
\end{multline}
Then because of \eqref{averaging.infinity.est}, in $\TT_{32}$ we can treat $|\delta_\beta \DAL\BZ(\theta)|$ the same as $|\delta_\beta \BZ'(\theta)|$ in $\TT_{31}$.  Thus $\TT_{32}$ also satisfies \eqref{TT41.estimate}.

Next for the term $\TT_{33}$ we again use  \eqref{TTG.y.estimate} and \eqref{apriori.bd.norm.Z} to obtain
\begin{multline}\label{TT45.estimate}
     \TT_{33}
\leq
    \LDS  ||   \delta_\beta \TLam^{\frac12}(\bX'-  \bY')||_{L^2_\theta}^2
    +
\frac{C \DTT^2 \WK_{3}^2}{\lambda \MA(\eta^{-1})^2}
M^2
||\delta_\beta (\BX'-\BY')||_{L^\infty_\theta }^2
\\
    +
   C || \delta_\beta(\bX'-  \bY')||_{L^2_\theta}^2
  \DTT \WK_{3} \eta^{-1}
 M.
\end{multline}
Then again because of \eqref{averaging.infinity.est} $\TT_{34}$ also satisfies \eqref{TT45.estimate}.

For the term $\TT_{35}$ we use  \eqref{TTG.y.estimate} to obtain
\begin{multline}\label{TT47.estimate}
     \TT_{35}
\leq
    \LDS  ||   \delta_\beta \TLam^{\frac12}(\bX'-  \bY')||_{L^2_\theta}^2
    +
\frac{C \DTT^2 \WK_{3}^2}{\lambda \MA(\eta^{-1})^2}
M^2
||\delta_\beta \BZ'||_{L^\infty_\theta }^2
||\BX'-\BY'||_{L^\infty_\theta }^2
\\
+
    \CDS || \delta_\beta(\bX'-  \bY')||_{L^2_\theta}^2
    +
   C || \delta_\beta \bZ'||_{L^2_\theta}^2
||\BX'-\BY'||_{L^\infty_\theta }^2
  \DTT^2 \WK_{3}^2 \eta^{-2}
M^2.
\end{multline}
And again with \eqref{averaging.infinity.est} then $\TT_{36}$ also satisfies \eqref{TT47.estimate}.   These are our main estimates for $\TT_{3}$.

The last term to estimate  is  $\TT_{4}$.  From \eqref{K.bound.infty} we can bound
\begin{equation*}
            \TT_{4}
        \leq
       \WK_1 \DDTT  || \bX'-  \bY'||_{L^\infty_\theta} \int_{\T} d\theta \int_{\T} \frac{d\alpha}{\alpha^2} ~
   | \delta_\beta\delta_\alpha  (\bX'-  \bY')(\theta)|  
   | \delta_\beta \delta_\alpha  \bY'(\theta)|.
\end{equation*}
For the term $\TT_{4}$  we apply Cauchy-Schwartz to obtain 
\begin{equation*}
            \TT_{4}
\leq
 \WK_1 \DDTT  || \bX'-  \bY'||_{L^\infty_\theta} 
||   \delta_\beta \TLam^{\frac12}(\bX'-  \bY')||_{L^2_\theta}
||   \delta_\beta \TLam^{\frac12}\bY'||_{L^2_\theta}.
\end{equation*}
Notice that this term does not have the same opportunity to achieve an extra smallness using the regularity from Definition \ref{subw.definition} similar to the other terms, as in \eqref{TTG.y.estimate}.  Thus the presence of the term $\TT_{4}$ is the reason why we use the equivalent norm with \eqref{nu.definition}.  For now we apply Young's inequality 
\begin{equation}\label{TT61.estimate}
            \TT_{4}
\leq
\LDS ||   \delta_\beta \TLam^{\frac12}(\bX'-  \bY')||_{L^2_\theta}^2
+
C \lambda^{-1} \WK_1^2 \DDTT^2   || \bX'-  \bY'||_{L^\infty_\theta}^2 
||   \delta_\beta \TLam^{\frac12}\bY'||_{L^2_\theta}^2.
\end{equation}
This completes our individual estimates for all of the terms in \eqref{energy.difference.first1}.

Next we collect all the estimates above in \eqref{TT1.estimate}, \eqref{TT2.bound}, \eqref{TT11.bound},  \eqref{TT7.bound}, \eqref{TT9.bound},  \eqref{TT18.bound}, \eqref{TT31.estimate}, \eqref{TT32.estimate}, \eqref{TT13.estimate}, \eqref{TT41.estimate},  \eqref{TT45.estimate}, \eqref{TT47.estimate} and \eqref{TT61.estimate} and put them into \eqref{energy.difference.first1} to obtain
\begin{multline}\notag
    \frac{d}{dt} || \delta_\beta (\bX' -  \bY')||_{L^2_\theta}^2
    +
  \frac{3\lambda}{4} || \delta_\beta \TLam^{\frac12} (\bX'-  \bY')||_{L^2_\theta}^2
    \\
    \leq   
    C\kappa_{7} \left(||   \delta_\beta \BX'||_{L^\infty_\theta}^2+||   \delta_\beta \BY'||_{L^\infty_\theta}^2 \right)||  \bX'-\bY'||_{\BA}^2
    +
C|| \partial_{\beta}( \bX'-\bY')||_{L^2_\theta}^2 \WK_{9}
\\
+
C\kappa_{9} || \partial_{\beta}( \bX'-\bY')||_{L^\infty_\theta}^2
+
C\left(||   \delta_\beta \BX'||_{L^2_\theta}^2+||   \delta_\beta \BY'||_{L^2_\theta}^2 \right)
|| \bX'-\bY'||_{L^\infty_\theta}^2
    \WK_{7}
\\
+     C \kappa_1  \DTT  ||\delta_\beta \TLam^{\frac12}(\bX'-\bY')||_{L^2_\theta}^2
+
C \lambda^{-1} \DDTT^2  \WK_1^2
||   \delta_\beta \TLam^{\frac12} \bY'||_{L^2_\theta}^2
|| \bX'-  \bY'||_{L^\infty_\theta}^2.
\end{multline}
Here we recall \eqref{kappa1.def}.  Further recalling \eqref{kappa3.def}, \eqref{kappa5.def}, \eqref{kappa6.def}, \eqref{WK21.constant.def},  \eqref{WK8.constant.def} and \eqref{WK3.constant.def}, we define
\begin{multline}\label{kappa7.def}
        \kappa_{7}
    \eqdef \kappa_3\DTT^2 
    +
    \kappa_3\DDTT^2 M^2 
    +
    \kappa_5 \DDTT^2
    +
    \kappa_6 \DDDTT^2
    +
    \frac{ \WK_{21}^2\DTT^2}{\lambda \MA(\eta^{-1})^2}
    +
    \frac{ \WK_{22}^2\DTT^2 M^2}{\lambda \MA(\eta^{-1})^2}
     \\
     +
     \frac{ \WK_{8}^2 M^2 \DDTT^2
}{\lambda \MA(\eta^{-1})^2}
+
 \frac{ \DTT^2 \WK_{3}^2}{\lambda \MA(\eta^{-1})^2} (1+M^2),
\end{multline}
and additionally recalling \eqref{WK2.constant.def} and \eqref{WK5.constant.def} we have 
\begin{multline}\label{WK7.constant.def}
        \WK_{7}
    \eqdef \DTT^2 
\WK_2
+
M^2\DDTT^2 \WK_2
+
\DDTT^2
\WK_5
+
\DDDTT^2 
 M^2 (1+ \rho^{-2} M^2)^2 \eta^{-2}
 +
 \WK_{21}^2\DTT^2 \eta^{-2}
\\
+
M^2   \WK_{22}^2\DTT^2 \eta^{-2}
+
M^2   \WK_{8}^2
 \eta^{-2} \DDTT^2
 +
  \DTT^2 \WK_{3}^2 \eta^{-2}(1+M^2).
\end{multline}
Above we also used the defintion
\begin{equation}\label{kappa8.def}
            \kappa_{9}
    \eqdef
    \kappa_6 \DDTT^2
    +
    \frac{\DTT^2 \WK_{3}^2}{\lambda \MA(\eta^{-1})^2}
M^2,
\end{equation}
and additionally recalling \eqref{BB.const.def} we define
\begin{equation}\label{WK88.constant.def}
    \WK_{9}
    \eqdef
    1+
    \BB 
    +
    \DDTT 
 M (1+ \rho^{-2} M^2) \eta^{-1}
 +
   \DTT \WK_{3} \eta^{-1}
 M.
\end{equation}
We now choose $\eta>0$ small enough so that we have $C \kappa_1   \DTT \leq \frac{1}{4}\lambda$.
  Next we further integrate in time over $0 \le s \le t$ and afterwards we take the essential supremum in time over $0 \le t \le T$ to obtain 
\begin{multline}\notag
 \sup_{0 \le t \le T} || \delta_\beta (\bX' -  \bY')||_{L^2}^2(t)
    +
  \frac{\lambda}{2} || \delta_\beta \TLam^{\frac12} (\bX'-  \bY')||_{L^2_T(L^2_\theta)}^2
    \\
    \leq   
     || \delta_\beta (\bX'_0 -  \bY'_0)||_{L^2}^2
    +
    C\kappa_{7} \left(||   \delta_\beta \BX'||_{L^2_T(L^\infty_\theta)}^2+||   \delta_\beta \BY'||_{L^2_T(L^\infty_\theta)}^2 \right)||  \bX'-\bY'||_{L^{\infty}_T(\BA)}^2
\\
+
C \kappa_{9}|| \partial_{\beta}( \bX'-\bY')||_{L^2_T(L^\infty_\theta)}^2
+
C|| \partial_{\beta}( \bX'-\bY')||_{L^2_T(L^2_\theta)}^2 \WK_{9}
\\
+
C\left(||   \delta_\beta \BX'||_{L^2_T(L^2_\theta)}^2+||   \delta_\beta \BY'||_{L^2_T(L^2_\theta)}^2 \right)
|| \bX'-\bY'||_{L^{\infty}_T(L^\infty_\theta)}^2
\WK_{7}
\\
+
\frac{1}{2}T\sup_{0 \le t \le T}|| \partial_{\beta}( \bX'-\bY')||_{L^2_\theta}^2(t)
+
C \lambda^{-1} \DDTT^2  
||   \delta_\beta \TLam^{\frac12}\bY'||_{L^2_T(L^2_\theta)}^2
|| \bX'-  \bY'||_{L^{\infty}_T(L^\infty_\theta)}^2.
\end{multline}
Next we suppose that $0<T \le 1$, and we use the inequality \eqref{ineq.gen} to obtain
\begin{multline}\notag
   \frac{1}{2} \sup_{0 \le t \le T} || \delta_\beta (\bX' -  \bY')||_{L^2}(t)
    +
  \frac{\lambda^{1/2}}{2} || \delta_\beta \TLam^{\frac12} (\bX'-  \bY')||_{L^2_T(L^2_\theta)}
    \\
    \leq   
     || \delta_\beta (\bX'_0 -  \bY'_0)||_{L^2}
    +
    C\kappa_{7}^{1/2} \left(||   \delta_\beta \BX'||_{L^2_T(L^\infty_\theta)}+||   \delta_\beta \BY'||_{L^2_T(L^\infty_\theta)} \right)||  \bX'-\bY'||_{L^{\infty}_T(\BA)}
\\
+
C \kappa_{9}^{1/2}|| \partial_{\beta}( \bX'-\bY')||_{L^2_T(L^\infty_\theta)}
+
C|| \partial_{\beta}( \bX'-\bY')||_{L^2_T(L^2_\theta)} \WK_{9}^{1/2}
\\
+
C\left(||   \delta_\beta \BX'||_{L^2_T(L^2_\theta)}+||   \delta_\beta \BY'||_{L^2_T(L^2_\theta)} \right)
|| \bX'-\bY'||_{L^{\infty}_T(L^\infty_\theta)}
\WK_{7}^{1/2}
\\
+
C \lambda^{-1/2} \DDTT
||   \delta_\beta \TLam^{\frac12}\bY'||_{L^2_T(L^2_\theta)}
|| \bX'-  \bY'||_{L^{\infty}_T(L^\infty_\theta)}.
\end{multline}
We further integrate the above in $d\beta$ against $|\beta|^{-3/2}\nu(|\beta|^{-1})$ 
for $\nu$ defined in \eqref{nu.definition} to obtain
\begin{multline}\notag
   \frac{1}{2} ||\bX' -  \bY'||_{\CN}
    +
  \frac{\lambda^{1/2}}{2} || \bX'-  \bY'||_{\DN}
    \leq   
     || \bX'_0 -  \bY'_0||_{\BN}
     \\
    +
    C\kappa_{7}^{1/2} \left(||    \BX'||_{\widetilde{L}^{2}_T(\dot{B}_{\infty, 1}^{\frac12, \nu})}+|| \BY'||_{\widetilde{L}^{2}_T(\dot{B}_{\infty, 1}^{\frac12, \nu})} \right)||  \bX'-\bY'||_{\CN}
\\
+
C \kappa_{9}^{1/2}|| \bX'-\bY'||_{\widetilde{L}^{2}_T(\dot{B}_{\infty, 1}^{\frac12, \nu})}
+
CT^{1/2}\WK_{9}^{1/2} || \bX'-\bY'||_{\CN} 
\\
+
CT^{1/2}\left(||\BX'||_{\CN}+||\BY'||_{\CN} \right)
|| \bX'-\bY'||_{L^{\infty}_T(L^\infty_\theta)}
\WK_{7}^{1/2}
\\
+
C \lambda^{-1/2} \DDTT
||  \bY'||_{\DN}
|| \bX'-  \bY'||_{L^{\infty}_T(L^\infty_\theta)}.
\end{multline}
We use the embedding \eqref{embedding.infty.use} to see that 
$||    f||_{\widetilde{L}^{2}_T(\dot{B}_{\infty, 1}^{\frac12, \nu})} \leq C_{\nu}  ||    f||_{\DN}$.  We also use the embeddings in \eqref{embed.infty}, Proposition \ref{besov.ineq.prop} and then we use Definition \ref{subw.definition} to obtain
 \begin{equation}\notag
|| f||_{L^{\infty}_T(L^\infty_\theta)}
     \leq
C || f ||_{L^{\infty}_T(\dot{B}^{\frac{1}{2}}_{2,1})}
   \leq C 
|| f ||_{\CN}. 
 \end{equation}
 The last inequality above follows simply because $\nu \ge 1$ in \eqref{nu.definition}.
 Thus we have 
 \begin{multline}\notag
   \frac{1}{2} ||\bX' -  \bY'||_{\CN}
    +
  \frac{\lambda^{1/2}}{2} || \bX'-  \bY'||_{\DN}
    \leq   
     || \bX'_0 -  \bY'_0||_{\BN}
     \\
    +
    C \max\{1,M\}C_{\nu} \kappa_{7}^{1/2} \left(||    \BX'||_{\DN}+|| \BY'||_{\DN} \right)||  \bX'-\bY'||_{\CN}
\\
+
C C_{\nu} \kappa_{9}^{1/2}|| \bX'-\bY'||_{\DN}
+
CT^{1/2}\WK_{9}^{1/2} || \bX'-\bY'||_{\CN} 
\\
+
CT^{1/2}\left(||\BX'||_{\CN}+||\BY'||_{\CN} \right)
|| \bX'-\bY'||_{\CN}
\WK_{7}^{1/2}
\\
+
C \lambda^{-1/2} \DDTT
||  \bY'||_{\DN}
|| \bX'-  \bY'||_{\CN}.
\end{multline}
Now using \eqref{apriori.bd.norm.Z} and \eqref{equivalent.nu.norm} we can choose $\eta>0$ additionally small enough so that $\kappa_7>0$ from \eqref{kappa7.def} enforces
\begin{equation*}
    C \max\{1,M\}C_{\nu} \kappa_{7}^{1/2} \left(||    \BX'||_{\DN}+|| \BY'||_{\DN} \right)
    \leq 4 C \max\{1,M\}C_{\nu} \kappa_{7}^{1/2} M
   <
    \frac{1}{8}.
\end{equation*}
Then we can further choose $\eta>0$ additionally possibly smaller  so that $\kappa_9$ from \eqref{kappa8.def} enforces
\begin{equation*}
    C C_{\nu} \kappa_{9}^{1/2}
    <
   \frac{\lambda^{1/2}}{4} .
\end{equation*}
Thus we obtain 
 \begin{multline}\notag
   \frac{3}{8} ||\bX' -  \bY'||_{\CN}
    +
  \frac{\lambda^{1/2}}{4} || \bX'-  \bY'||_{\DN}
    \leq   
     || \bX'_0 -  \bY'_0||_{\BN}
     \\
     +CT^{1/2} \WK_0 || \bX'-\bY'||_{\CN} 
+
C \lambda^{-1/2} \DDTT
||  \bY'||_{\DN}
|| \bX'-  \bY'||_{\CN}.
\end{multline}
where recalling \eqref{WK88.constant.def}, \eqref{WK7.constant.def}, \eqref{apriori.bd.norm.Z} and  \eqref{equivalent.nu.norm} we define
\begin{equation}\label{WK0.constant.def}
    \WK_0 \eqdef 
\WK_{9}^{1/2} 
+
M
\WK_{7}^{1/2}.
\end{equation}
Next from \eqref{nu.definition} and \eqref{D.space.temporal} we have that 
   \begin{equation}\notag
||  \bY'||_{\DN}
=
\int_{\T} \frac{d\beta}{|\beta|^{3/2}} 
 ||\delta_\beta \TLam^{\frac{1}{2}}\bY'||_{L^2_T(L^2_\theta)}
 +
 \frac{1}{\KC}
||  \bY'||_{\DM}.
    \end{equation}
    Since we can bound $||  \bY'||_{\DM} \leq 
 C \lambda^{-1/2}  || \BY'_0||_{\BS}$ as in Proposition \ref{prop:general.apriori.final.local} and \eqref{apriori.bd.norm.Z}, then we can make the second term above arbitrarily small.  In particular we can choose $\KCC\ge 1$ large enough so that
    \begin{equation}\notag
 C \lambda^{-1/2} \frac{\DDTT ||  \bY'||_{\DM}}{\KC}
\leq  C \lambda^{-1}  \frac{\DDTT || \BY'_0||_{\BS}}{\KC}
\leq  C \lambda^{-1}  \frac{\DDTT}{\KCC} < \frac{1}{16}.
    \end{equation}
It is important that $\KCC = \KCC(\lambda, \DDTT)$ but $\KCC$ does not depend upon $\eta$.  Then 
 for the first term above, we split into $|\beta| < \eta_1$ and $|\beta| \geq \eta_1$ for some small $\eta_1>0$.  Then similar to \eqref{extra.smallness.besov} using also \eqref{apriori.bd.norm.Z} we have 
   \begin{equation}\notag
\int_{|\beta| < \eta_1} \frac{d\beta}{|\beta|^{3/2}} 
 ||\delta_\beta \TLam^{\frac{1}{2}}\bY'||_{L^2_T(L^2_\theta)}
\leq
\frac{1}{\subw(\eta_1^{-1})}
||  \bY'||_{\DM}
\leq
C\frac{M}{\lambda^{\frac12}\subw(\eta_1^{-1})}.
    \end{equation}
For the other part, again with \eqref{apriori.bd.norm.Z}, we have
\begin{multline*}
    \int_{|\beta| \geq \eta_1} \frac{d\beta}{|\beta|^{3/2}} 
 ||\delta_\beta \TLam^{\frac{1}{2}}\bY'||_{L^2_T(L^2_\theta)}
\leq
C\eta_1^{-1/2}
 ||\TLam^{\frac{1}{2}}\bY'||_{L^2_T(L^2_\theta)}
 \\
 \leq
 C T^{1/2}\eta_1^{-1/2}
||\bY'||_{L^\infty_T(\dot{B}^{\frac{1}{2}}_{2,2})}
 \leq
  C T^{1/2}M\eta_1^{-1/2}.
\end{multline*}
Notice that $\frac{1}{\subw(\eta_1^{-1})}$ can be made arbitrarily small for $\eta_1>0$ chosen small enough.  Thus if we choose $\eta_1>0$ small enough we have 
\begin{equation}\notag
C \lambda^{-1/2} \DDTT
\frac{M}{\lambda^{\frac12}\subw(\eta_1^{-1})}
< 
\frac{1}{16}.
\end{equation}
 Thus we obtain 
      \begin{equation}\notag
 ||\bX' -  \bY'||_{\CN}
    +
  \lambda^{1/2} || \bX'-  \bY'||_{\DN}
    \leq   
    4 || \bX'_0 -  \bY'_0||_{\BN}
     +C || \bX'-\bY'||_{\CN} T^{1/2} \WK, 
 \end{equation}
 where using \eqref{WK0.constant.def} we define 
 \begin{equation}\label{WK.constant.def}
   \WK \eqdef 
\WK_0+
\eta_1^{-1/2}
M \lambda^{-1/2} \DDTT.
\end{equation}
The proof is complete.
\end{proof}

\subsection{$L^2$ continuity estimate}\label{sec:l2continuity} For some $\CTS >0$ we now suppose for $T>0$ that for some $c>0$ and $\lambda>0$ as in \eqref{e:DTdefn} for some $M > 0$ that we have
\begin{equation}\label{apriori.bd.norm.H1}
||\bZ'||_{L^\infty_T \dot{H}^{\frac12}}+ c\lambda^{\frac12} ||\bZ'||_{L^2_T \dot{H}^1} \leq \CTS M.
\end{equation}
Notice that this condition is implied by \eqref{apriori.bd.norm.Z}.  
Then in this subsection we will prove in the following proposition that as long as \eqref{apriori.bd.norm.H1} holds then the $L^2_\theta$ norm of the difference of two solutions to \eqref{peskin.general.tension} is stable.

\begin{proposition}\label{prop.L2.cont}
Let $\bX, \bY: [0,T]\times \T \to \R^2$ be two weak solutions to the Peskin problem \eqref{peskin.general.tension} with tension $\TE$ \eqref{tension.map.def} in the sense of Definition \ref{def:solution} with initial data $\BX_0,$ $\BY_0$ respectively.  Assume that $\BX_0,$ $\BY_0,$ and $\bT$ satisfy the assumptions of Theorem \ref{thm:mainquant}, in particular we assume \eqref{e:QuantitativeTensionMap}.  Additionally assume \eqref{apriori.bd.Z} holds with $T>0$. 
Then for two solutions $\BX'$ and $\BY'$  over $0\le t \le T$  we have 
\begin{equation}\notag
    ||(\bX'-\bY')(t)||_{L^2_\theta} \leq C  ||\bX_0'-\bY_0'||_{L^2_\theta},
\end{equation}
where $C = C(M, \rho, \lambda, \DTT, \DDTT)>0.$
\end{proposition}

\begin{proof}
Direct calculation gives us that 
\begin{multline}\notag
    \frac{d}{dt}||\bX'-\bY'||_{L^2_\theta}^2 = -\int_\T \int_\T \frac{d\alpha d\theta}{\alpha^2} \delta_\alpha (\bX'-\bY')\cdot \left( \mathcal{K}[\bX]\delta_\alpha \bT(\bX') - \mathcal{K}[\bY]\delta_\alpha \bT(\bY')\right)
    \\ = -\int_\T \int_\T \frac{d\alpha d\theta}{\alpha^2} \delta_\alpha (\bX'-\bY')\cdot \mathcal{K}[\bX]\delta_\alpha (\bT(\bX') -  \bT(\bY'))
    \\ -\int_\T \int_\T \frac{d\alpha d\theta}{\alpha^2} \delta_\alpha (\bX'-\bY')\cdot \left( \mathcal{A}[\bX] - \mathcal{A}[\bY]\right)\delta_\alpha \bT(\bY')
 = I_{\mathcal{K}} + I_{\mathcal{A}}.
\end{multline}
 Recalling \eqref{DBTX.def}, we use \eqref{kerbel.eqn.deriv} and \eqref{difference.alpha.T} to expand out
$I_{\mathcal{K}}$ as 
\begin{multline}\notag
I_{\mathcal{K}} =  -\frac{1}{4\pi}\int_\T \int_\T \frac{d\alpha d\theta}{\alpha^2} \delta_\alpha (\bX'-\bY') \cdot \DBT[\bX'] \delta_\alpha (\bX'-\bY') 
\\
 -\int_\T \int_\T \frac{d\alpha d\theta}{\alpha^2} \delta_\alpha (\bX'-\bY') \cdot \mathcal{A}[\bX]\DBT[\bX'] \delta_\alpha (\bX'-\bY') 
 \\
 -\int_\T \int_\T \frac{d\alpha d\theta}{\alpha^2} \delta_\alpha (\bX'-\bY') \cdot  \mathcal{K}[\bX](\DBT[\bX']-\DBT[\bY'])\delta_\alpha \bY'
 =I_{\mathcal{K}}^1+I_{\mathcal{K}}^2+I_{\mathcal{K}}^3. 
\end{multline}
Then from \eqref{e:DTdefn} we have $I_{\mathcal{K}}^1 \leq -\lambda ||\TLam^{\frac12}(\bX'-\bY')||_{L^2_\theta}^2$.  

Next we estimate the following sample term for an integer $j\geq 1$ using Proposition \ref{besov.ineq.prop} and Lemma \ref{Besov.interpolation} and Young's inequality for any small constant $c>0$ as
\begin{multline}\label{sample.term.est}
    \int_\T d\theta \int_\T \frac{d\alpha }{\alpha^2}
    |\delta_\alpha (\bX'-\bY')|^2 |\delta_\alpha \bZ'|^j 
    \leq  C \int_\T \frac{d\alpha}{\alpha^2} ||\delta_\alpha (\bX'-\bY')||_{L^2_\theta}^2
    ||\delta_\alpha \BZ'||_{L^\infty_\theta}^j
    \\ 
    \leq  C ||\bX'-\bY'||_{\dot{B}^{\frac14}_{2,4}}^2 ||\bZ'||_{\dot{B}^{\frac{1}{2j}}_{\infty, 2j}}^j 
        \leq  C ||\bX'-\bY'||_{\dot{H}^{\frac14}}^2 ||\bZ'||_{\dot{B}^{\frac{j+1}{2j}}_{2, 2j}}^j 
    \\ \leq  C ||\bX'-\bY'||_{\dot{H}^{\frac12}}
    ||\bX'-\bY'||_{L^2_\theta}
    ||\bZ'||_{\dot{H}^{\frac12}}^{j-1} 
  ||\bZ'||_{\dot{H}^{1}}
  \\
  \leq c\lambda ||\bX'-\bY'||_{\dot{H}^{\frac12}}^2
  + 
  C \lambda^{-1} ||\bX'-\bY'||_{L^2_\theta}^2
    ||\bZ'||_{\dot{H}^{\frac12}}^{2(j-1)} 
  ||\bZ'||_{\dot{H}^{1}}^2.
\end{multline}
Then for $I_{\mathcal{K}}^2$ we use \eqref{DBTX.def}, \eqref{e:QuantitativeTensionMap}, \eqref{e:Abounds}, \eqref{apriori.bd.Z} and \eqref{sample.term.est}  to obtain
\begin{multline}\label{est.IK2}
    I_{\mathcal{K}}^2
    \leq
    C\int_\T d\theta \int_\T \frac{d\alpha }{\alpha^2}|\delta_\alpha (\bX'-\bY')|^2 (\rho^{-1}|\delta_\alpha \bX'| + \rho^{-2}|\delta_\alpha \bX'|^2)|\DBT[\bX']| 
    \\ \leq \frac{\lambda}{8}||\bX'-\bY'||_{\dot{H}^{\frac12}}^2 +C\frac{\DTT^2}{\lambda}\left(1+\frac{||\bX'||_{\dot{H}^{\frac12}}^2}{\rho^2}\right)\frac{||\bX'||_{\dot{H}^1}^2}{\rho^2}||\bX'-\bY'||_{L^2_\theta}^2.
\end{multline}
Next we will estimate $I_{\mathcal{K}}^3$.  First similar to \eqref{sample.term.est} for an integer $j\geq 1$ we estimate 
\begin{multline}\label{sample.term2.est}
    \int_\T d\theta \int_\T \frac{d\alpha }{\alpha^2}
    |\delta_\alpha (\bX'-\bY')||\BX'-\BY'| |\delta_\alpha \bZ'|^j 
    \\
    \leq  C 
     || \bX'-\bY'||_{L^2_\theta}\int_\T \frac{d\alpha}{\alpha^2} ||\delta_\alpha (\bX'-\bY')||_{L^2_\theta}
    ||\delta_\alpha \BZ'||_{L^\infty_\theta}^j
    \\ 
    \leq  C  || \bX'-\bY'||_{L^2_\theta} || \bX'-\bY'||_{\dot{H}^{\frac12}} ||\bZ'||_{\dot{B}^{\frac{1}{2j}}_{\infty, 2j}}^j 
    \\ \leq  C ||\bX'-\bY'||_{\dot{H}^{\frac12}}
    ||\bX'-\bY'||_{L^2_\theta}
    ||\bZ'||_{\dot{H}^{\frac12}}^{j-1} 
  ||\bZ'||_{\dot{H}^{1}}
  \\
  \leq c\lambda ||\bX'-\bY'||_{\dot{H}^{\frac12}}^2
  + 
  C \lambda^{-1} ||\bX'-\bY'||_{L^2_\theta}^2
    ||\bZ'||_{\dot{H}^{\frac12}}^{2(j-1)} 
  ||\bZ'||_{\dot{H}^{1}}^2.
\end{multline}
Now we use \eqref{difference.DBTXY} with \eqref{DBTX.def} and \eqref{e:QuantitativeTensionMap} to see that 
\begin{equation}\notag
  \left|  \DBTX - \DBTY \right|
 \lesssim \DDTT   |  \BX'-\BY'  |.
\end{equation}
Thus for $I_{\mathcal{K}}^3$ with \eqref{kerbel.eqn.deriv}, \eqref{e:Abounds} and \eqref{sample.term2.est}  we have the following bound
\begin{multline}\label{e:L2contDbound}
 I_{\mathcal{K}}^3  \leq C \DDTT \int_\T \int_\T \frac{d\alpha d\theta}{\alpha^2}|\delta_\alpha (\bX'-\bY')| |\BX'-\BY'| \left(|\delta_\alpha \bZ'|+\frac{|\delta_\alpha \bZ'|^2}{\rho} + \frac{|\delta_\alpha \bZ'|^3}{\rho^2}\right)
     \\ 
     \leq 
     C\frac{\DDTT^2}{\lambda}\left(1+\rho^{-2}||\bZ'||_{\dot{H}^{\frac12}}^2+\rho^{-4}||\bZ'||_{\dot{H}^{\frac12}}^4\right)||\bZ'||_{\dot{H}^1}^2||\bX'-\bY'||_{L^2_\theta}^2
     \\
     +     \frac{\lambda}{8}||\bX'-\bY'||_{\dot{H}^{\frac12}}^2.
\end{multline}
These are all of our estimates for $ I_{\mathcal{K}}$.

To estimate $I_{\mathcal{A}}$ we use the bounds in \eqref{delta.alpha.BX.bound} and \eqref{A.diff.est.noPM} to see that 
\begin{multline}\notag
I_{\mathcal{A}} \lesssim  \DTT \int_\T \int_\T \frac{d\alpha d\theta}{\alpha^2} |\delta_\alpha (\bX'-\bY')|^2 
\left(\rho^{-1}|\delta_\alpha \bZ'| +
\rho^{-2}|\delta_\alpha \bZ'|^2
\right)
\\
+\DTT \rho^{-2}\int_\T \int_\T \frac{d\alpha d\theta}{\alpha^2} |\delta_\alpha (\bX'-\bY')|
\left| \DAL (\BX - \BY)\right| 
\left(|\delta_\alpha \bZ'|^2 +
\rho^{-1}|\delta_\alpha \bZ'|^3
\right)
=I_{\mathcal{A}}^1+I_{\mathcal{A}}^2.
\end{multline}
Then similar to \eqref{sample.term.est} and \eqref{est.IK2} we have 
\begin{equation*}
        I_{\mathcal{A}}^1
 \leq \frac{\lambda}{8}||\bX'-\bY'||_{\dot{H}^{\frac12}}^2 +C\frac{\DTT^2}{\lambda}\left(1+\frac{||\bZ'||_{\dot{H}^{\frac12}}^2}{\rho^2}\right)\frac{||\bZ'||_{\dot{H}^1}^2}{\rho^2}||\bX'-\bY'||_{L^2_\theta}^2.
\end{equation*}
Also using \eqref{operator.bd.first} then similar to \eqref{sample.term2.est} and \eqref{e:L2contDbound} we have 
\begin{multline}\notag
I_{\mathcal{A}}^2  
     \leq 
     C\frac{\DDTT^2\rho^{-4}}{\lambda}\left(||\bZ'||_{\dot{H}^{\frac12}}^2+\rho^{-2}||\bZ'||_{\dot{H}^{\frac12}}^4\right)||\bZ'||_{\dot{H}^1}^2||\bX'-\bY'||_{L^2_\theta}^2
     \\
     +     \frac{\lambda}{8}||\bX'-\bY'||_{\dot{H}^{\frac12}}^2.
\end{multline}
These are our main estimates for $I_{\mathcal{A}}$.

Now from all of the bounds above we define
\begin{equation*}
        \JT (t)
    \eqdef
     \frac{\DTT^2}{\rho^2\lambda}\left(1+\frac{||\bZ'(t)||_{\dot{H}^{\frac12}}^2}{\rho^2}\right)
    +
        \frac{\DDTT^2}{\lambda}\left(1+\rho^{-2}\right)\left(1+\rho^{-4}||\bZ'(t)||_{\dot{H}^{\frac12}}^4\right).
\end{equation*}
Then putting all of these bounds together, we get that 
\begin{equation*}
        \frac{d}{dt}\log(||(\bX'-\bY')(t)||_{L^2_\theta}^2)
        \leq C \JT(t) ||\bZ'(t)||_{\dot{H}^1}^2.
\end{equation*}
We conclude that 
\begin{equation}\notag
    ||(\bX'-\bY')(t)||_{L^2_\theta}^2 \leq \exp\left(C\int_0^t ds~\JT(s) ||\bZ'(s)||_{\dot{H}^1}^2 \right) ||\bX'_0-\bY'_0||_{L^2_\theta}^2.
\end{equation}
Then applying \eqref{apriori.bd.norm.H1} completes the proof.
\end{proof}

\begin{corollary}\label{cor.L2.cont.m}
Let $\bX, \bY: [0,T]\times \T \to \R^2$ be two weak solutions to the Peskin problem \eqref{peskin.general.tension} with tension $\TE$ in the sense of Definition \ref{def:solution} with initial data $\BX_0,$ $\BY_0$ respectively, satisfying all the conditions in Proposition \ref{prop.L2.cont}.    Let $\subw$ and $\omega$ satisfy in Definition \ref{subw.definition} and additionally suppose that there exists $r_* \ge 1$ such that $\frac{\omega(r)}{\subw(r)}$ is decreasing for $r \ge r_*$ and in particular 
\begin{equation*}
    \displaystyle\lim\limits_{r\to \infty}\frac{\omega(r)}{\subw(r)} = 0.
\end{equation*}
For any $\varepsilon>0,$ there exists $\delta_*>0$ such that for any $0<\delta\le \delta_*$ then \eqref{apriori.bd.norm.Z} and  $||\bX'_0-\bY'_0||_{L^2_\theta}<\delta$ imply
\begin{equation}\notag
 ||\bX'-\bY'||_{\mathcal{B}^{\omega}_T}<\varepsilon.
\end{equation}
\end{corollary}

\begin{proof}
For any small $\eta>0,$ we can bound
\begin{multline}\notag
    ||\bX'-\bY'||_{\mathcal{B}^\omega_T}  = \int_\T \frac{d\beta}{|\beta|^{3/2}}\omega(|\beta|^{-1}) \sup_{0\leq t\leq T} ||\delta_\beta(\bX'-\bY')(t)||_{L^2_\theta} \leq 
\int_{|\beta|<\eta} + \int_{|\beta|>\eta}
\\ \leq \frac{\omega(\eta^{-1})}{\mu(\eta^{-1})}(||\bX'||_{\CM}+||\bY'||_{\CM}) + \frac{\omega(\eta^{-1})}{\eta^{1/2}} \sup\limits_{0\leq t\leq T} ||(\bX'-\bY')(t)||_{L^2_\theta}.
\end{multline}
Thus by our assumptions on $\subw$, $\omega$, $\bX'$, and $\bY'$, we can take $\eta>0$ sufficiently small to guarantee that
$$\displaystyle\frac{\omega(\eta^{-1})}{\mu(\eta^{-1})} (||\bX'||_{\CM}+||\bY'||_{\CM})<\frac{\varepsilon}{2}.$$
Then applying Proposition \ref{prop.L2.cont}, we can take $\delta>0$ sufficiently small to obtain the result.  
\end{proof}

\section{Higher regularity}\label{sec:smoothing}

In this section we establish the gain of higher regularity for the solutions $\BX'(t,\theta)$ to the Peskin problem \eqref{peskin.general.tension} satisfying \eqref{initial.assumption}, \eqref{apriori.bd} and \eqref{apriori.bd.norm}.  In \secref{sec:oneTwoEst} we prove the $C^{\frac12}_{t,x}$ estimate.  Then in \secref{sec:HigherEst} we prove the $C^{1,\alpha}_{t,x}$ estimate and the higher regularity.

\subsection{$C^{\frac12}_{t,x}$ estimate for $\BX'(t,\theta)$}\label{sec:oneTwoEst}
We now prove the $C^{\frac12}_{t,x}$ estimate for solutions $\BX'(t,\theta)$ to the Peskin problem \eqref{peskin.general.tension}.  We first prove in Lemma \ref{estimate.q} a general estimate of some quantities that will come up repeatedly in subsequent estimates.

\begin{lemma}\label{estimate.q}  For any $q\in \N$ we have the following uniform estimates: 
\begin{multline}\notag 
\YTT_q \eqdef \int_{\T} \frac{d\beta}{\beta^2} \int_\T d\theta
\bigg|\int_{\T} \frac{d\alpha}{\alpha^2}   | \delta_\alpha \bX'(\theta)|^{q} |\delta_\beta\delta_\alpha  \bX'(\theta)|\bigg|^2 \lesssim ||\bX'||_{\dot{H}^{\frac12}}^{2(q-1)}||\bX'||_{\dot{H}^1}^4,
\\ 
\STT_q \eqdef \int_{\T} \frac{d\beta}{\beta^2} \int_\T d\theta 
\bigg|\int_{\T} \frac{d\alpha}{\alpha^2}   | \delta_\alpha \bX'(\theta)|^{q+1} |\delta_\beta  \bX'(\theta)|\bigg|^2
\lesssim
||\bX'||_{\dot{H}^{\frac12}}^{2q} ||\bX'||_{\dot{H}^{1}}^4.
\end{multline}
\end{lemma}

\begin{proof}
Fix $q\in \N$.  We apply Minkowski's inequality in $\theta$ and $\alpha$, and then we use the Cauchy-Schwartz inequality to obtain
\begin{multline}\notag
   \int_\T d\theta
\bigg|\int_{\T} \frac{d\alpha}{\alpha^2}   | \delta_\alpha \bX'(\theta)|^{q} |\delta_\beta\delta_\alpha  \bX'(\theta)|\bigg|^2 
\leq 
\left(\int_\T \frac{d \alpha}{\alpha^2} \left[\int_\T d\theta |\delta_\alpha \bX'|^{2q} |\delta_\beta\delta_\alpha \bX'|^2\right]^\frac{1}{2}\right)^2 
\\ \leq \left(\int_\T \frac{d\alpha}{\alpha^2}||\delta_\alpha \bX'||_{L^\infty_\theta}^q||\delta_\beta \delta_\alpha \bX'||_{L^2_\theta} \right)^2 \leq ||\bX'||_{\dot{B}^{1/2q}_{\infty, 2q}}^{2q}||\delta_\beta \bX'||_{\dot{H}^{1/2}}^2,
\end{multline}
and similarly
\begin{multline}\notag
   \int_\T d\theta
\bigg|\int_{\T} \frac{d\alpha}{\alpha^2}   | \delta_\alpha \bX'(\theta)|^{q+1} |\delta_\beta  \bX'(\theta)|\bigg|^2 \leq \left(\int_\T \frac{d \alpha}{\alpha^2} \left[\int_\T d\theta |\delta_\alpha \bX'|^{2(q+1)} |\delta_\beta \bX'|^2\right]^\frac{1}{2}\right)^2 
\\ \leq \left(\int_\T \frac{d\alpha}{\alpha^2}||\delta_\alpha \bX'||_{L^{2(q+1)}_\theta}^{q+1}||\delta_\beta  \bX'||_{L^\infty_\theta} \right)^2 \leq ||\bX'||_{\dot{B}^{1/(q+1)}_{2(q+1), q+1}}^{2(q+1)}||\delta_\beta \bX'||_{L^\infty_\theta}^2.
\end{multline}
Integrating against $\displaystyle\frac{d\beta}{\beta^2}$ 
we obtain
\begin{eqnarray}\notag
    &\YTT_q \leq ||\bX'||_{\dot{B}^{1/2q}_{\infty, 2q}}^{2q}\int_\T\frac{d\beta}{\beta^2}||\delta_\beta \bX'||_{\dot{H}^{1/2}}^2\lesssim ||\bX'||_{\dot{B}^{1/2q}_{\infty, 2q}}^{2q}||\bX'||_{\dot{H}^1}^2,
    \\ \notag
&\STT_q \leq ||\bX'||_{\dot{B}^{1/(q+1)}_{2(q+1), q+1}}^{2(q+1)}\int_\T\frac{d\beta}{\beta^2}||\delta_\beta \bX'||_{L^\infty_\theta}^2 
\lesssim
||\bX'||_{\dot{B}^{1/(q+1)}_{2(q+1), q+1}}^{2(q+1)}
||\bX'||_{\dot{B}^{1/2}_{\infty, 2}}^{2}.
\end{eqnarray}
Then above we will use $||\bX'||_{\dot{B}^{1/2}_{\infty, 2}}^{2} \lesssim || \bX'||_{\dot{H}^1}^2$ from Proposition \ref{besov.ineq.prop} .
Finally, since $q\ge 1$, applying Proposition \ref{besov.ineq.prop} and Lemma \ref{Besov.interpolation} gives 
\begin{multline}\notag
    ||\bX'||_{\dot{B}^{1/2q}_{\infty, 2q}}^{2q} \lesssim ||\bX'||_{\dot{H}^{1/2 + 1/2q}}^{2q} \lesssim ||\bX'||_{\dot{H}^{1/2}}^{2(q-1)}||\bX'||_{\dot{H}^1}^2, 
\\ ||\bX'||_{\dot{B}^{1/(q+1)}_{2(q+1), q+1}}^{2(q+1)} \lesssim ||\bX'||_{\dot{H}^{1/2 + 1/2(q+1)}}^{2(q+1)} \lesssim ||\bX'||_{\dot{H}^{1/2}}^{2q}||\bX'||_{\dot{H}^1}^2,
\end{multline}
completing the estimate.  
\end{proof}

Let $\bX'$ be a smooth solution of \eqref{peskin.general.tension} with \eqref{kerbel.eqn.deriv} and \eqref{kerbel.A.eqn.deriv}, we will use the equation in the form \eqref{v.theta.def}.  Next we prove the $\dot{H}^1$ estimate.

\begin{proposition}\label{prop:H1.estimate}
For any $0 < t_0 <t < T$ we have the following estimate
\begin{equation}\label{H1.estimate}
    ||\bX'||_{\dot{H}^1}^2(t) \leq  ||\bX'||_{\dot{H}^1}^2(t_0)\exp\left(C
    || \HEU||_{L^\infty(t_0,t)}\int_{t_0}^t ds ||\bX'||_{H^1}^2(s)\right) .
\end{equation}
Here  $\HEU=\HEU(s)=\HEU(||\bX'(s)||_{\dot{H}^{\frac12}}, \rho^{-1}, \lambda^{-1}, \DTT, \DDTT)$ is a polynomial that is written explicitly in \eqref{HEU.def}.  Thus in particular we have that
\begin{equation}\label{H1.estimate2}
    ||\bX'||_{\dot{H}^1}^2(t)\leq C(M, \mu, \rho, \lambda, \DTT, \DDTT)  ||\bX'||_{\dot{H}^1}^2(t_0).
\end{equation}
\end{proposition}

\begin{proof}
Notice that we use $||\bX'||_{\dot{H}^1}^2 = \int_{\T} d\theta~ | \TLam \BX'(\theta)|^2$ with \eqref{tildeLambda:eq}.  Thus from \eqref{v.theta.def} we have 
\begin{equation}\label{time.deriv.energy}
        \frac{1}{2}\frac{d}{dt}||\bX'||_{\dot{H}^1}^2 = -\int_{\T} d\theta ~\TLam^{\frac32}\bX' \cdot \TLam^{\frac32}\bT(\bX') 
        + \int_{\T} d\theta ~\TLam^{\frac32}\bX' \cdot \TLam^{\frac12}\VTT. 
\end{equation}
To estimate the first term we split 
\begin{equation*}
            \TLam^{3/2}\bT(\bX') = D\bT(\bX'(\theta)) \TLam^{3/2} \bX' + \HEK,
\end{equation*}
where similar to \eqref{e:deltaalphaT} and \eqref{DBTX.def} we have
\begin{equation*}
       \HEK \eqdef \frac{1}{4\pi}\int_{\T} \frac{d\alpha}{\alpha^{5/2}} \left(\int_0^1\left( D\bT(\bX'(\theta) + s\delta_\alpha \bX'(\theta)) - D\bT(\bX'(\theta)) \right)ds\right)\delta_\alpha \bX'(\theta).
\end{equation*}
Then similar to \eqref{DDBT.term} we have 
\begin{equation*}
     \left| D\bT(\bX'(\theta) + s\delta_\alpha \bX'(\theta)) - D\bT(\bX'(\theta)) \right|
\lesssim \DDTT |\delta_\alpha \bX'(\theta)|.
\end{equation*}
Then using Minkowski's inequality and the Besov space embeddings in Proposition \ref{besov.ineq.prop}, we bound $\HEK$ in $L^2$ as
\begin{equation}\notag
|| \HEK ||_{L^2_\theta}
\lesssim
\DDTT \int_{\T} \frac{d\alpha}{\alpha^{5/2}}||\delta_\alpha \BX'||_{L^4_\theta}^2
\lesssim  
\DDTT ||\bX'||_{\dot{B}^{3/4}_{4,2}}^2 \lesssim  \DDTT ||\bX'||_{\dot{H}^1}^2 .
\end{equation}
Recalling $D\bT(z)\geq \lambda I$ from \eqref{e:DTdefn}, applying Young's inequality we thus have
\begin{equation}\label{HEK.estimate}
\begin{split}
-\int_{\T} d\theta~ \TLam^{\frac32}\bX' \cdot \TLam^{\frac32}\bT(\bX') &\leq -\lambda ||\bX'||_{\dot{H}^{\frac32}}^2 + C\DDTT  ||\bX'||_{\dot{H}^{\frac32}} ||\bX'||_{\dot{H}^{1}}^{2} 
\\&\leq -\frac{\lambda}{2} ||\bX'||_{\dot{H}^{\frac32}}^2 + C \lambda^{-1}\DDTT^2 ||\bX'||_{\dot{H}^{1}}^4.
\end{split}
\end{equation}
This is our main estimate for the first term in \eqref{time.deriv.energy}.  

To estimate the second term in \eqref{time.deriv.energy}, it suffices to bound 
$\TLam^{\frac12}\VTT$ from \eqref{v.theta.def} in $L^2$.  This is equivalent to bounding 
$    \int_{\T} \frac{d\beta}{\beta^2} \int_{\T} d\theta \left(\delta_\beta  \VTT(\theta)\right)^2$.   Thus, we have 
\begin{multline*}
     || \TLam^{\frac12}\VTT||_{L^2_\theta}^2 \approx    \int_{\T} \frac{d\beta}{\beta^2} \int_{\T} d\theta \bigg|\int_\T \frac{d\alpha}{\alpha^2}  \delta_\beta[\mathcal{A}(\theta, \alpha) \delta_\alpha \bT(\bX'(\theta))] \bigg|^2
         \\
     \lesssim
          \int_{\T} \frac{d\beta}{\beta^2} \int_{\T} d\theta \bigg|\int_\T \frac{d\alpha}{\alpha^2}  |\mathcal{A}(\theta, \alpha)| \ |\delta_\beta\delta_\alpha \bT(\bX'(\theta))|\bigg|^2
          \\
          +
               \int_{\T} \frac{d\beta}{\beta^2} \int_{\T} d\theta \bigg|\int_\T \frac{d\alpha}{\alpha^2}    |\delta_\beta\mathcal{A}(\theta, \alpha)| \ |\tau_\beta\delta_\alpha \bT(\bX'(\theta))| \bigg|^2 = \HUK+\HJK.
\end{multline*}
We now use \eqref{beta.alpha.TXP},  \eqref{e:Abounds}, \eqref{apriori.bd} and Lemma \ref{estimate.q} to calculate that
\begin{multline*}
        \HUK
    \lesssim
    \DTT^2 (\rho^{-2} \YTT_1+\rho^{-4} \YTT_2)+\DDTT^2 (\rho^{-2} \STT_1+\rho^{-4} \STT_2)
    \\
    \lesssim 
    (\DTT^2+ \DDTT^2 ||\bX'||_{\dot{H}^{\frac12}}^2 ) (\rho^{-2} + \rho^{-4}||\bX'||_{\dot{H}^{\frac12}}^2 ) ||\bX'||_{\dot{H}^{1}}^4.
\end{multline*}
These are our main estimates for the term containing $\HUK$.

To bound the term $\HJK$, we will use  \eqref{delta.alpha.BX.bound} and the estimate of $|\delta_\beta \mathcal{A}(\theta, \alpha)|$ in Lemma \ref{A.bound.lem}, \eqref{A1betaRemark} and \eqref{A2betaRemark}.  Then as in Lemma \ref{estimate.q} we have
\begin{multline*}
        \HJK
    \lesssim
    \DTT^2 (\rho^{-2} \YTT_1+\rho^{-4}\YTT_2 +\rho^{-4}\STT_1 + \rho^{-6} \STT_2 )
    \\
    \lesssim 
    \DTT^2 (\rho^{-2} + \rho^{-4} ||\bX'||_{\dot{H}^{\frac12}}^2+ \rho^{-6} || \bX'||_{\dot{H}^{\frac12}}^{4} ) ||\bX'||_{\dot{H}^{1}}^4.
\end{multline*}
Notice that above the estimates in \eqref{A2betaRemark} with $|\delta_\beta \DAL \BX(\theta)|$ can be treated the same as  $|\delta_\beta \BX'(\theta)|$ in Lemma \ref{estimate.q} due to \eqref{operator.bd.first}.

Thus putting everything together, 
we have that 
\begin{equation}\notag
|| \TLam^{\frac12}\VTT ||^2_{L^2_\theta} 
 \lesssim 
     (\DTT^2\rho^{-2}+ \DDTT^2 ) \rho^{-2} ||\bX'||_{\dot{H}^{\frac12}}^2 (1 + \rho^{-2}||\bX'||_{\dot{H}^{\frac12}}^2 ) ||\bX'||_{\dot{H}^{1}}^4
     +
\DTT^2
 \rho^{-2}
||\bX'||_{\dot{H}^{1}}^4.
\end{equation}
Thus for the second term in \eqref{time.deriv.energy} after applying Young's inequality we have  
\begin{multline*}
        \bigg|\int_{\T} d\theta \TLam^{\frac32}\bX' \cdot \TLam^{
        \frac12}\VTT \bigg|  \leq \frac{\lambda}{4}||\bX'||_{\dot{H}^{\frac32}}^2
        +
C\lambda^{-1} 
\DTT^2
 \rho^{-2}
||\bX'||_{\dot{H}^{1}}^4 
        \\
        +  C\lambda^{-1} 
(\DTT^2\rho^{-2}+ \DDTT^2 ) \rho^{-2} ||\bX'||_{\dot{H}^{\frac12}}^2 (1 + \rho^{-2}||\bX'||_{\dot{H}^{\frac12}}^2 ) ||\bX'||_{\dot{H}^{1}}^4.
\end{multline*}
From the above estimate and \eqref{HEK.estimate} we are motivated to define $\HEU = \HEU (s)$ by
\begin{equation}\label{HEU.def}
    \HEU \eqdef
    \lambda^{-1}
    (\DTT^2\rho^{-2}+ \DDTT^2 )  \left(
    \rho^{-2} ||\bX'(s)||_{\dot{H}^{\frac12}}^2 (1 + \rho^{-2}||\bX'(s)||_{\dot{H}^{\frac12}}^2 ) 
     +
1
 \right).
\end{equation}
We plug these estimates into \eqref{time.deriv.energy} and apply Gr\"onwall's inequality to get \eqref{H1.estimate}.  

Recalling \eqref{apriori.bd.norm} and noting that 
\begin{equation}\notag
    ||\bX'||_{L^\infty_t \dot{H}^{\frac12}} \lesssim ||\bX'||_{\CM}, \qquad ||\bX'||_{L^2_t \dot{H}^{1}} \lesssim  ||\bX'||_{\DM},
\end{equation}
then gives \eqref{H1.estimate2}.  
\end{proof}

Next, we prove the gain of $\dot{H}^1$ for small times.

\begin{lemma}\label{H1.Linfinity.time.bound}
Let $\bX'$ be a solution to the Peskin problem \eqref{peskin.general.tension}.  Then for any fixed $\varepsilon>0$ sufficiently small, there exists a time $T_\varepsilon= T_\varepsilon(\varepsilon, \rho, \mu, M, \lambda)>0$ such that for all $0<t\leq T_\varepsilon$ we have
\begin{equation}\notag
    ||\bX'||_{\dot{H}^1}(t) \leq \varepsilon t^{-1/2}.
\end{equation}
\end{lemma}

\begin{proof}
For a fixed $\varepsilon>0$ by Lemma \ref{H1.small.time} for all $t>0$ sufficiently small we have
\begin{equation}\label{small.time.integ}
\int_0^{t}ds ||\bX'||_{\dot{H}^1}^2(s) \leq \frac{\varepsilon^2\log 2}{4}.
\end{equation}
Then as 
\begin{equation}\notag
    \int_{t/2}^{t}ds \frac{\varepsilon^2}{4s} = \frac{\varepsilon^2 \log 2}{4}, 
\end{equation}
there must some time $t_0\in [ t/2, t]$ such that 
\begin{equation}\notag
    ||\bX'||_{\dot{H}^1}^2(t_0) \leq \frac{\varepsilon^2}{4t_0} \leq \frac{\varepsilon^2}{2t}. 
\end{equation}
Then combining the $\dot{H}^1$ estimate \eqref{H1.estimate} with Lemma \ref{H1.small.time}  gives us that
\begin{multline}\notag
    ||\bX'||_{\dot{H}^1}^2(t)  \leq  \frac{\varepsilon^2}{2t} \exp\left(C\sup\limits_{t/2\leq s\leq t}\HEU(s)\int_{t/2}^t ds ||\bX'||_{\dot{H}^1}^2(s)\right)
    \\
    \leq \frac{\varepsilon^2}{2t} \exp\left(C\sup\limits_{t/2\leq s\leq t}\HEU(s)\varepsilon^2\right) \leq \frac{\varepsilon^2}{t},
\end{multline}
so long as $\varepsilon$ is sufficiently small.
\end{proof}

Next we will prove the $C^{1/2}_{t,\theta}$ estimate.

\begin{lemma}\label{lem:chalf}
Let $Q_t = [\frac{t}{2}, t]\times \T$ for all times $0<t\leq T_*,$ where
$0<T_*\le T_\varepsilon$ for some fixed $\varepsilon>0$ and $T_\varepsilon$ as in Lemma \ref{H1.Linfinity.time.bound}.  Then there exists a finite constant $ C=C(\mu, M, \rho, \lambda, \DTT, \DDTT)>0$ such that 
\begin{equation}\notag
    ||\bX'||_{C^{1/2}_{t,\theta}(Q_t)}\leq C t^{-1/2}.
\end{equation}
\end{lemma}

\begin{proof}
Combining Proposition \ref{prop:H1.estimate}, Lemma \ref{H1.Linfinity.time.bound} and the embedding in Proposition \ref{besov.ineq.prop} gives us for any time $t/2\leq s\leq t$ that  
\begin{equation}\label{oneTwoTheta}
    ||\bX'(s)||_{C^{1/2}_\theta} \lesssim ||\bX'(s)||_{\dot{H}^1} \lesssim t^{-1/2}.
\end{equation}
Thus $\bX'$ is uniformly $C^{1/2}$ in $\theta$ on the time interval $[t/2, t]$.

To show H{\"o}lder continuity in time, let $t/2\leq s_1<s_2\leq t$, and $\theta\in \T.$  Fixing some $\alpha>0$ to be determined, by the $C^{1/2}_\theta$ estimate above we have for $i\in\{1,2\}$ that
\begin{equation}\label{point.minus.avg}
    \bigg| \bX'(s_i, \theta) - \frac{1}{2\alpha}\int_{-\alpha}^{\alpha} d\beta \bX'(s_i, \theta+\beta) \bigg| \lesssim \sqrt{\frac{\alpha}{t}}.
\end{equation}
Taking the difference of the two averages at times $s_1$ and $s_2$, we get that 
\begin{multline}\label{avg.difference}
        \bigg|\frac{1}{2\alpha}\int_{-\alpha}^{\alpha} d\beta \left( \bX'(s_2, \theta+\beta)-\bX'(s_1, \theta+\beta) \right)\bigg| 
        \\
        = \bigg|\frac{1}{2\alpha}\int_{-\alpha}^{\alpha} d\beta \int_{s_1}^{s_2}ds ~\partial_t \bX'(s, \theta+\beta)\bigg|.
\end{multline}
Applying Cauchy-Schwartz in the $d\beta$ integral to equation \eqref{avg.difference}, using Lemma \ref{time.space.equivalent} and \eqref{oneTwoTheta} we get that 
\begin{multline}\label{avg.bound}
 \bigg|\frac{1}{2\alpha}\int_{-\alpha}^{\alpha} d\beta \int_{s_1}^{s_2}ds ~\partial_t \bX'(s, \theta+\beta)\bigg| \lesssim  \frac{1}{\sqrt{\alpha }}\int_{s_1}^{s_2}ds~ || \partial_t \bX'(s)||_{L^2}
 \\
 \lesssim  \frac{1}{\sqrt{\alpha }}\int_{s_1}^{s_2}ds ~|| \bX'(s)||_{\dot{H}^1}
 \lesssim \frac{|s_2-s_1|}{\sqrt{\alpha t}}.
\end{multline}
Taking $\alpha = s_2-s_1>0$ and combining equations \eqref{point.minus.avg} and \eqref{avg.bound} then gives us
\begin{equation}\notag
    |\bX'(s_2,\theta) - \bX'(s_1,\theta)|\lesssim \frac{|s_1-s_2|^{1/2}}{t^{1/2}}.
\end{equation}
This completes the proof.
\end{proof}

\subsection{$C^{1,\alpha}$ estimate for $\bX'$}\label{sec:HigherEst}

With Lemma \ref{lem:chalf}, we have shown that our solution $\bX'$ is in $C^{1/2}$ in both time $t$ and the parametrization $\theta$.   Our next goal is to prove that $\bX'\in C^{1,\alpha}_{t,\theta}([\tau, T]\times\T ; \R^2)$ for any fixed $\tau>0$.

Our proof follows from the paper \cite{MR3656476}, where the authors prove regularity estimates for the (scalar) fractional porous medium equation
\begin{equation}\notag
    \partial_t u + (-\Delta)^{\sigma/2}\varphi(u) = 0.
\end{equation}
They make similar assumptions on their scalar nonlinearity $\varphi$ as we make on our tension map $\bT$, and their proof transfers over to our vector valued case.  

We shall go through the argument of \cite{MR3656476} and show that it applies.  But first, recall that $\bX'$ solves the equation 
\begin{equation}\notag
\partial_t \bX' + \TLam \bT(\bX') = \mathcal{V}(t,\theta),
\end{equation}
where $\mathcal{V}$ is defined in \eqref{v.theta.def}.  Thus we are dealing with a fractional porous media equation with an additional forcing term, so we shall need some estimates on $\mathcal{V}$.  

\begin{lemma}\label{lem.v.regularity}
Let $\mathcal{V}(t, \theta)$ be as in \eqref{v.theta.def}.  If $\bX'\in L^\infty_t \dot{H}^1_\theta \cap L^\infty_t \dot{H}^{\frac12}_\theta$, then 
\begin{equation}\notag
    \mathcal{V}(t,\theta)\in L^\infty_{t,\theta}.
\end{equation}
If $\bX'\in C^{\beta}_{t,\theta}$ for some $\displaystyle\frac{1}{2}<\beta<1,$ then 
\begin{equation}\notag
 \mathcal{V}(t,\theta)\in C^{2\beta-1}_{t,\theta}.
\end{equation}
If $\bX'\in C^{0,1}_{t, \theta}$, then $\mathcal{V}$ is log-Lipschitz.  Finally, if $\bX'\in C^{k, \beta}_{t, \theta}$ and $\TE\in C^{k,\beta}_{r}$ for some $k\geq 1,$ $0<\beta\leq 1$ then all $k$-th order derivatives of $\mathcal{V}$ are log-$C^\beta$.  
\end{lemma}

\begin{proof}
To prove the $L^\infty$ estimate, 
as in \eqref{Vbound.g} we bound
\begin{multline}\notag
    |\mathcal{V}(t,\theta)| \lesssim \DTT \int_\T\frac{d\alpha}{\alpha^2} \left(\frac{|\delta_\alpha \bX'|^2}{\rho} + \frac{|\delta_\alpha \bX'|^3}{\rho^2}\right) \lesssim \DTT \left(\frac{||\bX'(t)||_{\dot{B}^{1/2}_{\infty,2}}^2}{\rho} + \frac{||\bX'(t)||_{\dot{B}^{1/3}_{\infty,3}}^3}{\rho^2}\right)
   \\ \lesssim \DTT \left(1+ \frac{||\bX'||_{L^\infty_t H^{1/2}}}{\rho}\right) \frac{||\bX'||_{L^\infty_t H^1}^2}{\rho}.  
\end{multline}
With Proposition \ref{besov.ineq.prop},  we just used the following embedding and interpolation 
$||\bX'||_{\dot{B}^{1/3}_{\infty,3}} \lesssim ||\bX'||_{\dot{H}^{\frac56}}\lesssim ||\bX'||_{\dot{H}^{\frac12}}^{\frac{1}{3}} ||\bX'||_{\dot{H}^{1}}^{\frac{2}{3}}$.

Now assume that $\bX' \in C^{\beta}_{t,\theta}$ for some $1/2<\beta<1.$
Letting $\Theta = (t,\theta),$ and $\Phi = (s,\phi)$, we need to bound the difference of $ |\mathcal{V}(\Theta) -\mathcal{V}(\Phi)|.$
To begin, we split $\mathcal{A}$ from \eqref{kerbel.A.eqn.deriv} into two pieces $\mathcal{A}_L$ and  $\mathcal{A}_Q,$ where 
\begin{multline}\notag
  \mathcal{A}_L \eqdef   \frac{(\delta_\alpha^+ \bX'+\delta_\alpha^- \bX') \cdot \DO(\DATX) \DATX}{|\DATX|^2} \IO
\\ - \frac{(\delta_\alpha^+ \bX'+\delta_\alpha^- \bX') \cdot \RO(\DATX) \DATX}{|\DATX|^2} \RO(\DATX), 
\end{multline}
and 
\begin{multline}\notag
    \mathcal{A}_Q \eqdef \frac{\delta_\alpha^+ \bX' \cdot \DO(\DATX) \delta_\alpha^- \bX'}{|\DATX|^2}\IO
-\frac{\delta_\alpha^+ \bX' \cdot \RO(\DATX) \delta_\alpha^- \bX'}{|\DATX|^2}
\RO(\DATX)
\\ 
+ \frac{\delta_\alpha^+ \bX' \cdot (\DO(\DATX) - \IO) \delta_\alpha^- \bX'}{|\DATX|^2} \DO(\DATX).
\end{multline}
Correspondingly, we define $\mathcal{V}_L$ and $\mathcal{V_Q}$.  We will focus on proving that $\mathcal{V}_L$ is $C^{2\beta-1}$ when $\bX'$ is $C^\beta$.  Since $\bX'\in L^\infty \cap C^\beta$ then $\mathcal{A}_Q$ is   $\min\{|\alpha|^\beta, 1\}$ smoother than $\mathcal{A}_L$ so that the proof for $\mathcal{V}_Q$ follows similarly.

To show that $\mathcal{V}_L$ is $2\beta-1$ H\"older continuous, fix any $\Theta\neq \Phi\in [0,T]\times \T$
\begin{multline}\label{e.v.holder.int.bound}
    |\mathcal{V}_L(\Theta) - \mathcal{V}_L(\Phi)| \lesssim \frac{\DTT}{\rho} \int_\T d\alpha  \frac{|(\delta_\alpha^+ + \delta_\alpha^-)(\bX'(\Theta)-\bX'(\Phi)| \ |\delta_\alpha \bX'|}{\alpha^2} 
    \\+  \frac{\DTT}{\rho} \int_\T d\alpha  \frac{|(\delta_\alpha^+ + \delta_\alpha^-)\bX'| \ |\delta_\alpha \bX'(\Theta) - \delta_\alpha \bX'(\Phi)|}{\alpha^2}  
    \\+ \frac{\DTT}{\rho^2} \int_\T d\alpha  \frac{|(\delta_\alpha^+ + \delta_\alpha^-)\bX'| \ |\delta_\alpha \bX'|}{\alpha^2}  | \DAL ( \bX(\Theta) -  \bX(\Phi))| 
    \\+  \frac{\DDTT}{\rho} \int_\T d\alpha  \frac{|(\delta_\alpha^+ + \delta_\alpha^-)\bX'| \ |\delta_\alpha \bX'|}{\alpha^2} \left(|\bX'(\Theta) - \bX'(\Phi)|+|\tau_\alpha (\bX'(\Theta) - \bX'(\Phi)|\right)
    \\ =   \frac{\DTT}{\rho}I_1 +  \frac{\DTT}{\rho}I_2 + \frac{\DTT}{\rho^2} I_3 + \frac{\DDTT}{\rho} I_4.
\end{multline}
Note that above and below when we do not write the dependence on the variable $\Theta$ or $\Phi$ it is because it will not have an effect on the following argument.

As $\beta>1/2$, we can easily bound
\begin{equation}\notag
    \int_\T d\alpha  \frac{|(\delta_\alpha^+ + \delta_\alpha^-)\bX'| \ |\delta_\alpha \bX'|}{\alpha^2} \lesssim ||\bX'||_{C^\beta}^2 + ||\bX'||_{L^\infty}^2.
\end{equation}
Thus 
\begin{equation}\label{eqn.I34.bound}
    I_3+I_4 \lesssim (||\bX'||_{C^\beta}^2 + ||\bX'||_{L^\infty}^2) ||\bX'||_{C^\beta}|\Theta - \Phi|^\beta.
\end{equation}
To bound $I_1$ and $I_2$, we split each integral into the regions where $|\alpha|< |\Theta-\Phi|$ and $|\alpha|> |\Theta-\Phi|$.  For small $\alpha$, we use the bounds
\begin{equation}\notag
|\delta_\alpha \bX'|, |\delta_\alpha^\pm \bX'| \lesssim ||\bX'||_{C^\beta} |\alpha|^\beta,
\end{equation}
and for large $\alpha$ we bound
\begin{multline}\notag
    |(\delta_\alpha^+ + \delta_\alpha^-)(\bX'(\Theta)-\bX'(\Phi)| \ |\delta_\alpha \bX'| \lesssim ||\bX'||_{C^\beta}^2|\Theta-\Phi|^\beta |\alpha|^\beta,
    \\ |(\delta_\alpha^+ + \delta_\alpha^-)\bX'| \ |\delta_\alpha \bX'(\Theta) - \delta_\alpha \bX'(\Phi)|\lesssim ||\bX'||_{C^\beta}^2 |\Theta-\Phi|^\beta |\alpha|^\beta.
\end{multline}
Plugging in these bounds, we then get that 
\begin{equation}\label{eqn.I12.bound}
    I_1+I_2 \lesssim ||\bX'||_{C^\beta}^2 |\Theta-\Phi|^{2\beta-1}.
\end{equation}
As $2\beta-1 < \beta$, plugging in \eqref{eqn.I34.bound} and \eqref{eqn.I12.bound} into \eqref{e.v.holder.int.bound} gives us that $\mathcal{V}_L \in C^{2\beta-1}_{t,\theta}.$
The proof for $\mathcal{V}_Q$ follows similarly, giving the result for $\mathcal{V}$.  

Now suppose that $\bX'$ is Lipschitz.  Then again focusing on the $\mathcal{V}_L$ bound, we again are left to bound \eqref{e.v.holder.int.bound}.  As $\mathcal{V}$ is bounded, we may assume without loss of generality that $|\Theta-\Phi|\leq 1$.  We can bound $I_3,I_4$ using the same argument as the $1/2<\beta<1$ case to get 
\begin{equation}\label{eqn.I34.bound2}
    I_3, I_4 \lesssim (||\bX'||_{C^{0,1}}^2 + ||\bX'||_{L^\infty}^2) ||\bX'||_{C^{0,1}}|\Theta - \Phi|.
\end{equation}
To bound $I_1,I_2$ we now need to split our integral into 3 regions.  For $|\alpha|\leq |\Theta-\Phi|,$ we again use the bounds
\begin{equation}\label{eqn.I12.small.alpha}
    |\delta_\alpha \bX'|, |\delta_\alpha^\pm \bX'| \lesssim ||\bX'||_{C^{0,1}} |\alpha|. 
\end{equation}
For $|\Theta-\Phi|\leq |\alpha|\leq 1,$ we use the bounds 
\begin{multline}\label{eqn.I12.med.alpha}
    |(\delta_\alpha^+ + \delta_\alpha^-)(\bX'(\Theta)-\bX'(\Phi)| \ |\delta_\alpha \bX'|\lesssim ||\bX'||_{C^{0,1}}^2|\Theta-\Phi| \ |\alpha|,
    \\ |(\delta_\alpha^+ + \delta_\alpha^-)\bX'| \ |\delta_\alpha \bX'(\Theta) - \delta_\alpha \bX'(\Phi)|\lesssim ||\bX'||_{C^{0,1}}^2 |\Theta-\Phi| \ |\alpha|.
\end{multline}
And for $|\alpha|>1, $ we use 
\begin{multline}\label{eqn.I12.large.alpha}
    |(\delta_\alpha^+ + \delta_\alpha^-)(\bX'(\Theta)-\bX'(\Phi)| \ |\delta_\alpha \bX'|\lesssim ||\bX'||_{C^{0,1}} ||\bX'||_{L^\infty}|\Theta-\Phi| ,
    \\ 
    |(\delta_\alpha^+ + \delta_\alpha^-)\bX'| \ |\delta_\alpha \bX'(\Theta) - \delta_\alpha \bX'(\Phi)| \lesssim ||\bX'||_{C^{0,1}}||\bX'||_{L^\infty} |\Theta-\Phi| .
\end{multline}
Integrating and plugging in the above bounds \eqref{eqn.I12.small.alpha},  \eqref{eqn.I12.med.alpha} and \eqref{eqn.I12.large.alpha} we then get that 
\begin{equation}\label{eqn.I12.bound2}
    I_1+I_2 \lesssim ||\bX'||_{C^{0,1}}(||\bX'||_{C^{0,1}} + ||\bX'||_{L^\infty})(1-\log|\Theta-\Phi|))|\Theta-\Phi|.
\end{equation}
Plugging \eqref{eqn.I34.bound2}, \eqref{eqn.I12.bound2} into \eqref{e.v.holder.int.bound} gives us that $\mathcal{V}$ is log-Lipschitz.  

Now assume that $\bX'\in C_{t,\theta}^{k,\beta}$ and $\TE\in C^{k,\beta}_r$ for some $k\geq 1,$ and $0<\beta\leq 1.$  We claim that for every $0\leq j\leq k$ that $\partial_t^j \partial_\theta^{k-j}\mathcal{V}$ is log-$C^\beta$.  The difference of $|\partial_t^j \partial_\theta^{k-j}\mathcal{V}(\Theta) - \partial_t^j \partial_\theta^{k-j}\mathcal{V}(\Phi)|$ can be bounded by the sum of a number of integrals.  They can all be bounded similarly as above but for clarity we will directly show how to bound the two most difficult integrals, namely
\begin{multline}\notag
   J_1 = \int_\T d\alpha  \frac{|\partial_t^j \partial_\theta^{k-j}(\delta_\alpha^+ + \delta_\alpha^-)( \bX'(\Theta)-\bX'(\Phi)| \ |\delta_\alpha \bX'|}{\alpha^2}, 
   \\ J_2 = \int_\T d\alpha  \frac{|(\delta_\alpha^+ + \delta_\alpha^-) \bX'|}{\alpha^2}\bigg|\delta_\alpha D^k \bT(\bX'(\Theta)) -\delta_\alpha D^k \bT(\bX'(\Phi))\bigg| \ |\partial_t \bX'|^j  |\bX''|^{k-j}.
\end{multline}
Without loss of generality, we assume $|\Theta-\Phi|\leq 1$.  To bound $J_1$, we again split our integral into 3 regions.  For $|\alpha|\leq |\Theta-\Phi|$ we use the bound
\begin{equation}\notag
    |\partial_t^j \partial_\theta^{k-j}(\delta_\alpha^+ + \delta_\alpha^-)( \bX'(\Theta)-\bX'(\Phi)| \ |\delta_\alpha \bX'| \lesssim ||\bX'||_{C^{k,\beta}} ||\bX'||_{C^{0,1}} |\alpha|^{1+\beta}.
\end{equation}
For $|\Theta-\Phi|\leq |\alpha|\leq 1,$ we use the bounds
\begin{equation}\notag
    |\partial_t^j \partial_\theta^{k-j}(\delta_\alpha^+ + \delta_\alpha^-)( \bX'(\Theta)-\bX'(\Phi)| \ |\delta_\alpha \bX'| \lesssim ||\bX'||_{C^{k,\beta}} ||\bX'||_{C^{0,1}}|\Theta-\Phi|^\beta   |\alpha|.
\end{equation}
Finally for $|\alpha|>1$ we use 
\begin{equation}\notag
    |\partial_t^j \partial_\theta^{k-j}(\delta_\alpha^+ + \delta_\alpha^-)( \bX'(\Theta)-\bX'(\Phi)| \ |\delta_\alpha \bX'| \lesssim ||\bX'||_{C^{k,\beta}} ||\bX'||_{L^\infty}|\Theta-\Phi|^\beta .
\end{equation}
Plugging these in, we get that 
\begin{equation}\notag
    J_1 \lesssim ||\bX'||_{C^{k,\beta}}(||\bX'||_{C^{0,1}}+||\bX'||_{L^\infty})(1-\log|\Theta-\Phi|) |\Theta-\Phi|^\beta.
\end{equation}
The other important integral to bound is $J_2$.  Note that 
\begin{equation}\notag
    |\partial_t \bX'|^j |\bX''|^{k-j}\leq ||\bX'||_{C^{0,1}}^k.
\end{equation}
To bound the rest of $J_2$, we split the integral into the same 3 regions for $\alpha$. Using the 3 bounds
\begin{multline}\notag
    |(\delta_\alpha^+ + \delta_\alpha^-) \bX'| \ |\delta_\alpha D^k\bT(\Theta)- \delta_\alpha D^k\bT(\Phi)| \lesssim ||\bT||_{C^{k,\beta}} ||\bX'||_{C^{0,1}}^2 |\alpha|^{1+\beta}, 
    \\ |(\delta_\alpha^+ + \delta_\alpha^-) \bX'| \ |\delta_\alpha D^k\bT(\Theta)- \delta_\alpha D^k\bT(\Phi)| \lesssim ||\bT||_{C^{k,\beta}} ||\bX'||_{C^{0,1}}^2 |\Theta-\Phi|^\beta |\alpha|,
    \\ |(\delta_\alpha^+ + \delta_\alpha^-) \bX'| \ |\delta_\alpha D^k\bT(\Theta)- \delta_\alpha D^k\bT(\Phi)| \lesssim ||\bT||_{C^{k,\beta}} ||\bX'||_{C^{0,1}} ||\bX'||_{L^\infty}  |\Theta-\Phi|^\beta,
\end{multline}
for small, medium, and large $\alpha$ respectively.  Plugging these in, we then get that 
\begin{equation}\notag
    J_2 \lesssim ||\bT||_{C^{k,\beta}}||\bX'||_{C^{0,1}}^{k+1}(||\bX'||_{C^{0,1}}+||\bX'||_{L^\infty})(1-\log |\Theta-\Phi|) |\Theta-\Phi|^\beta.
\end{equation}
All the other integrals involved in bounding $|\partial_t^j \partial_\theta^{k-j}\mathcal{V}(\Theta)-\partial_t^j \partial_\theta^{k-j}\mathcal{V}(\Phi)|$ can be bounded either following similar arguments, or by using only lower order norms.  Thus all $k$-th order derivatives of $\mathcal{V}$ are log-$C^\beta.$
\end{proof}

With the regularity estimates for $\mathcal{V}$, we can now slightly modify \cite{MR3656476}'s proof of regularity for the scalar fractional porous medium equation.  The crux of their argument is an a priori estimate for solutions to the fractional heat equation.

\begin{lemma}\label{lem.heat.eqn.estimates}(V\'{a}zquez, de Pablo, Quir\'{o}s and Rodr\'{i}guez \cite{MR3656476})
Let $f,g: [0,T]\times \R \to \R$ be such that 
\begin{equation}\notag
    \left\{\begin{array}{l} \partial_t g + \Lambda g = \Lambda f,
    \\ g(0,\cdot) \equiv 0 . \end{array}\right.
\end{equation}
Fix $\Theta_0 = (t_0, \theta_0)\in (0,T)\times \R.$  Suppose that there exist some $0<\beta, \epsilon< 1$ and $r_0>0$ such that $f$ satisfies 
\begin{equation}\notag
    \left\{\begin{array}{l}
         |f(\Theta_1)-f(\Theta_0)| \leq c |\Theta_1 - \Theta_0|^{\beta+\epsilon}, 
         \\ |f(\Theta_1) - f(\Theta_2)| \leq c r^\epsilon |\Theta_1-\Theta_2|^\beta ,
    \end{array}\right.
\end{equation}
for all $\Theta_1, \Theta_2\in B_r(\Theta_0)=\{\Theta: |\Theta - \Theta_0|<r\}$ and $0<r\leq r_0.$

Then $g$ satisfies 
\begin{equation}\notag
    |g(\Theta_0+\Phi) + g(\Theta_0-\Phi) - 2g(\Theta_0)| \lesssim |\Phi|^{\beta+\epsilon}.
\end{equation}
\end{lemma}

Note that Lemma \ref{lem.heat.eqn.estimates} above is a collection of Lemmas 4.1, 5.1, and 5.3 from \cite{MR3656476}.  Lemma \ref{lem.heat.eqn.estimates} effectively says that if $f$ is $C^{\beta}$ everywhere and $C^{\beta+\epsilon}$ at a fixed point $\Theta_0,$ then so is the solution $g$.  We also remark that Lemma \ref{lem.heat.eqn.estimates} generalizes automatically from $\R$ to $\T$.  Also as in \secref{sec:para} then  Lemma \ref{lem.heat.eqn.estimates} generalizes automatically from $\partial_t g + \Lambda g = \Lambda f$ to $\partial_t g + \TLam g = \TLam f$.

As in  \cite{MR3656476}, we apply Lemma \ref{lem.heat.eqn.estimates} repeatedly to steadily improve the regularity of our solution $\bX'$ in a bootstrapping argument.  

\begin{proposition}\label{prop.C1beta}
Let $\bX: [0,T]\times \T\to \R^2$ be the solution to the Peskin problem we constructed.  Then for any $0<\tau<T$, $\bX' \in C^{1,\beta}_{t,\theta}([\tau, T]\times \T; \R^2)$ for all $0<\beta<1.$  
\end{proposition}

\begin{proof}
To begin, fix some point $\Theta_0 = (t_0,\theta_0)\in (\tau, T)\times \T.$  Let $\bV$ be the solution to the equation 
\begin{equation}\label{eqn.v.pde}
   \left\{ \begin{array}{l}
   \partial_t \bV + D\bT(\bX'(\Theta_0)) \TLam \bV = -\mathcal{V}[\bX'],
   \\ \bV(\tau/2,\cdot) = \bX'(\tau/2, \cdot).
    \end{array} \right.
\end{equation}
Together Proposition \ref{prop:H1.estimate} and Lemma \ref{H1.Linfinity.time.bound} imply that $\bX'\in L^\infty_t([\tau/2, T]; H^1(\T;\R^2)).$  Thus in particular, by Lemma \ref{lem.v.regularity}, $\mathcal{V}\in L^\infty([\tau/2, T]\times\T).$  Notice that \eqref{eqn.v.pde} can be diagonalized using $\bV\cdot \widehat{\bX'}(\Theta_0)$ and  $\bV\cdot \widehat{\bX'}(\Theta_0)^\perp$.  Then since $\bV$ is a solution to the fractional heat equation with bounded initial data and bounded forcing term, we thus have for any $0<\beta<1$ that 
\begin{equation}\label{eqn.v.low.reg}
    \bV\in C^{\beta}([\tau,T]\times \T),
\end{equation}
with the constant depending on $\tau, \beta, ||\bX'||_{L^\infty}, ||\mathcal{V}||_{L^\infty}, \TLam,$ and  $\DTT$.  
Now take $\bU(t,\theta) \eqdef \bX'(t,\theta) - \bV(t,\theta).$  Then using \eqref{v.theta.def} we see that $\bU$ solves the system 
\begin{equation}\notag
    \left\{\begin{array}{l} \partial_t \bU + D\bT(\bX'(\Theta_0)) \TLam \bU = \TLam \bF,
    \\ \bU(\tau/2,\cdot)\equiv 0, \end{array}\right.
\end{equation}
where 
\begin{equation}\notag
    \bF(t,\theta) = D\bT(\bX'(\Theta_0))\bX'(t,\theta) - \bT(\bX'(t,\theta)).  
\end{equation}
Note that $\bF$ satisfies 
\begin{multline}\label{eqn.F1}
    |\bF(\Theta_1) - \bF(\Theta_0)| 
    \\ = |\bT(\bX'(\Theta_1)) - \bT(\bX'(\Theta_0)) - D\bT(\bX'(\Theta_0))(\bX'(\Theta_1) - \bX'(\Theta_0)|
    \\ \leq \DDTT |\bX'(\Theta_1) - \bX'(\Theta_0)|^2,
\end{multline}
and 
\begin{multline}\label{eqn.F2}
    |\bF(\Theta_1) - \bF(\Theta_2)| 
    \\ = |\bT(\bX'(\Theta_1)) - \bT(\bX'(\Theta_2)) - D\bT(\bX'(\Theta_0))(\bX'(\Theta_1) - \bX'(\Theta_2)|
    \\ = \bigg|\left(\int_0^1 ds D\bT(s \bX'(\Theta_1) + (1-s)\bX'(\Theta_2)) - D\bT(\bX'(\Theta_0))\right)(\bX'(\Theta_1)-\bX'(\Theta_2))\bigg| 
    \\ \leq \DDTT \max\{|\bX'(\Theta_1) - \bX'(\Theta_0)|,|\bX'(\Theta_2) - \bX'(\Theta_0)|\} |\bX'(\Theta_1)-\bX'(\Theta_2)|.
\end{multline}
We will use \eqref{eqn.F1} and \eqref{eqn.F2} to apply the bounds in Lemma \ref{lem.heat.eqn.estimates}.

Let $\bU_1 = \bU\cdot \widehat{\bX'}(\Theta_0)$ and $\bU_2 = \bU\cdot \widehat{\bX'}(\Theta_0)^\perp.$  Then using \eqref{e:DTdefn} we see that $\BU_1$ solves the scalar equation 
\begin{equation}\notag
    \left\{\begin{array}{l} \partial_t \bU_1 +\TE'(|\BX'(\Theta_0)|)  \TLam \bU_1 = \TLam (\bF\cdot \widehat{\BX'}(\Theta_0)),
    \\ \bU_1(\tau/2,\cdot)\equiv 0, \end{array}\right.
\end{equation}
and $\bU_2$ solves 
\begin{equation}\notag
    \left\{\begin{array}{l} \partial_t \bU_2 +\frac{\TE(|\BX'(\Theta_0)|)}{|\BX'(\Theta_0)|}  \TLam \bU_2 = \TLam (\bF\cdot \widehat{\BX'}(\Theta_0)^\perp),
    \\ \bU_2(\tau/2,\cdot)\equiv 0, \end{array}\right.
\end{equation}
Note that from \eqref{e:DTdefn} and \eqref{e:QuantitativeTensionMap} we have
$$\lambda\leq \TE'(|\bX'(\Theta_0)|), \frac{\TE(|\BX'(\Theta_0)|)}{|\BX'(\Theta_0)|} \leq \DTT.$$
As $\bF$ satisfies \eqref{eqn.F1}, \eqref{eqn.F2} and $\bX'\in C^{1/2}_{t,\theta}$ by Lemma \ref{lem:chalf}, after rescaling in time we can apply Lemma \ref{lem.heat.eqn.estimates} to $\bU_i
$ with $\beta = \epsilon=1/2 $ to get 
\begin{equation}\notag
|\bU(\Theta_0+\Phi)+\bU(\Theta_0-\Phi)-\bU(\Theta_0)| \lesssim |\Phi|,
\end{equation}
where the constant depends on $||\bX'||_{C^{1/2}},  \lambda, \DTT,$ and $\DDTT.$  In particular, we have that $\bU$ is $C^{\beta}$ at $\Theta_0$ for any $\beta<1$.  As $\bX' = \BU+\bV$, we thus have for any $\beta<1$ that 
\begin{equation}\notag
    |\bX'(\Theta_0+\Phi)-\bX'(\Theta_0)|\lesssim |\Phi|^\beta.
\end{equation}
Since $\Theta_0\in [\tau, T]\times \T$ was arbitrary, we thus have that $\bX'\in C^\beta([\tau, T]\times \T; \R^2)$ for all $0<\beta<1.$  

But now as $\bX'\in C^\beta$ for all $\beta<1,$ by Lemma \ref{lem.v.regularity} we have that $\mathcal{V}$ is $C^\beta$ for all $\beta<1$ as well.  Thus as $\bV$ solves \eqref{eqn.v.pde} with a $C^\beta$ forcing term, we must have $\bV\in C^{1,\beta}([\tau,T]\times \T)$ for any $\beta<1.$  

As $\bF$ satisfies \eqref{eqn.F1}, \eqref{eqn.F2} and $\bX'\in C^{\beta}_{t,\theta}$, after rescaling in time we can again apply Lemma \ref{lem.heat.eqn.estimates} to $\bU_i
$ with for any $\epsilon=\beta<1$ to get 
\begin{equation}\notag
|\bU(\Theta_0+\Phi)+\bU(\Theta_0-\Phi)-\bU(\Theta_0)| \lesssim |\Phi|^{2\beta}.
\end{equation}
Since $\bX' = \bU+\bV$ and $\Theta_0\in [\tau, T]\times \T$ and $\beta<1$ were arbitrary, we thus have that $\bX'\in C^{1,\beta}([\tau, T]\times \T)$ for all $\beta<1.$
\end{proof}

\begin{proposition}\label{prop:higherReg}
Assume that $\TE\in C^{k,\gamma}([0,\infty))$ for some $k\geq 2, $ and $0<\gamma<1.$  Then for any $\tau>0$, $\bX'\in C^{k,\gamma}([\tau, T]\times \T; \R^2)$.  
\end{proposition}
\begin{proof}
If $k=2$, we will show $\bX'\in C^{2,\gamma}.$  Else, we will show that $\bX'\in C^{2,\beta}$ for all $\beta<1$ and then proceed by induction on $k$.  

So to begin, we will prove that $\bX''\in C^{1,\gamma}.$  Differentiating our equation for $\bX'$ \eqref{v.theta.def}, we get that 
\begin{equation}\label{eqn.X''}
    \partial_t \bX'' + \TLam (D\bT(\bX')\bX'') = \mathcal{V}'.
\end{equation}
Fix some point $\Theta_0\in [\tau, T)\times \T.$  Then we can rewrite \eqref{eqn.X''} as 
\begin{multline}\label{eqn.X''.theta0}
    \partial_t \BX'' + D\bT(\BX'(\Theta_0))\TLam \bX'' = \mathcal{V}' - \bX''(\Theta_0)\TLam D\bT(\bX') 
    \\- \TLam \left[(D\bT(\bX')-D\bT(\bX'(\Theta_0))(\bX''-\bX''(\Theta_0))\right].
\end{multline}
As in the proof of Proposition \ref{prop.C1beta} again take $\bV$ to be the solution to 
\begin{equation}\label{eqn.v.pde2}
   \left\{ \begin{array}{l}
   \partial_t \bV + D\bT(\bX'(\Theta_0)) \TLam \bV = \mathcal{V}'[\bX']- \bX''(\Theta_0)\TLam D\bT(\bX') ,
   \\ \bV(\tau/2,\cdot) = \bX''(\tau/2, \cdot).
    \end{array} \right.
\end{equation}
By Proposition \ref{prop.C1beta} and Lemma \ref{lem.v.regularity} we have that $\mathcal{V}'\in C^{\beta}$ for all $\beta<1$.  If $k>2$, then  $\TLam D\bT(\bX')$ is $C^{\beta}$ for all $\beta<1$, and if $k=2$ then $\TLam D\bT(\bX')$ is $C^\gamma$.  Thus $\bV\in C^{1,\beta}([\tau, T]\times \T)$ for all $\beta<1$ if $k>2$ and  $\bV\in C^{1,\gamma}([\tau, T]\times \T)$ if $k=2$.  
Taking $\bU = \bX''-\bV$, subtracting \eqref{eqn.v.pde2} from \eqref{eqn.X''.theta0} gives us that $\bU$ solves 
\begin{equation}\notag
    \left\{\begin{array}{l} \partial_t \bU + D\bT(\bX'(\Theta_0)) \TLam \bU = \TLam \bF,
    \\ \bU(\tau/2,\cdot)\equiv 0, \end{array}\right.
\end{equation}
where 
\begin{multline}\notag
\bF(\Theta) = (D\bT(\bX'(\Theta))-D\bT(\bX'(\Theta_0)))(\bX''(\Theta)-\bX''(\Theta_0)) 
\\ = O(|\bX''(\Theta)-\bX''(\Theta_0)|^2) = O(|\Theta-\Theta_0|^{2\beta}),
\end{multline}
for all $\beta<1.$  Using Lemma \ref{lem.heat.eqn.estimates} and following the same argument as in Proposition \ref{prop.C1beta}, we then get that $\BU$ is $C^{2\beta}$ at $\Theta_0$.  If $k=2,$ then we get that $\bX'' = \bU+\bV$ is $C^{1,\gamma}.$ And if $k\geq 2,$ then $\bX''$ is $C^{1,\beta}$ for all $\beta<1.$  A symmetric argument works for $\partial_t \bX',$ so we get that $\bX'\in C^{2,\beta}$ for all $\beta<1$ if $k>2,$ and $\bX'\in C^{2,\gamma}$ if $k=2$.  

We now proceed by induction.  Suppose that we have proven that $\bX'\in C^{j,\beta} $ for all $\beta<1$ for some $j<k.$ Let $\partial^j = \partial_t^{l}\partial_\theta^{j-l}$ for some $0\leq l\leq j$ be some $j$-th order derivative.  Then for any $\Theta_0\in [\tau, T]\times \T$, similar to \eqref{eqn.X''.theta0} we can write the equation for $\partial^j \bX'$ as
\begin{multline}\notag
    \partial_t (\partial^j\bX') + D\bT(\bX'(\Theta_0)) \TLam \partial^j \bX' = \partial^j \mathcal{V} - \TLam(\partial^j \bT(\bX') - D\bT(\bX')\partial^j\bX')
    \\ -\partial^j \bX'(\Theta_0)\TLam D\bT(\bX')
     -\TLam \left[(D\bT(\bX') - D\bT(\bX'(\Theta_0))(\partial^j\bX' - \partial^j\bX'(\Theta_0)\right].
\end{multline}
Then Lemma \ref{lem.v.regularity} $\partial^j \mathcal{V}$ is $C^{\beta}$ for all $\beta<1$.  Since $k>2$, $\TLam D\bT(\bX')$ is also $C^{\beta}$ for all $\beta<1$.  Finally, $\TLam (\partial^j \bT(\bX') - D\bT(\bX')\partial^j\bX')$ is either $C^{\beta}$ for all $\beta<1$ if $k>j+1$, or its $C^\gamma$ if $k=j+1$.  

Thus by defining $\bV, \bU, $and $\bF$ analogously, we can follow the same proof scheme as in Proposition \ref{prop.C1beta} and get that $\partial^j \bX'$ is $C^{1,\beta} $ for all $\beta<1$ if $k>j+1$ or  $\partial^j \bX'$ is $C^{1,\gamma}$ if $k=j+1$.   Thus by induction, we have that $\bX'\in C^{k,\gamma}$ if $\TE\in C^{k,\gamma}.$
\end{proof}

\section{Proof of the main theorem}\label{sec:mainThmProof}

In this section we will collect the previous a priori estimates to explain the proofs of our main theorems from \secref{sec:mainResults}.  We will use an approximation argument starting with the existence and uniqueness theorem for general tension from \cite{rodenberg_thesis}:

\begin{theorem}\label{rodenberg.thm} \cite[Theorem 1.2.9 on page 17]{rodenberg_thesis}.
From \eqref{tension.map.def} we suppose the tension $\TE: [0,\infty) \to [0,\infty)$ satisfies $\TE(s) \in h^{1,\gamma}(0,\infty)$, for any fixed $0 < \gamma < 1$, is such that both $\TE(s)>0$ and $\TE'(s)>0$.  Consider the fully nonlinear Peskin problem \eqref{e:boundaryintegral} and \eqref{stokeslet.def} with initial data $\BX_0 \in h^{1,\gamma}(\T)$ with $|\bX_0|_*>0$.  (a) Then there exists $T>0$ such that  \eqref{e:boundaryintegral} and \eqref{stokeslet.def} has a unique solution $\BX(t) \in C([0,T]; h^{1,\gamma}(\T))\cap C^1([0,T]; h^{0,\gamma}(\T))$. (b) There exists some $\varepsilon>0$ such that if $\BY_0 \in h^{1,\gamma}(\T)$ with $||\BX_0 - \BY_0 ||_{h^{1,\gamma}}<\varepsilon$ then \eqref{e:boundaryintegral} and \eqref{stokeslet.def} has a unique solution $\BY(t; \BY_0)\in C([0,T]; h^{1,\gamma}(\T))\cap C^1([0,T]; h^{0,\gamma}(\T))$ corresponding to the initial data $\BY_0$ where $T>0$ is the same as in statement (a).
\end{theorem}

In Theorem \ref{rodenberg.thm} recall that the little H{\"o}lder spaces $h^{1,\gamma}$ are the completion of $C^\infty$ in the $C^{1,\gamma}$ norm and that $C^{1,\alpha} \subset h^{1,\gamma}$ whenever $\alpha > \gamma >0$.  We refer to \cite{MR3935476,rodenberg_thesis} and the references therin for further discussion of the little H{\"o}lder spaces.

Now let $\bX_0'\in \dot{B}^{\frac{1}{2}}_{2,1}(\T; \R^2)$ with $|\bX_0|_*>0$.  We choose $\rho$ such that $\bX_0$ satisfies $$|\bX_0|_*\geq 3\rho>0.$$  Then by Lemma \ref{lem:VallePoisson} there is some function $\mu$ satisfying the conditions of Definition \ref{subw.definition} and a constant $M>0$ such that $$||\bX'_0||_{\dot{B}^{\frac{1}{2}, \mu}_{2,1}}\leq M <\infty .$$   Let the scalar tension $\TE: [0,\infty) \to [0,\infty)$ satisfy the strong bounds \eqref{e:QuantitativeScalarTension}, \eqref{e:DTdefn} and \eqref{e:QuantitativeTensionMap}.

Next we define the following approximations
\begin{equation*}
    \bX_{0,n}(\alpha)
    =
    \sum_{|k| \le n} \widehat{\bX}_0(k) e^{ik\alpha}, \quad n \ge 1.
\end{equation*}
Above, for $k\in \Z$, we define the standard Fourier transform on $\T$ as
\begin{equation}\notag
\mathcal{F}_{\T}(f)(k) = \widehat{f}(k) \eqdef \frac{1}{2\pi}\int_{\T}   f(\alpha)e^{-ik\alpha} d\alpha.
\end{equation}
Then $\bX_{0,n}(\alpha)$ is smooth, $|| \bX_{0,n}' ||_{\dot{B}^{\frac{1}{2},\mu}_{2,1}} \leq || \bX_{0}' ||_{\dot{B}^{\frac{1}{2},\mu}_{2,1}} \leq M$ for all $n$, and  we have  
\begin{equation}\notag
 \bX_{0,n}' \to \bX_0'  \mbox{ as }  n\to\infty \mbox{ in } \dot{B}^{\frac{1}{2}}_{2,1} \cap \dot{B}^{\frac{1}{2},\mu}_{2,1}.
\end{equation}
Since $\dot{B}^{\frac12}_{2,1}$ controls the $L^\infty$ norm as in Lemma \ref{Besov.embedding},
we have 
\begin{equation}\label{bounded.embed.12}
    ||f||_{L^\infty}\lesssim  \int_{\T} \frac{||\delta_\beta f||_{L^2_\theta}}{|\beta|^{3/2}} d\beta \approx ||f||_{\dot{B}^{\frac12}_{2,1}}.  
\end{equation}
Using this estimate and \eqref{chord.arc.upper} then as $n\to\infty$ we also have 
\begin{equation}\notag
    \left||\bX_0|_* - |\bX_{0,n}|_* \right|
\lesssim ||\bX_0'-\bX_{0,n}'||_{\dot{B}^{\frac12}_{2,1}} \to 0.
\end{equation}
We conclude in particular that $|\bX_{0,n}|_{*} \geq |\bX_{0}|_{*} + o(1)$.  Therefore, for any small $\varepsilon>0$ there is $1 \le N_\varepsilon<\infty$ such that $|\bX_{0,n}|_{*} \geq |\bX_{0}|_{*} -\varepsilon>0$ for all $n\ge N_\varepsilon$.  Since we will be taking the limit as $n\to\infty$, without loss of generality we can take $N_\varepsilon=1$ by throwing away the first $N_\varepsilon$ terms in the sequence and relabelling.  Specifically we choose $\varepsilon = \rho$ and then we have $$|\bX_{0,n}|_{*} \geq 2\rho>0$$ uniformly.  We also have $\bX_{0,n} \in h^{1,\gamma}(\T)$ for all $n \ge 1$ and any $0 < \gamma < 1$.

Then using the result in \cite{rodenberg_thesis}, as stated above in Theorem \ref{rodenberg.thm}, we have that there exists a unique solution 
\begin{equation}\notag
\bX_n(t,\theta)
\in C([0,T_{\text{max}}]; h^{1,\frac12}(\T))\cap C^1([0,T_{\text{max}}]; h^{0,\frac12}(\T))
\end{equation}
to the fully nonlinear Peskin problem \eqref{e:boundaryintegral}  and \eqref{stokeslet.def} with tension $\TE$ for some time $T_{\text{max}}>0$.  Notice that over $0<t<T_{\text{max}}$ the solution to \eqref{e:boundaryintegral}  and \eqref{stokeslet.def} in $C([0,T_{\text{max}}]; h^{1,\frac12}(\T))\cap C^1([0,T_{\text{max}}]; h^{0,\frac12}(\T))$ has enough regularity to be a weak solution the equation \eqref{peskin.general.tension} with kernel \eqref{kerbel.eqn.deriv} in the sense of Definition \ref{def:solution}.  If $T_{\text{max}}<\infty$  then either 
\begin{equation}\notag
    \liminf\limits_{t\to T_{\text{max}}} |\bX_n(t)|_{*} = 0, 
\end{equation}
or 
\begin{equation}\notag
     \limsup\limits_{t\to T_{\text{max}}} ||\bX'_n||_{C^{1/2}_\theta}(t) = \infty.
\end{equation}
We will show that our estimates imply that this can not happen over a uniform time interval that is independent of $n$.

To this end, since the tension $\TE$ satisfies \eqref{e:QuantitativeScalarTension} and \eqref{e:QuantitativeTensionMap} then our previous a priori estimates apply.  Next let $T_M^*$  be defined by 
\begin{equation}\label{e:tMdefn}
    T_M^* = \inf\left\{ T>0: ||\bX'_n||_{\CM} +2c\lambda^{1/2}||\bX'_n||_{\DM}> 5M\right\},
\end{equation}
where $c$ and $\lambda^{1/2}$ are the constants in Proposition \ref{prop:general.apriori.final.local}.  
We further define $T_\rho^*$ by
\begin{equation}\label{e:trhodefn}
 T_\rho^* = \inf\left\{ t>0: |\bX_n(t)|_* < \rho \right\}.
\end{equation}
We then take the time $T^*$ to be the minimum of the two
\begin{equation}\label{e:t*defn}
    T^* = \min\{T_M^*, T_\rho^*\}.
\end{equation}
Since under our assumptions the norms $||\bX'_n||_{\CM}$ and  $||\bX'_n||_{\DM}$ are continuous in $T>0$ then we have $T^*>0$.  We will estimate this time $T^*$ from below in terms of $M$, $\subw$ and $\rho$.  We will show that $T^*$ can be taken independent of $n$ and  $T_{\text{max}}\geq T^*$.  

We then estimate $\BX'_n(t)$ on the time interval $[0,T^*]$.  We shall first consider the case that $T_M^*\leq T_\rho^*$, and get a lower bound on $T_M^*$ using Proposition \ref{prop:general.apriori.final.local}.  For $0\leq t\leq T^*$ under \eqref{e:tMdefn}, \eqref{e:trhodefn} and \eqref{e:t*defn} we have that 
\begin{equation}\label{bounds.T.star}
||\BX'_n(t)||_{\dot{B}^{\frac12,\subw}_{2,1}} \leq 5M, 
\end{equation}
and
\begin{equation}\label{bounds.rho.T.star}
| \BX_n(t)|_*\geq \rho>0.
\end{equation}
Then for $\BPU=\BPU[M,\rho,\lambda, \DDTT, \DTT]$ as defined in \eqref{BU.def} with \eqref{BB.const.def}, \eqref{BK.const.def} and \eqref{BC.const.def}, we obtain from Proposition \ref{prop:general.apriori.final.local} for $0 < T_M$ sufficiently small that
\begin{equation}\notag
CT^{1/2}_M
\BPU^{1/2}[M,\rho,\lambda, \DDTT, \DTT]
\leq 
\frac{1}{2}.
\end{equation}
Thus we can plug this back into Proposition \ref{prop:general.apriori.final.local} to obtain 
\begin{equation}\notag
\frac{1}{2}||\bX'_n||_{\CM} + c\lambda^{1/2}||\bX'_n||_{\DM}
    \\
    \leq
    2 || \BX'_0||_{\CM}
\leq 2M,
\end{equation}
which holds for all $T\in [0,T_M]$.  Thus $T_M^* \ge T_M > 0$ uniformly in $n$.

Next, suppose that $T_M^*\geq T_\rho^*$.  Let $T^*$ be as defined in \eqref{e:t*defn}, then we have \eqref{bounds.T.star} and \eqref{bounds.rho.T.star} over $0 \le t \le T^*$, and we also have \eqref{bounded.embed.12}.  Thus for a fixed $\eta>0$ to be chosen sufficiently small, then 
breaking up the integral on the right-hand side of \eqref{bounded.embed.12} into $|\beta|<\eta$ and $|\beta|>\eta$, we get in general that
\begin{equation}\notag
\begin{split}
    \int_{\T} \frac{||\delta_\beta f||_{L^2}}{|\beta|^{3/2}} d\beta &= \int_{|\beta|<\eta} \frac{||\delta_\beta f||_{L^2}}{|\beta|^{3/2}} d\beta+\int_{|\beta|>\eta} \frac{||\delta_\beta f||_{L^2}}{|\beta|^{3/2}} d\beta
    \\&\leq \frac{1}{\mu(\eta^{-1})} \int_{\T} \frac{||\delta_\beta f||_{L^2}}{|\beta|^{3/2}} \subw(|\beta|^{-1})d\beta + \frac{4}{\eta^{1/2}}||f||_{L^2}.
\end{split}
\end{equation}
Since $\mu(\eta^{-1})^{-1} \to 0$ as $\eta \to 0$, then under \eqref{bounds.T.star} for any $\varepsilon>0$, we can choose $\eta = \eta(\mu, M, \varepsilon)>0$ such that 
\begin{equation}\notag
    ||\BX'_n(t)-\BX'_{0,n}||_{L^\infty}\leq \frac{\varepsilon}{2}  +\frac{C}{\eta^{1/2}}||\BX'_n(t)-\BX'_{0,n}||_{L^2},
\end{equation}
for some universal constant $C>0$.  Thus, it remains to control the continuity of $\BX'_n(t)$ in $L^2$.  Then from Corollary \ref{prop:time.estimate} over $0 \le  t \le T_\rho$ uniformly in $n$ we have 
\begin{equation}\notag
    ||\BX'_n(t)-\BX'_{0,n}||_{L^2} = \bigg|\bigg| \int_0^t \partial_t \BX'_n(s) ds \bigg|\bigg|_{L^2} \leq \int_0^t ||\partial_t \BX'_n(s)||_{L^2} ds \leq C_2 T_\rho^{1/2}.
\end{equation}
Then by taking $T_\rho$ sufficiently small for some time $T_\rho = T_\rho(\mu, M, \rho, \varepsilon)>0$ and using  using  \eqref{chord.arc.upper}
we can guarantee that
\begin{equation}\label{arccord.bound.ep}
    \left||\BX_n(t)|_* - |\BX_{0,n}|_* \right|
\leq 
    ||\BX'_n(t)-\BX'_{0,n}||_{L^\infty}\leq \varepsilon, \quad 0 \le t \le T_\rho.
\end{equation}
In particular, taking $\varepsilon = \rho$ we can guarantee that \eqref{bounds.rho.T.star} holds over $0 \le  t \le T_\rho$ uniformly in $n$. Thus $T_\rho^* \ge T_\rho >0$  uniformly in $n$.  

In particular then \eqref{bounds.rho.T.star} and \eqref{e:trhodefn} imply that $T_{\text{max}}>T^*_\rho$.  Next we consider $T^*_M>0$ defined in \eqref{e:tMdefn}.   By Lemma \ref{H1.Linfinity.time.bound} we have  for any small $t_0>0$ that 
$||\bX'_n(t_0)||_{\dot{H}^1} <\infty$ uniformly in $n$.  Then by \eqref{H1.estimate2} for any $t_0<t < T^*_M$ we have $||\bX'_n(t)||_{\dot{H}^1} \lesssim ||\bX'_n(t_0)||_{\dot{H}^1}$.  Further from Proposition \ref{Besov.equivalence.prop} and then Proposition \ref{besov.ineq.prop} we have 
$$||\bX'_n(t)||_{\dot{C}^{\frac12}_\theta} \approx ||\bX'_n(t)||_{\dot{B}^{\frac12}_{\infty, \infty}} 
\lesssim ||\bX'_n(t)||_{\dot{B}^{1}_{2, \infty}} 
\lesssim ||\bX'_n(t)||_{\dot{H}^{1}} \lesssim ||\bX'_n(t_0)||_{\dot{H}^1}.$$
Then using \eqref{bounded.embed.12} and \eqref{e:tMdefn} we have $||\bX'_n(t)||_{L^\infty_\theta} \leq C ||\bX'_n(t)||_{\dot{B}^{\frac12}_{2,1}}\leq 5 C M$ uniformly in $n$ over $0<t<T^*_M$ .  We conclude that $T_{\text{max}} \ge T^*_M$.  Thus   $T_{\text{max}} \ge T^*>0$ uniformly in $n$.

Thus our sequence of solutions $\bX_n(t)$ are all defined uniformly in $n$ on the interval $[0,T^*]$.  They also satisfy the uniform bounds 
\begin{equation}\label{e:uniformbounds}
\begin{split}
   \begin{array}{rl} ||\bX'_n||_{\CM} + c\lambda^{1/2}||\bX'_n||_{\DM}& \leq 5CM,
    \\ \inf_{0<t<T} |\bX_n(t)|_*& \geq \rho, 
    \\ ||\bX_n||_{C_{t,\theta}^{2,\beta}([\tau, T]\times \T)} &\leq C(M, \mu, \lambda, \DTT, \DDTT, \tau, \beta), \quad \forall 0<\beta< 1,
    \end{array}
\end{split}
\end{equation}
where the last bounds follow by Proposition \ref{prop.C1beta}.  
After passing to a subsequence, we then have that the sequence converges strongly in $L^\infty_t \dot{B}^{3/2}_{2,1}\cap L^2_t H^2_\theta\cap C^{2,\beta}_{loc}((0,T]\times \T)$ to a limit $\bX(t)$ satisfying the same bounds in \eqref{e:uniformbounds}.  Thus $\BX(t)$ will be a strong solution to the Peskin problem with tension $\TE$ and initial data $\bX_0$ in the sense of Definition \ref{def:StrongSolution}.    Thus Theorem \ref{thm:mainquant} follows, and the higher regularity in Theorem \ref{thm:mainquant} is a consequence of Proposition \ref{prop:higherReg}.  Theorem \ref{first:unique} is then a direct consequence of Corollary \ref{cor.L2.cont.m}.

Alternatively, for a tension $\TE$ satisfying also \eqref{tension.derivatives.continuity}, then by Proposition \ref{prop:continuity} using the equivalent weight $\nu$ in \eqref{nu.definition} we have
\begin{equation*}
     ||\bX'_n -  \bX'_m||_{\CN}
    +
  2\lambda^{\frac12} || \bX'_n -  \bX'_m||_{\DN}
    \leq   
    8 || \bX'_{0,n} -  \bX'_{0,m}||_{\BN}.
\end{equation*}
From \eqref{equivalent.nu.norm} we have $|| \bX'_{0,n} -  \bX'_{0,m}||_{\BN} \leq 2 || \bX'_{0,n} -  \bX'_{0,m}||_{\BA} \to 0$ as $m,n\to\infty$.  
Therefore $\{\bX'_n(t)\}$ is a Cauchy sequence in $\CN \cap \DN$ over $0<t<T=T^*$.  Since $\bX_{0,n}'\to \bX_0'$ in $\BA$ as $n \to \infty$ then $\bX'_n(t) \to \bX'(t)$ in $\CN \cap \DN$ over $0<t<T=T^*$.  Then the limit $\bX: [0,T^*] \to \R^2$ is a solution to the Peskin problem \eqref{peskin.general.tension} for tension $\TE$ with initial data $\bX_0$.  Now Theorem \ref{main:unique} follows from Proposition \ref{prop:continuity}.

Lastly, suppose that our scalar tension $\TE$ only satisfies the weaker qualitative assumptions \eqref{e:QualitativeScalarTension}.  We again assume that $\BX_0$ satisfies $||\bX_0'||_{\dot{B}^{\frac12, \subw}_{2,1}}\leq M$ and $|\bX_0|_*\geq 3\rho>0$.  Let $\tilde{\TE}:[0,\infty)\to [0,\infty)$ be such that  \begin{equation}\label{e:TildeTension}
    \tilde{\TE}(r) = \TE(r), \qquad \rho \leq r\leq ||\bX_0'||_{L^\infty}+\rho,
\end{equation}
and $\tilde{\TE}$ satisfies the stronger assumptions 
\eqref{e:QuantitativeScalarTension}, \eqref{e:DTdefn} and \eqref{e:QuantitativeTensionMap}.  Then by the above argument, there exists a strong solution $\BX:[0,T]\times \T\to \R^2$ to the Peskin problem with tension $\tilde{\TE}$ and initial data $\BX_0.$  

We claim that $\BX(t)$ is also a solution over [0,T] to the Peskin problem with our original tension $\TE$ as well.  To see this, notice that \eqref{arccord.bound.ep} implies that
\begin{equation}\label{e:Linfinityfinalbound}
||\bX'(t)-\bX'_0||_{L^\infty} \leq \rho, \qquad 0\leq t\leq T.
\end{equation}
We conclude that $\rho \leq |\bX(t)|_* \leq \inf_\theta |\bX'(t,\theta)|\leq  ||\bX'(t)||_{L^\infty_\theta} \leq ||\bX_0'||_{L^\infty_\theta}+\rho$ over $0\leq t\leq T$.  
Thus combining \eqref{e:TildeTension} and \eqref{e:Linfinityfinalbound} we obtain
\begin{multline}\notag
    \partial_t \bX(t,\theta) = \int_{\T} d\alpha  \frac{\bX'(\theta+\alpha) \DO(D_\alpha \bX)\bX'(\theta+\alpha)}{|\delta_\alpha \bX|^2} \frac{\tilde{\TE}(|\bX'|)(\theta+\alpha)}{|\bX'(\theta+\alpha)|} \delta_\alpha \bX(\theta)
    \\ = \int_{\T} d\alpha  \frac{\bX'(\theta+\alpha) \DO(D_\alpha \bX)\bX'(\theta+\alpha)}{|\delta_\alpha \bX|^2} \frac{\TE(|\BX'|)(\theta+\alpha)}{|\bX'(\theta+\alpha)|} \delta_\alpha \bX(\theta).
\end{multline}
We conclude that $\bX(t,\theta)$ is a solution to the Peskin problem \eqref{peskin.general.tension} for our original tension $\TE$ on the time interval $[0,T]$. The gain of higher regularity in Theorem \ref{thm:main} follows from Proposition \ref{prop:higherReg}.  We thus conclude that Theorem \ref{thm:main} holds.

\appendix
\section{Littlewood-Paley decomposition on the torus}\label{sec:LPtorus}

In this appendix we will state and prove some important Besov space embedding inequalities in the torus that are used in the main text.  Then we will work on  $\T^\DMD = [-\pi, \pi]^\DMD$ for $\DMD \ge 1$ since all the results are the same in any dimension. To this end we quickly build the Littlewood-Paley operators on $\T^\DMD$.    We refer to \cite[Section 2.3]{BCD} regarding the theory of Besov spaces using Littlewood-Paley operators in $\R^\DMD$.   The theory of Besov spaces on  $\T^\DMD$ is essentially the same, and it has been developed in \cite[Chapter 3.5]{SchTrie1987}.   We will explain the main embedding inequalities for Besov spaces in $\T^\DMD$ using the Littlewood-Paley operators in this appendix.  This approach allows us to develop the embeddings of the spaces $\dot{B}^{s,\subw}_{p,r}(\T^\DMD)$ in \eqref{Besov.mu.Space} and to develop the equivalences of $\dot{B}^{s,\subw}_{p,r}(\T^\DMD)$ in Proposition \ref{Besov.equivalence.prop}.  Although the proofs are  known, this appendix is included because we could not find any reference for these estimates of $\dot{B}^{s,\subw}_{p,r}(\T^\DMD)$.

To this end, we choose $\phi \in C^\infty_c(\mathbb{R}^\DMD)$ with $0 \le \phi \le 1$ such that $\phi(x) = \phi(|x|)$ and $\phi(x)=1$ for $|x|<\frac32$ and $\phi(x) =0$ for $|x|\geq \frac{8}{3}$, and $\phi(|x|)$ is non-increasing for $|x|\ge 0$.    Then define 
\begin{equation*}
    \varphi(x) \eqdef \phi(x)-\phi(2x),
\end{equation*}
so that $\varphi(x)\ge 0$ and $\varphi(x)=0$ for $|x|<\frac34$ and $\varphi(x)=0$ for $|x|\geq \frac83$. Further define
\begin{equation}\notag
\varphi_j(x) \eqdef \varphi(2^{-j}x), \quad j \in \Z. 
\end{equation}
Then we have for $x\ne 0$ that
\begin{equation}\notag
\sum_{j=m'}^m \varphi_j(x) = \phi(2^{-m}x)- \phi(2^{-m'+1}x) \to 1 \quad \text{as} \quad m\to\infty, m'\to-\infty. 
\end{equation}
In this sense it holds that
\begin{equation}\label{sum.one}
\sum_{j=-\infty}^\infty \varphi_j(x) =  1, \quad (x\ne 0). 
\end{equation}
These will be the building blocks of the Littlewood-Paley decomposition on $\T^\DMD$.  

 For a function $f:\T^\DMD \to \mathbb{C}$ we have the Fourier series representation
\begin{equation}\notag
    f(\alpha) = \sum_{k\in \mathbb{Z}^\DMD} \widehat{f}(k) e^{ik\cdot\alpha}, \quad 
    k\cdot\alpha = k_1 \alpha_1 + \cdots k_\DMD \alpha_\DMD.
\end{equation}
where the Fourier transform on $\T^\DMD$ is defined by 
\begin{equation}\notag
\mathcal{F}_{\T^\DMD}(f)(k) = \widehat{f}(k) \eqdef \frac{1}{(2\pi)^\DMD}\int_{\T^\DMD}   f(\alpha)e^{-ik\cdot\alpha} d\alpha, \quad k\in \mathbb{Z}^\DMD.
\end{equation}
Note that in the remainder of this section our function $f$ will always have mean zero, which means that 
\begin{equation}\notag
\mathcal{F}_{\T^\DMD}(f)(0)  = \frac{1}{(2\pi)^\DMD}\int_{\T^\DMD}   f(\alpha) d\alpha=0.
\end{equation}
Then we define the Littlewood-Paley projections on $\T^\DMD$ by 
\begin{equation}\notag
    \Delta_j f(\alpha) = \sum_{k\in \mathbb{Z}^\DMD} \varphi_j(k)\widehat{f}(k) e^{ik\cdot\alpha}, \quad j\in \Z,
\end{equation}
where the sum above clearly only contains a finite number of terms.  In particular we define the following sets 
\begin{equation}\notag
E_j \eqdef \{k\in \mathbb{Z}^\DMD: 3\cdot 2^{j-2} \leq |k| < 2^{j+3}/3 \}.
\end{equation}
Then 
\begin{equation}\notag
    \Delta_j f(\alpha) = \sum_{k\in E_j} \varphi_j(k)\widehat{f}(k) e^{ik\cdot \alpha}, \quad j\in \Z.
\end{equation}
We further have from \eqref{sum.one} that 
\begin{equation}\notag
    f(\alpha) = \sum_{j=-\infty}^\infty \Delta_j f(\alpha).
\end{equation}
We point out that the sum above terminates for sufficiently negative $j$.  In particular $\varphi_j(k) \ne 0$ if and only if $|k| < 2^{j+3}/3$.  Since $\widehat{f}(0)=0$ then $\Delta_j f =0$ whenever $2^{j+3}/3<1$.  Thus there exists a uniform fixed value $\jj\in \Z$ such that 
\begin{equation}\label{sum.terminates}
    f(\alpha) = \sum_{j=\jj}^\infty \Delta_j f(\alpha).
\end{equation}
Further notice that if $\varphi_j(k) \varphi_{j'}(k)\ne 0$ for $j' \ge j$ from the support condition this implies that $2^{j'} \frac34 \le \frac43 2^{j+1}$, which further implies $   0 \le  j' - j \le 1$. This combined with \eqref{sum.one} further implies that 
\begin{equation}\label{LP.extend}
    \Delta_j f(\alpha) 
    =
    \sum_{|j-j'|\le 1} \Delta_j \Delta_{j'}f(\alpha).
\end{equation}
This will be useful in several places below.

Next define 
\begin{equation}\notag
    h_j(\alpha) \eqdef \sum_{k\in \mathbb{Z}^\DMD} \varphi_j(k) e^{ik\cdot \alpha}. 
\end{equation}
Then we have that $\Delta_j f(\alpha) = (h_j * f)(\alpha).$
For a Schwartz function $\psi: \mathbb{R}^\DMD\to \mathbb{C}$ the Poisson summation formula grants
\begin{equation}\notag
\sum_{k \in \mathbb{Z}^\DMD} e^{ik\cdot \beta} \psi(k)
=
(2\pi)^\DMD \sum_{m \in \mathbb{Z}^\DMD}  \mathcal{F}_{\mathbb{R}^\DMD}^{-1}(\psi)(\beta+2\pi m), \quad \forall \beta \in \T^\DMD,
\end{equation}
where the inverse Fourier transform on $\mathbb{R}^\DMD$ is given by 
\begin{equation}\notag
\mathcal{F}_{\mathbb{R}^\DMD}^{-1}(f)(x) = \frac{1}{(2\pi)^{\frac{\DMD}{2}}} \int_{\mathbb{R}^\DMD}  f(\alpha)e^{ix\cdot \alpha} d\alpha.
\end{equation}
Therefore we are able to write 
\begin{equation}\notag
    h_j(\alpha) = \sum_{k\in \mathbb{Z}^\DMD} \varphi_j(k) e^{ik\cdot\alpha} 
    = (2\pi)^\DMD \sum_{m \in \mathbb{Z}^\DMD}  \mathcal{F}_{\mathbb{R}^\DMD}^{-1}(\varphi_j)(\alpha+2\pi m).
\end{equation}
In particular we will denote $\phi \eqdef \mathcal{F}_{\mathbb{R}^\DMD}^{-1}(\varphi)$ to conclude that
\begin{equation}\label{form.hj}
    h_j(\alpha) 
    = (2\pi)^\DMD ~2^{\DMD j} \sum_{m \in \mathbb{Z}^\DMD}  \phi(2^j\alpha+2\pi 2^j m).
\end{equation}
Since $\phi=\mathcal{F}_{\mathbb{R}^\DMD}^{-1}(\varphi)$ is a Schwartz function on $\mathbb{R}^\DMD$ then the sum converges aboslutely and $h_j$ is a periodic function on $\T^\DMD$.   If we use the form \eqref{form.hj} for the function $h_j$ then the proofs of the Besov space inequalities on $\mathbb{R}^\DMD$ translate to  $\mathbb{T}^\DMD$.

\begin{lemma}\label{lemma.bernstein}
With \eqref{form.hj}, for any $M\ge 0$ we have the following uniform estimate
\begin{equation}\label{l1.bound.h}
        \int_{\T^\DMD}  d\alpha~
  |2^j \alpha|^M  \left| h_j(\alpha) \right| 
\leq
 (2\pi)^\DMD \int_{\R^\DMD} d\alpha~ |\alpha|^M \left|\phi(\alpha)\right|<\infty.
\end{equation}
We also have for any $M\ge 0$ the  following uniform estimate for all $\alpha \in \T^\DMD$:
\begin{equation}\label{pointwise.hj}
|\alpha|^M \left| h_j(\alpha) \right| \lesssim 2^{\DMD j} 2^{-jM}.
\end{equation}
Further these bounds imply that for any $p\in (1,\infty)$ we have
\begin{equation}\label{lp.bound.h}
        \left(\int_{\T^\DMD}  d\alpha~
 ( |2^j \alpha|^M  \left| h_j(\alpha) \right| )^p \right)^{\frac1p}
    \lesssim 2^{\DMD j\left(1-\frac{1}{p} \right)}. 
\end{equation}
\end{lemma}

\begin{proof}We split \eqref{form.hj} as $h_j(\alpha) 
    = g_1(\alpha)+g_2(\alpha)$, where
\begin{equation}\notag
g_1(\alpha)\eqdef
 (2\pi)^\DMD 2^{\DMD j} \sum_{m: 2^j|\alpha+2\pi m|\le 1}  \phi(2^j\alpha+2\pi 2^j m),
\end{equation}
and
\begin{equation}\notag
g_2(\alpha)\eqdef
 (2\pi)^\DMD 2^{\DMD j} \sum_{m: 2^j|\alpha+2\pi m|> 1}  \phi(2^j\alpha+2\pi 2^j m).
\end{equation}
For $g_1(\alpha)$, since $2^j|\alpha+2\pi m|\le 1$, notice that we have 
\begin{equation}\notag
|m|\leq \frac{1}{2\pi} |\alpha+2\pi  m| + \frac{|\alpha|}{2\pi}
\leq 
\frac{1}{2\pi} 2^{-j} + \frac{1}{2} \le \frac{1}{2\pi} 2^{-\jj} + \frac{1}{2} \lesssim 1.
\end{equation}
Since $\left| \phi(\alpha) \right| \lesssim 1$ we have that 
$\left| g_1(\alpha)\right| \lesssim 
2^{\DMD j}$.

For $g_2(\alpha)$, since $2^j|\alpha+2\pi  m|> 1$, we use that $\phi$ is Schwartz on $\mathbb{R}^\DMD$ so that 
\begin{equation}\label{schwartz.estimate}
\left| \phi(\alpha) \right| \leq C_N |\alpha|^{-N}, \quad \forall |\alpha| \ge 1, \quad \forall N \ge \DMD+1.
\end{equation}
In particular 
\begin{equation}\notag
\left| g_2(\alpha)\right| \leq 
2^{\DMD j} C_N   \sum_{m: 2^j|\alpha+2\pi  m|> 1}  |2^j(\alpha+2\pi   m)|^{-N}.
\end{equation}
Here we have the following uniform in $j$ bound  
\begin{equation}\label{uniform.sum.bound}
\sum_{m: 2^j|\alpha+2\pi  m|> 1}  |2^j(\alpha+2\pi   m)|^{-N} \lesssim 1.
\end{equation}
We conclude that $\left| g_2(\alpha)\right|  \lesssim 2^{\DMD j}$.  This establishes \eqref{pointwise.hj} for $M=0$.  

We now prove \eqref{pointwise.hj} for $M>0$.  To this end, for $\alpha \in \T^d$ using \eqref{distance.alpha} we define 
\begin{equation*}
|\SL(\alpha)|^2 
\eqdef 
\sum_{i=1}^d \SL(\alpha_i)^2.
\end{equation*}
Then we have $|\SL(\alpha)| \approx |\alpha|$ uniformly for all $\alpha \in \T^\DMD$.  Further, from \eqref{distance.alpha} we have that  $|\SL(\alpha)| = |\SL(\alpha+2\pi  m)|$ for any $m\in \Z^\DMD$ and any $\alpha \in \T^\DMD$.  Thus from \eqref{form.hj} for any $\alpha \in \T^\DMD$ we have
\begin{equation}\notag
|\alpha|^M \left| h_j(\alpha) \right| 
\lesssim
|\SL(\alpha)|^M \left| h_j(\alpha) \right| 
\lesssim
 2^{\DMD j} \sum_{m \in \mathbb{Z}^\DMD}  |\SL(\alpha+2\pi  m)|^M \left| \phi(2^j\alpha+2\pi  2^j m) \right|.
\end{equation}
Notice from \eqref{distance.alpha} and \eqref{sine.bound} that we have the global uniform bound
\begin{equation}\notag
\frac{ |\SL(\alpha+2\pi m)|}{|\alpha+2\pi  m|} \lesssim 1, \quad \forall \alpha \in \T^\DMD \quad \forall m \in \Z^\DMD.
\end{equation}
We thus conclude that 
\begin{equation}\label{bound.M.hja}
|\alpha|^M \left| h_j(\alpha) \right| 
\lesssim
2^{\DMD j} \sum_{m \in \mathbb{Z}^\DMD}  |\alpha+2\pi  m|^M \left| \phi(2^j\alpha+2\pi  2^j m) \right|.
\end{equation}
We will split this sum into $2^j|\alpha+2\pi   m|\le 1$ and $2^j|\alpha+2\pi  m|> 1$ as previously.  On the region $2^j|\alpha+2\pi   m|\le 1$, as before independent of $j$ we have 
\begin{equation}\notag
 \sum_{m:2^j|\alpha+2\pi  m|\le 1}  |\alpha+2\pi  m|^M \left| \phi(2^j\alpha+2\pi  2^j m) \right|
\lesssim 1.
\end{equation}
This follows exactly as in the proof of \eqref{pointwise.hj} for $M=0$.  Next on the region $2^j|\alpha+2\pi   m|> 1$, we use the estimate \eqref{schwartz.estimate} with $N$ replaced by $N+M$ to obtain
\begin{multline}\notag
 \sum_{m:2^j|\alpha+2\pi   m|> 1}  |\alpha+2\pi   m|^M \left| \phi(2^j\alpha+2\pi  2^j m) \right|
\\
\lesssim 
C_{N+M}  \sum_{m:2^j|\alpha+2\pi   m|> 1}  |\alpha+2\pi   m|^M \left|2^j(\alpha+2\pi m) \right|^{-N-M}
\\
\lesssim 
2^{-jM} \sum_{m:2^j|\alpha+2\pi   m|> 1}  \left|2^j(\alpha+2\pi m) \right|^{-N}
\lesssim 
 2^{-jM}.
\end{multline}
The last uniform inequality follows as in \eqref{uniform.sum.bound}.  Collecting these estimates we obtain \eqref{pointwise.hj} for $M>0$.  We will now prove \eqref{l1.bound.h}. Since $h_j$ is $2\pi$ periodic, from \eqref{form.hj}
\begin{multline}\notag
        \int_{\T^\DMD}  d\alpha~
    \left| h_j(\alpha) \right| 
= (2\pi)^\DMD 2^{\DMD j} \sum_{m \in \mathbb{Z}^\DMD} 
            \int_{\T^\DMD}  d\alpha~
    \left| \phi(2^j\alpha+2\pi  2^j m) \right| 
    \\
\leq (2\pi)^\DMD  \sum_{m \in \mathbb{Z}^\DMD} 
            \int_{2^j\T^\DMD} d\alpha~
    \left| \phi(\alpha+2\pi 2^j m) \right|
        =
 (2\pi)^\DMD \int_{\R^\DMD} d\alpha~ \left|\phi(\alpha)\right|. 
\end{multline}
This yields \eqref{l1.bound.h} for $M=0$.  For $M>0$ we use \eqref{bound.M.hja} and then the proof is exactly the same.  Lastly, to prove \eqref{lp.bound.h} for $M=0$ for any $1<p<\infty$ we interpolate as 
\begin{equation}\notag
    || h_j ||_{L^p} 
    \lesssim     || h_j ||_{L^1}^{\frac{1}{p}}     || h_j ||_{L^\infty}^{1-\frac{1}{p}}. 
\end{equation}
Then \eqref{lp.bound.h} follows from \eqref{l1.bound.h} and \eqref{pointwise.hj}.  The proof of \eqref{lp.bound.h} for $M>0$ is exactly the same.
\end{proof}

\begin{lemma}\label{lemma.bernstein.second}  We have the following Bernstein inequalities
\begin{equation}\label{bernstein.1}
    || \Delta_j f ||_{L^q}
    \lesssim 
     2^{ j\left(\frac{1}{p} -  \frac{1}{q}\right) \DMD}|| \Delta_j f ||_{L^p},
     \quad q \ge p \ge 1.
\end{equation}
For any $0< m \le 1$ we have that 
\begin{equation}\label{bernstein.2}
2^{jm  }|| \Delta_j f ||_{L^p}
\lesssim
    || \Lambda^m \Delta_j f ||_{L^p}
    \lesssim 
     2^{jm }|| \Delta_j f ||_{L^p},
     \quad p \ge 1.
\end{equation}
\end{lemma}

The proof of Lemma \ref{lemma.bernstein.second} is in \cite[Lemma 2.1 on page 52]{BCD}, if we use \eqref{form.hj} and Lemma \ref{lemma.bernstein} in $\T^\DMD$.  Next, we recall the Besov spaces given in  \eqref{Besov.Space}, \eqref{Besov.mu.Space}, \eqref{Besov.CL.Space} and \eqref{Besov.mu.CL.Space}.  Then for $0<s<1$ and $p,q,r\in [1,\infty]$ and $\subw$ satisfying Definition \ref{subw.definition} we more generally define the semi-norm representation of the Besov spaces over $\T^\DMD$ by 
\begin{equation}\label{Besov.Space.DMD}
    ||f||_{\dot{B}^s_{p,r}(\T^\DMD)} \eqdef \left(\int_{\T^\DMD}  \frac{d\beta}{|\beta|^\DMD}  \left(\frac{||\delta_\beta f||_{L^p(\T^\DMD)}}{|\beta|^{s}} \right)^{r}  \right)^{1/r},
\end{equation}
\begin{equation}\label{Besov.CL.Space.DMD}
    || f||_{\widetilde{L}^{q}_T(\dot{B}_{p, r}^{s}(\T^\DMD))}
    \eqdef 
    \left(\int_{\T^\DMD}  \frac{d\beta}{|\beta|^\DMD}  \left(\frac{||\delta_\beta f||_{L^{q}_T(L^p(\T^\DMD))}}{|\beta|^{s}} \right)^{r}  \right)^{1/r},
\end{equation}
\begin{equation}\label{Besov.mu.Space.DMD}
    ||f||_{\dot{B}^{s,\mu}_{p,r}(\T^\DMD)} \eqdef 
    \left(\int_{\T^\DMD}  \frac{d\beta}{|\beta|^\DMD}  \left(\subw(|\beta|^{-1})\frac{||\delta_\beta f||_{L^p(\T^\DMD)}}{|\beta|^{s}} \right)^{r}  \right)^{1/r},
\end{equation}
\begin{equation}\label{Besov.mu.CL.Space.DMD}
    || f||_{\widetilde{L}^{q}_T(\dot{B}_{p, r}^{s,\subw}(\T^\DMD))}
    \eqdef
    \left(\int_{\T^\DMD}  \frac{d\beta}{|\beta|^\DMD}  \left(\subw(|\beta|^{-1})\frac{||\delta_\beta f||_{L^{q}_T(L^p(\T^\DMD))}}{|\beta|^{s}} \right)^{r}  \right)^{1/r}.
\end{equation}
For all the spaces above we use the standard modification when $r=\infty$.  We can equivalently write these semi-norms using the Littlewood-Paley operators as follows.  We define the $\ell^r=\ell^r(\Z)$ spaces with the norm
\begin{equation}\notag
    || a_j ||_{\ell^r}
    \eqdef 
    \left(\sum_{j \in \Z} |a_j|^r \right)^{1/r},
    \quad 1 \le r < \infty, 
    \quad 
        || a_j ||_{\ell^\infty}
        \eqdef \sup_{j \in \Z} |a_j|.
\end{equation}
Then we have the following equivalent representations of these Besov spaces.

\begin{proposition}\label{Besov.equivalence.prop}
We consider any $0<s<1$ and $p,r\in [1,\infty]$ and $\subw$ satisfying Definition \ref{subw.definition}.  Then we have for \eqref{Besov.Space.DMD} and \eqref{Besov.mu.Space.DMD} that 
\begin{equation}\notag
||f||_{\dot{B}^s_{p,r}(\T^\DMD)} \approx
    || 2^{js} || \Delta_j f||_{L^p_\theta} ||_{\ell^r},
\quad
||f||_{\dot{B}^{s,\subw}_{p,r}(\T^\DMD)} \approx
    || 2^{js} \subw(2^j) || \Delta_j f||_{L^p_\theta} ||_{\ell^r}.
\end{equation}
If also $q \in [1,\infty]$ then for \eqref{Besov.mu.CL.Space.DMD} we have
\begin{equation}\label{uniform.LP.next}
    ||f||_{\widetilde{L}^{q}_T(\dot{B}^{s,\subw}_{p,r}(\T^\DMD))} 
\approx
    || 2^{js} \subw(2^j) || \Delta_j f||_{L^{q}_T(L^p_\theta)} ||_{\ell^r},
\end{equation}
and for \eqref{Besov.CL.Space.DMD} we have 
$
    || f||_{\widetilde{L}^{q}_T(\dot{B}_{p, r}^{s}(\T^\DMD))}
\approx
    || 2^{js}  || \Delta_j f||_{L^{q}_T(L^p_\theta)} ||_{\ell^r}.
$
\end{proposition}

\begin{remark}\label{rem:besov.define}
These equivalences motivate the standard definition of these Besov spaces 
 for all  $s\in \R$, for all $p,r \in [1,\infty]$ and for any $\subw$ satisfying Definition \ref{subw.definition} as
\begin{equation}\notag
||f||_{\dot{B}^s_{p,r}(\T^\DMD)} \eqdef
    || 2^{js} || \Delta_j f||_{L^p_\theta} ||_{\ell^r},
\quad
||f||_{\dot{B}^{s,\subw}_{p,r}(\T^\DMD)} \eqdef
    || 2^{js} \subw(2^j) || \Delta_j f||_{L^p_\theta} ||_{\ell^r}.
\end{equation}
And if also $q \in [1,\infty]$ then we similarly can define
\begin{equation}\notag
    || f||_{\widetilde{L}^{q}_T(\dot{B}_{p, r}^{s}(\T^\DMD))}
\eqdef
    || 2^{js}  || \Delta_j f||_{L^{q}_T(L^p_\theta)} ||_{\ell^r},
\end{equation}
and
$
||f||_{\widetilde{L}^{q}_T(\dot{B}^{s,\subw}_{p,r}(\T^\DMD))} \eqdef
    || 2^{js} \subw(2^j) || \Delta_j f||_{L^{q}_T(L^p_\theta)} ||_{\ell^r}.
$
\end{remark}

\begin{proof}[Proof of Proposition \ref{Besov.equivalence.prop}]
We only show the proof of the equivalence of \eqref{Besov.mu.CL.Space.DMD} as in \eqref{uniform.LP.next}.  The proofs of the other equivalences are exactly the same, or easier.  In this proof we will write the semi-norm on the RHS in \eqref{uniform.LP.next} as 
    \begin{equation}\notag
    ||f||_{\mathcal{P}}^r
\eqdef
|| 2^{js} \subw(2^j) || \Delta_j f||_{L^{q}_T(L^p_\theta)} ||_{\ell^r}^r.
    \end{equation}
For brevity we write the LHS of \eqref{uniform.LP.next} as 
    $    ||f||_{\mathcal{Q}}
\eqdef ||f||_{\widetilde{L}^{q}_T(\dot{B}^{s,\subw}_{p,r}(\T^\DMD))} $ from \eqref{Besov.mu.CL.Space.DMD}.    
    
Then, from \eqref{LP.extend}, we that 
$
\Delta_j = \sum_{|j-j'|\le 1} \Delta_j \Delta_{j'}.
$
Next we use \eqref{form.hj} to obtain
\begin{equation*}
     \delta_\beta \Delta_j f(\theta) 
      = \sum_{|j-j'|\le 1} \left(\delta_\beta h_j * \Delta_{j'}f \right)(\theta),
\end{equation*}
where we expand
$\delta_\beta h_j(\alpha) = \int_0^1 ds  ~ 2^j \beta \cdot (\nabla h)_j( \alpha+s\beta)$.  Then as in \eqref{form.hj}  we have 
\begin{equation}\notag
    (\nabla h)_j(\alpha) 
    = (2\pi)^\DMD 2^{\DMD j} \sum_{m \in \mathbb{Z}^\DMD}  (\nabla \phi)(2^j\alpha+2\pi  2^j m).
\end{equation}
Notice that for any $y\in \mathbb{R}^\DMD$, exactly the same as \eqref{l1.bound.h}, we have
\begin{equation}\notag
    \int_{\T^\DMD}  d\alpha~
    \left| (\nabla h)_j(\alpha+y) \right| 
    \lesssim
    \int_{\mathbb{R}^\DMD} d\alpha~ \left|\nabla \phi(\alpha)\right| 
    \lesssim 1.
\end{equation}
 We conclude from \eqref{l1.bound.h} and the above that 
\begin{equation}\label{hj.beta.bound}
||\delta_\beta h_j||_{L^1_\theta}
\lesssim \min\{1, 2^j |\beta|\}.
\end{equation}
Thus using Young's inequality we have 
\begin{equation}\notag
    ||\delta_\beta \Delta_j f||_{L^q_T(L^p_\theta)} 
    \lesssim \min\{1, 2^j |\beta|\} \sum_{|j-j'|\le 1} 
    ||\Delta_{j'} f||_{L^q_T(L^p_\theta)}. 
\end{equation}
Thus we have
\begin{equation}\notag
    ||\delta_\beta \Delta_j f||_{L^q_T(L^p_\theta)} 
    \lesssim c_{r,j} 2^{-s j} \subw(2^{j})^{-1} \min\{1, 2^j |\beta|\}     ||f||_{\mathcal{P}}, 
\end{equation}
where above and in the rest of the proof $c_{r,j}\ge 0$ is an element of the unit sphere of $\ell^r(\mathbb{Z})$ (which could be a different element on different lines).  In this case 
\begin{equation}\notag
c_{r,j}=\frac{ \sum_{|j-j'|\le 1} 
2^{s j'}\subw(2^{j'})||\Delta_{j'} f||_{L^q_T(L^p_\theta)}  }{3 ||f||_{\mathcal{P}}}.
\end{equation}
Thus we have that 
\begin{multline}\label{triangle.upper.pq}
    ||\delta_\beta f||_{L^q_T(L^p_\theta)}
\leq
   \sum_{j=-\infty}^{\infty} ||\delta_\beta \Delta_j f||_{L^q_T(L^p_\theta)}
   \\
    \lesssim     ||f||_{\mathcal{P}}
    \left( |\beta| \sum_{j \le j_0}c_{r,j} 2^{j(1-s)}\subw(2^{j})^{-1}  +
    \sum_{j > j_0}c_{r,j} 2^{-sj} \subw(2^{j})^{-1}  \right). 
\end{multline}
Here $j_0=j_0(|\beta|)$ satisfies that 
$\frac{1}{2}\frac{1}{|\beta|} < 2^{j_0} \leq \frac{1}{|\beta|}$.

We first suppose that $1 \le r < \infty$.  Then 
we conclude that 
\begin{equation}\label{Q.P.upper.bd}
    ||f||_{\mathcal{Q}}^r
    \leq C 2^r    ||f||_{\mathcal{P}}^r
    \left( I_1+I_2\right),
\end{equation}
where as in \eqref{Besov.mu.CL.Space.DMD} we have
\begin{equation}\notag
I_1 \eqdef         \int_{\T^\DMD}   \frac{d\beta}{|\beta|^\DMD} \subw(|\beta|^{-1})^r  |\beta|^{r(1-s)} 
    \left(  \sum_{j \le j_0}c_{r,j} 2^{j(1-s)}\subw(2^{j})^{-1}   \right)^r,
\end{equation}
and
\begin{equation}\notag
    I_2 \eqdef         \int_{\T^\DMD}   \frac{d\beta}{|\beta|^\DMD} \subw(|\beta|^{-1})^r  |\beta|^{-rs} 
    \left( 
    \sum_{j > j_0}c_{r,j} 2^{-sj} \subw(2^{j})^{-1}  \right)^r.
\end{equation}
We use H{\"o}lder's inequality, and $2^{-sj_0} \lesssim |\beta|^s$, as 
\begin{multline}\label{use.holder}
     \left( 
    \sum_{j > j_0}c_{r,j} 2^{-sj} \subw(2^{j})^{-1}  \right)^r
    \lesssim
        \left(   \sum_{j > j_0} 2^{-sj}  \right)^{r-1}
    \sum_{j > j_0}c_{r,j}^r 2^{-sj} \subw(2^{j})^{-r}  
    \\
      \lesssim
   |\beta|^{s(r-1)}    \sum_{j > j_0}c_{r,j}^r 2^{-sj} \subw(2^{j})^{-r}.
\end{multline}
Then for $I_2$ by Fubini's theorem we have the estimate
\begin{equation}\notag
    I_2 \lesssim      \sum_{j=-\infty }^{\infty} c_{r,j}^r 2^{-sj} \subw(2^{j})^{-r}
    \int_{\T^\DMD}   \frac{d\beta}{|\beta|^\DMD} \subw(|\beta|^{-1})^r  |\beta|^{-s} 
    \1_{2^j|\beta|  > 1}
        \lesssim        \sum_{j=-\infty }^{\infty} c_{r,j}^r
    \lesssim        1.
\end{equation}
Above we used that $\subw$ is increasing from Definition \ref{subw.definition}.

Next for $I_1$, we use H{\"o}lder's inequality similar to \eqref{use.holder}  to get
\begin{equation*}
  \left(  \sum_{j \le j_0}c_{r,j} 2^{j(1-s)}\subw(2^{j})^{-1}   \right)^r
      \lesssim
   |\beta|^{-(1-s)(r-1)}    \sum_{j \le j_0} c_{r,j}^r 2^{(1-s)j} \subw(2^{j})^{-r}.  
\end{equation*}
Then also by Fubini's theorem, we have 
\begin{equation}\notag
    I_1 \lesssim      \sum_{j=-\infty }^{\infty} c_{r,j}^r   2^{(1-s)j} \subw(2^{j})^{-r}  \int_{\T^\DMD}   \frac{d\beta}{|\beta|^\DMD} \subw(|\beta|^{-1})^r  |\beta|^{1-s} \1_{2^j|\beta|  \le 1}.
\end{equation}
To estimate this term we will use a decomposition that is similar to the one from \cite[Equation (18) on Page 10]{AlaNgu2021lipsh}.  The intuition of the decomposition is that under our assumptions a term like $2^{(1-s)j/2} \subw(2^{j})^{-r} $ will be effectively eventually increasing.  In particular we split
\begin{equation}\label{splitting.upper}
2^{(1-s)j/2} \subw(2^{j})^{-r} 
 =
(\pi_1(2^{j}) \pi_2(2^{j})\pi_3(2^{j}))^{-1},
\end{equation}
where for $c_s \eqdef \exp\left(\frac{3r}{1-s}\right)$ we have 
\begin{equation}\notag
 \pi_1(\tau) 
 \eqdef
 \frac{\subw(\tau)^r }{(\log(4+\tau))^r },
 ~~
  \pi_2(\tau) 
 \eqdef
 \frac{(\log(4+\tau))^r }{(\log(c_s+\tau))^r },
~~
  \pi_3(\tau) 
 \eqdef
   \tau^{-(1-s)/2}(\log(c_s+\tau))^r .
\end{equation}
Then, by Definition \ref{subw.definition}, $ \pi_1(\tau)$ is decreasing for $\tau\in [0,\infty)$.  Further $\pi_2(\tau)$ is clearly uniformly bounded from above and below.  And we will see that $\pi_3(\tau)$ is decreasing for $\tau\in [0,\infty)$.  In particular
\begin{multline}\notag
        \frac{d}{d\tau}  \pi_3(\tau) 
    =\tau^{-(1-s)/2}  \left(\frac{r}{c_s+\tau} - \frac{(1-s)}{2\tau} \log(c_s+\tau )\right)(\log(c_s+\tau))^{r-1}
    \\
    =\tau^{-(1-s)/2} \frac{(1-s)\log(c_s+\tau )}{2\tau}\left(\frac{2\tau r}{(1-s)(c_s+\tau )\log(c_s+\tau )} - 1 \right)<0.
\end{multline}
The above holds for all $\tau\in [0,\infty)$, so that  $\pi_3(\tau)$ is decreasing, because 
\begin{equation}\notag
\frac{2\tau r}{(1-s)(c_s+\tau )\log(c_s+\tau )}
<
\frac{2 r}{(1-s)\log(e^{3r/(1-s)}+\tau )}
\leq
\frac{2}{3}<1.
\end{equation}
Note that this is also true with a better constant than $c_s$, which was chosen for clarity of the exposition.  We conclude that 
\begin{multline}\notag
2^{(1-s)j} \subw(2^{j})^{-r}  \int_{\T^\DMD}   \frac{d\beta}{|\beta|^\DMD} \subw(|\beta|^{-1})^r  |\beta|^{1-s} \1_{2^j|\beta|  \le 1}
 \\
    \lesssim
 2^{(1-s)j/2}         \pi_2(2^{j})^{-1}  \int_{\T^\DMD}   \frac{d\beta}{|\beta|^\DMD} |\beta|^{(1-s)/2} \pi_2(|\beta|^{-1}) \1_{2^j|\beta|  \le 1}
      \lesssim 1.
\end{multline}
Therefore we have that $I_1 \lesssim        \sum_{j=-\infty }^{\infty} c_{r,j}^r \lesssim 1$.  Thus we conclude that
$
||f||_{\mathcal{Q}}
\lesssim     ||f||_{\mathcal{P}}. 
$
This proves the upper bound in \eqref{uniform.LP.next} for $1 \le r <\infty$.  If $r=\infty$ we use $c_{r, j} \lesssim 1$ in \eqref{triangle.upper.pq} then following the same argument we obtain $
||f||_{\mathcal{Q}}
\lesssim     ||f||_{\mathcal{P}}. 
$

We will now prove the opposite inequality.  Using that the mean value of $h_j(\beta)$ from \eqref{form.hj} is zero we have
\begin{equation}\notag
 \Delta_j f(\theta) =
    \int_{\T^\DMD}  d\beta ~ h_j(\beta)  ~ \tau_{\beta}f(\theta)
    =\int_{\T^\DMD}  d\beta ~ h_j(\beta)  ~ \delta_{\beta}f(\theta).
\end{equation}
Then if $r=\infty$ we have 
\begin{equation*}
    2^{js} \subw(2^j) || \Delta_j f||_{L^{q}_T(L^p_\theta)}
    \lesssim
    ||f||_{\mathcal{Q}}
    \int_{\T^\DMD} d\beta ~ \frac{\subw(2^{j})}{\subw(|\beta|^{-1})} ~ |2^{j}\beta|^s  | h_j(\beta)| .
\end{equation*}
Then splitting into $2^j |\beta| \leq 1$ and $2^j |\beta| > 1$ we can prove that $\frac{\subw(2^{j})}{\subw(|\beta|^{-1})} \lesssim 1$ as in Definition \ref{subw.definition} and \eqref{splitting.upper}.  Then when $r=\infty$, $||f||_{\mathcal{P}} \lesssim ||f||_{\mathcal{Q}}$ follows from \eqref{l1.bound.h}.

If $1 \le r < \infty$, then we have $||f||_{\mathcal{P}}^r \leq 2^r(\Sigma_1^r + \Sigma_2^r)$ where
\begin{equation}\notag
    \Sigma_1^r 
    \eqdef
\sum_{j\in \Z} 2^{j s r} \subw(2^{j})^r
    \left(\int_{2^j|\beta| \le 1} d\beta ~ \left| h_j(\beta) \right| ~ 
    ||\delta_{\beta}f||_{L^q_T(L^p_\theta)}\right)^r,
\end{equation}
and
\begin{equation}\notag
    \Sigma_2^r
    \eqdef
\sum_{j\in \Z} 2^{j s r} \subw(2^{j})^r
   \left( \int_{2^j|\beta| > 1} d\beta ~ \left| h_j(\beta) \right| ~ 
    ||\delta_{\beta}f||_{L^q_T(L^p_\theta)}\right)^r.
\end{equation}
For $\Sigma_1^r$
we use H{\"o}lder's inequality and \eqref{lp.bound.h} as 
\begin{multline}\notag
    \left(\int_{2^j|\beta| \le 1} d\beta ~ \left| h_j(\beta) \right| ~ 
    ||\delta_{\beta}f||_{L^q_T(L^p_\theta)}\right)^r
    \\
    \lesssim
    \left(\int_{2^j|\beta| \le 1} d\beta ~ \left| h_j(\beta) \right|^{r'} ~ 
 \right)^{r-1}
\int_{2^j|\beta| \le 1} d\beta ~ 
    ||\delta_{\beta}f||_{L^q_T(L^p_\theta)}^r
    \\
      \lesssim
2^{\DMD j}\int_{2^j|\beta| \le 1} d\beta ~ 
    ||\delta_{\beta}f||_{L^q_T(L^p_\theta)}^r.
\end{multline}
We plug this in and use Fubini's theorem to obtain
\begin{equation}\notag
    \Sigma_1^r 
        \lesssim
\int_{\T^\DMD}  d\beta ~ \left(\sum_{j\in \Z} 2^{j (s r+\DMD)} \subw(2^{j})^r \1_{2^j|\beta| \le 1}\right) 
    ||\delta_{\beta}f||_{L^q_T(L^p_\theta)}^r
    \lesssim ||f||_{\mathcal{Q}}^r.
\end{equation}
The last inequality follows since $\subw(\tau)$ is increasing from Definition \ref{subw.definition}.

Lastly we consider the term $\Sigma_2^r$.  We again use H{\"o}lder's inequality  as 
\begin{multline}\notag
    \left(\int_{2^j|\beta| > 1} d\beta ~ \left| h_j(\beta) \right| ~ 
    ||\delta_{\beta}f||_{L^q_T(L^p_\theta)}\right)^r
    \\
    =
       2^{-j(\DMD+1)r} \left(\int_{2^j|\beta| > 1} \frac{d\beta}{|\beta|^\DMD} ~ |2^j\beta|^{\DMD+1} \left| h_j(\beta) \right| ~ 
    \frac{||\delta_{\beta}f||_{L^q_T(L^p_\theta)}}{|\beta|}\right)^r
    \\
    \lesssim
       2^{-j r} \int_{2^j|\beta| > 1} \frac{d\beta}{|\beta|^\DMD} ~ 
    \frac{||\delta_{\beta}f||_{L^q_T(L^p_\theta)}^r}{|\beta|^r}.
\end{multline}
Above from \eqref{lp.bound.h} we used that 
\begin{equation*}
    \left(\int_{2^j|\beta| > 1} \frac{d\beta}{|\beta|^\DMD} ~ |2^j\beta|^{(\DMD+1)r'} \left| h_j(\beta) \right|^{r'} ~ 
    \right)^{\frac{r}{r'}}
    \lesssim 2^{j \DMD  } 2^{j \DMD (r-1)}.
\end{equation*}
We thus conclude that 
\begin{equation}\notag
    \Sigma_2^r 
        \lesssim
\int_{\T^\DMD}  \frac{d\beta}{|\beta|^\DMD}  ~ \left(\sum_{j\in \Z} 2^{-j r (1-s )} \subw(2^{j})^r \1_{2^j|\beta| > 1}\right) 
    \frac{||\delta_{\beta}f||_{L^q_T(L^p_\theta)}^r}{|\beta|^r}
    \lesssim ||f||_{\mathcal{Q}}^r.
\end{equation}
Above we just used the following uniform inequality
\begin{equation}\notag
\sum_{j\in \Z} 2^{-j r (1-s )} \subw(2^{j})^r \1_{2^j|\beta| > 1}
\lesssim
|\beta|^{r(1-s)} \subw(|\beta|^{-1})^r,
\end{equation}
which is again a consequence of \eqref{splitting.upper}.  This completes the proof. \end{proof}

We also refer the reader to the analogous proofs of Proposition \ref{Besov.equivalence.prop} of these equivalences (without the factor $\subw$ and without the argument in \eqref{splitting.upper}) in the whole space case from \cite[Theorem 2.36 on page 74]{BCD}.  

\begin{proposition}\label{besov.ineq.prop}
For $s_1\in \R$, $1\le p_1 \le p_2 \le \infty$, $1\le r_1 \le r_2 \le \infty$, any $\subw$ satisfying Definition \ref{subw.definition} and $s_2 = s_1+\DMD(\frac{1}{p_1}-\frac{1}{p_2}),$ we have the uniform estimates 
\begin{equation}\notag
    || f||_{\dot{B}_{p_2, r_2}^{s_1}(\T^\DMD)}
    \lesssim
    || f||_{\dot{B}_{p_1, r_1}^{s_2}(\T^\DMD)}, \quad
    || f||_{\dot{B}_{p_2, r_2}^{s_1,\mu}(\T^\DMD)}
    \lesssim
    || f||_{\dot{B}_{p_1, r_1}^{s_2,\mu}(\T^\DMD)}. 
\end{equation}
Additonally for any $1\le q \le \infty$ we have
\begin{equation}\notag
    || f||_{\widetilde{L}^{q}_T(\dot{B}_{p_2, r_2}^{s_1}(\T^\DMD))}
    \lesssim
    || f||_{\widetilde{L}^{q}_T(\dot{B}_{p_1, r_1}^{s_2}(\T^\DMD))}, 
    \quad
    || f||_{\widetilde{L}^{q}_T(\dot{B}_{p_2, r_2}^{s_1,\mu}(\T^\DMD))}
    \lesssim
    || f||_{\widetilde{L}^{q}_T(\dot{B}_{p_1, r_1}^{s_2,\mu}(\T^\DMD))}.
\end{equation}
\end{proposition}

The proof is the standard, see \cite[Proposition 2.20 on page 64]{BCD}.  Next we state a lemma about interpolation in Besov spaces.

\begin{lemma}\label{Besov.interpolation}
If $s_1<s_2$ are real numbers and $\theta\in (0,1)$, then for any $1\leq p\leq \infty$ we have 
\begin{equation}\notag
||f||_{\dot{B}^{\theta s_1+(1-\theta)s_2}_{p,1}}\lesssim 
\frac{1}{s_2-s_1}\left(\frac{1}{\theta}+ \frac{1}{1-\theta}\right) ||f||_{\dot{B}^{s_1}_{p,\infty}}^\theta ||f||_{\dot{B}^{s_2}_{p,\infty}}^{1-\theta}
\end{equation}
Additionally for any $1\leq r\leq \infty$ we have 
$
||f||_{\dot{B}^{\theta s_1+(1-\theta)s_2}_{p,r}}\leq
||f||_{\dot{B}^{s_1}_{p,r}}^\theta ||f||_{\dot{B}^{s_2}_{p,r}}^{1-\theta}.
$
\end{lemma}

The above is proven in \cite[Proposition 2.22 on page 65]{BCD}.

\begin{lemma}\label{Besov.increase}
If $s_1<s_2$ are real numbers, then for any $1\leq p\leq \infty$ we have 
\begin{equation}\notag
||f||_{\dot{B}^{s_1}_{p,1}}\lesssim 
||f||_{\dot{B}^{s_2}_{p,\infty}}
\end{equation}
\end{lemma}

The proof of Lemma \ref{Besov.increase} follows directly from the property \eqref{sum.terminates}.

\begin{lemma}\label{Besov.embedding}
For any $(p,q)\in [1,\infty]^2$ such that $p \leq q$ the space $\dot{B}^{\frac{\DMD}{p}-\frac{\DMD}{q}}_{p,1}(\T^{\DMD})$ is continuously embedded in $L^q(\T^{\DMD})$.  In particular, 
\begin{equation}\notag
||f||_{L^p(\T^{\DMD})}\lesssim 
||f||_{\dot{B}^{\frac{\DMD}{p}-\frac{\DMD}{q}}_{p,1}(\T^{\DMD})}.
\end{equation}
In addition if $p<\infty$ then $\dot{B}^{\frac{\DMD}{p}}_{p,1}(\T^{\DMD})$ is continuously embedded in the space of continuous functions.
\end{lemma}

This is proven in \cite[Proposition 2.39 on page 79]{BCD} using \eqref{form.hj}.

\section{Estimates on the differences of the kernels}\label{sec:kernelDIFF}

The purpose of this appendix is to prove the pointwise bounds that are stated in Lemma's \ref{A.bound.lem}, \ref{lemm.A.diff} and \ref{lem:Abeta.upper}.  To ease the notation in this appendix we will drop the dependencies on $\theta$.  In particular we write  $\delta_\alpha^+ \BX'(\theta) = \DAP$, $\delta_\alpha^- \BX'(\theta) = \DAM$, $\delta_\alpha^+ \BY'(\theta) = \DAYP$, $\delta_\alpha^- \BY'(\theta) = \DAYM$, $\DATX = \DAX$, $\DATY = \DAY$,
$\delta_\alpha^+ (\bX'-\bY')(\theta) = \DAPXY$, and 
$\delta_\alpha^- (\bX'-\bY')(\theta) = \DAMXY$ etc.  

First we will give the proof of Lemma \ref{A.bound.lem}.

\begin{proof}[Proof of Lemma \ref{A.bound.lem}]
Considering $\mathcal{A}(\theta, \alpha)$ from \eqref{kerbel.A.eqn.deriv}, for \eqref{Abeta.split} we can write
\begin{equation}\label{A1X.def}
    4\pi \mathcal{A}_{1\beta}(\theta, \alpha) \eqdef    4\pi\mathcal{A}_{1\beta 1}+4\pi\mathcal{A}_{1\beta 2}+4\pi\mathcal{A}_{1\beta 3}+4\pi\mathcal{A}_{1\beta 4},
\end{equation}
where
\begin{equation}\notag
4\pi\mathcal{A}_{1\beta 1}
\eqdef
\delta_\beta \DAP\cdot \tau_\beta\frac{ \DO(\DAX) }{|\DAX|^2}\tau_\beta \DAM\IO 
+
\DAP\cdot \frac{ \DO(\DAX) }{|\DAX|^2}\delta_\beta\DAM \IO,  
\end{equation}
\begin{multline}\notag
4\pi\mathcal{A}_{1\beta 2}
\eqdef
(\delta_\beta\DAP + \delta_\beta\DAM)\cdot \tau_\beta \frac{\DO(\DAX) \DAX}{|\DAX|^2}\IO 
\\ 
- (\delta_\beta\DAP + \delta_\beta\DAM)\cdot  \tau_\beta \frac{ \RO(\DAX) \DAX}{|\DAX|^2}
\tau_\beta \RO(\DAX),
\end{multline}

\begin{multline}\notag
4\pi\mathcal{A}_{1\beta 3}
\eqdef
-\delta_\beta\DAP \cdot \tau_\beta \frac{\RO(\DAX) }{|\DAX|^2} \tau_\beta\DAM
\tau_\beta \RO(\DAX)
\\
-\DAP \cdot \frac{\RO(\DAX) }{|\DAX|^2} \delta_\beta\DAM
\tau_\beta \RO(\DAX),
\end{multline}
\begin{multline}\notag
4\pi\mathcal{A}_{1\beta 4}
\eqdef
 \delta_\beta\DAP \cdot \tau_\beta\frac{(\DO(\DAX) - \IO) }{|\DAX|^2}\tau_\beta\DAM \tau_\beta \DO(\DAX)
\\
+ \DAP\cdot \frac{(\DO(\DAX) - \IO) }{|\DAX|^2}\delta_\beta\DAM \tau_\beta \DO(\DAX).
\end{multline}
Then for $\mathcal{A}_{2\beta}$ we also further split  
\begin{equation}\label{A2X.def}
    4\pi \mathcal{A}_{2\beta}(\theta, \alpha) \eqdef    4\pi\mathcal{A}_{2\beta 1}+4\pi\mathcal{A}_{2\beta 2}+4\pi\mathcal{A}_{2\beta 3}+4\pi\mathcal{A}_{2\beta 4},
\end{equation}
where
\begin{equation*}
4\pi\mathcal{A}_{2\beta 1}
\eqdef
\DAP \cdot\delta_\beta\left(  \frac{ \DO(\DAX) }{|\DAX|^2}\right)\tau_\beta \DAM\IO,
\end{equation*}
\begin{multline}\notag
4\pi\mathcal{A}_{2\beta 2}
\eqdef
-\DAP \cdot\delta_\beta\left(  \frac{\RO(\DAX) }{|\DAX|^2} \right) \tau_\beta\DAM
\tau_\beta \RO(\DAX)
\\
-\DAP \cdot \frac{\RO(\DAX) }{|\DAX|^2} \DAM
 \delta_\beta \RO(\DAX),
\end{multline}
\begin{multline}\notag
    4\pi\mathcal{A}_{2\beta 3}
\eqdef
(\DAP + \DAM) \cdot \delta_\beta\left( \frac{\DO(\DAX) \DAX}{|\DAX|^2}\right)\IO
\\
- (\DAP + \DAM) \cdot \delta_\beta\left(  \frac{ \RO(\DAX) \DAX}{|\DAX|^2}
\RO(\DAX) \right),
\end{multline}
\begin{multline}\notag
4\pi\mathcal{A}_{2\beta 4}
\eqdef
 \DAP \cdot \delta_\beta\left(  \frac{(\DO(\DAX) - \IO) }{|\DAX|^2} \right) \tau_\beta\DAM \tau_\beta \DO(\DAX)
\\
+ \DAP\cdot \frac{(\DO(\DAX) - \IO) }{|\DAX|^2}\DAM 
\delta_\beta \DO(\DAX).
\end{multline}
Then the bound \eqref{A1X.bound} for $\mathcal{A}_{1\beta}(\theta, \alpha)$ follows directly from \eqref{A1X.def}.  And the bound for term $\mathcal{A}_{2\beta}(\theta, \alpha)$ in\eqref{A2X.bound} similarly follows from \eqref{A2X.def}.  This completes the proof.
\end{proof}

Next we give the proof of Lemma \ref{lemm.A.diff}.

\begin{proof}[Proof of Lemma \ref{lemm.A.diff}] We recall \eqref{kerbel.A.eqn.deriv} and then we split
\begin{equation}\notag
   4\pi \mathcal{A}[\bX]-4\pi\mathcal{A}[\bY]= 4\pi\mathcal{A}_1+4\pi\mathcal{A}_2.
\end{equation}
Here we are splitting so that $\mathcal{A}_1$ contains all the $\bX - \bY$ differences on the terms such as  $\delta_\alpha^\pm (\bX'-\bY')(\theta)$, and $\mathcal{A}_2$ contains the terms that have differences on $\DAL(\bX'-\bY')(\theta)$.  Thus for $\mathcal{A}_1$ we further split
\begin{equation*}
 \mathcal{A}_1 =  \mathcal{A}_{11}+\mathcal{A}_{12}+\mathcal{A}_{13}+\mathcal{A}_{14},
\end{equation*}
where
\begin{equation*}
4\pi \mathcal{A}_{11} = \DAPXY \cdot \frac{\DO(\DAX) }{|\DAX|^2}\DAM\IO
+
\DAYP \cdot \frac{\DO(\DAX) }{|\DAX|^2}\DAMXY\IO,
\end{equation*}
\begin{multline} \notag
4\pi \mathcal{A}_{12}  =   (\DAPXY+\DAMXY) \cdot\frac{ \DO(\DAX) \DAX}{|\DAX|^2} \IO
\\
- (\DAPXY+\DAMXY) \cdot\frac{ \RO(\DAX) \DAX}{|\DAX|^2} \RO(\DAX),
\end{multline}
\begin{multline} \notag
4\pi \mathcal{A}_{13}  =  
-\DAPXY \cdot \frac{\RO(\DAX) }{|\DAX|^2}\DAM
\RO(\DAX)
\\ 
-\DAYP \cdot \frac{\RO(\DAX) }{|\DAX|^2}\DAMXY
\RO(\DAX),
\end{multline}
\begin{multline} \notag
4\pi \mathcal{A}_{14}  =   
 \DAPXY \cdot\frac{ (\DO(\DAX) - \IO) }{|\DAX|^2}\DAM \DO(\DAX)
\\ 
+ \DAYP \cdot\frac{ (\DO(\DAX) - \IO) }{|\DAX|^2}\DAMXY \DO(\DAX).
\end{multline}
Therefore we observe that $\mathcal{A}_1$ satisfies the following uniform estimate
\begin{multline} \notag
\left| \mathcal{A}_1 \right|
\lesssim
|\BX|_*^{-2}\left(\left| \DAPXY\right| \left| \DAM\right|
+
\left| \DAMXY\right| \left| \DAYP\right|
\right)
\\
+  
|\BX|_*^{-1}\left(\left| \DAPXY\right| 
+
\left| \DAMXY\right| 
\right).
\end{multline}
And then for $\mathcal{A}_2$ we also split
\begin{equation*}
 \mathcal{A}_2 =  \mathcal{A}_{21}+\mathcal{A}_{22}+\mathcal{A}_{23}+\mathcal{A}_{24},
\end{equation*}
where
\begin{multline} \notag
4\pi \mathcal{A}_{21}  = \DAYP \cdot \left(\frac{\DO(\DAX) }{|\DAX|^2}-\frac{\DO(\DAY) }{|\DAY|^2} \right)\DAYM\IO
\\
+  (\DAYP+\DAYM) \cdot
\left(\frac{\DO(\DAX) \DAX}{|\DAX|^2} - \frac{\DO(\DAY) \DAY}{|\DAY|^2} \right)\IO,
\end{multline}
\begin{multline} \notag
4\pi \mathcal{A}_{22}  = 
-\DAYP \cdot \left(\frac{ \RO(\DAX) }{|\DAX|^2}-\frac{ \RO(\DAY) }{|\DAY|^2}\right)\DAYM
\RO(\DAX)
\\ 
-\DAYP \cdot    \frac{ \RO(\DAY) }{|\DAY|^2}\DAYM
\left(\RO(\DAX) - \RO(\DAY) \right),
\end{multline}
\begin{multline} \notag
4\pi \mathcal{A}_{23}  = 
- (\DAYP+\DAYM) \cdot
\left(\frac{ \RO(\DAX) \DAX}{|\DAX|^2} - \frac{ \RO(\DAY) \DAY}{|\DAY|^2} \right)\RO(\DAX)
\\
- (\DAYP+\DAYM) \cdot
\frac{ \RO(\DAY) \DAY}{|\DAY|^2} 
\left(\RO(\DAX) - \RO(\DAY) \right),
\end{multline}
\begin{multline} \notag
4\pi \mathcal{A}_{24}  
= 
\DAYP \cdot \left(\frac{(\DO(\DAX) - \IO) }{|\DAX|^2}
-
\frac{(\DO(\DAY) - \IO) }{|\DAY|^2}\right)\DAYM \DO(\DAX)
\\ 
+ \DAYP \cdot 
\frac{(\DO(\DAY) - \IO) }{|\DAY|^2}
\DAYM 
\left(\DO(\DAX)-\DO(\DAY)\right).
\end{multline}
Thus by inspection we have the following uniform estimate for $\mathcal{A}_2$ as
\begin{multline} \notag
\left| \mathcal{A}_2 \right|
\lesssim
|\BX,\BY|_*^{-3}\left| \DAL(\bX'-\bY')\right| \left| \DAYM\right|
\left| \DAYP\right|
\\
+ 
|\BX,\BY|_*^{-2}\left| \DAL(\bX'-\bY')\right|\left( \left| \DAYM\right|
+\left| \DAYP\right|
\right).
\end{multline}
This completes the proof.
\end{proof}

We lastly give the proof of Lemma \ref{lem:Abeta.upper}.

\begin{proof}[Proof of Lemma \ref{lem:Abeta.upper}]
Then as in \eqref{A1X.def} and \eqref{A2X.def} we decompose the terms $\mathcal{A}_{j\beta i}[\bX] - \mathcal{A}_{j\beta i}[\bY]$ for $j\in\{1,2\}$ and for $i\in\{1,2,3,4\}$ individually as 
\begin{equation}\notag
\mathcal{A}_{j\beta i}[\bX] - \mathcal{A}_{j\beta i}[\bY]
\eqdef
\mathcal{A}_{j\beta i1}[\bX,\bY]+\mathcal{A}_{j\beta i2}[\bX,\bY],
\end{equation}
where
\begin{multline}\notag
4\pi\mathcal{A}_{1\beta 11}[\bX,\bY]
=
\delta_\beta \DAPXY\cdot \tau_\beta\frac{ \DO(\DAX) }{|\DAX|^2}\tau_\beta \DAM\IO 
\\
+
\delta_\beta \DAYP\cdot \tau_\beta\frac{ \DO(\DAX) }{|\DAX|^2}\tau_\beta \DAMXY\IO
\\
+
\DAPXY\cdot \frac{ \DO(\DAX) }{|\DAX|^2}\delta_\beta\DAM \IO
\\
+
\DAYP\cdot \frac{ \DO(\DAX) }{|\DAX|^2}\delta_\beta\DAMXY \IO, 
\end{multline}
\begin{multline}\notag
4\pi\mathcal{A}_{1\beta 12}[\bX,\bY]
=
\DAYP\cdot \left(\frac{ \DO(\DAX) }{|\DAX|^2}-\frac{ \DO(\DAY) }{|\DAY|^2}\right)\delta_\beta\DAYM \IO,
\\
+
\delta_\beta \DAYP\cdot \tau_\beta
\left(\frac{ \DO(\DAX) }{|\DAX|^2}-\frac{ \DO(\DAY) }{|\DAY|^2}\right)\tau_\beta \DAYM\IO,
\end{multline}
\begin{multline}\notag
4\pi\mathcal{A}_{1\beta 21}[\bX,\bY]
=
(\delta_\beta\DAPXY + \delta_\beta \DAMXY)\cdot \tau_\beta \frac{\DO(\DAX) \DAX}{|\DAX|^2}\IO 
\\ 
- (\delta_\beta\DAPXY + \delta_\beta \DAMXY)\cdot  \tau_\beta \frac{ \RO(\DAX) \DAX}{|\DAX|^2}
\tau_\beta \RO(\DAX),
\end{multline}
\begin{multline}\notag
4\pi\mathcal{A}_{1\beta 22}[\bX,\bY]
\\
=
(\delta_\beta\DAYP + \delta_\beta\DAYM)\cdot \tau_\beta 
\left(\frac{\DO(\DAX) \DAX}{|\DAX|^2} - \frac{\DO(\DAY) \DAY}{|\DAY|^2}\right)\IO 
\\ 
- (\delta_\beta\DAYP + \delta_\beta\DAYM)\cdot  \tau_\beta \frac{ \RO(\DAX) \DAX}{|\DAX|^2}
\tau_\beta \RO(\DAX)
\\ 
+ (\delta_\beta\DAYP + \delta_\beta\DAYM)\cdot  \tau_\beta 
\frac{ \RO(\DAY) \DAY}{|\DAY|^2}
\tau_\beta \RO(\DAY),
\end{multline}
\begin{multline}\notag
4\pi\mathcal{A}_{1\beta 31}[\bX,\bY]
=
-\delta_\beta\DAPXY \cdot \tau_\beta \frac{\RO(\DAX) }{|\DAX|^2} \tau_\beta\DAM
\tau_\beta \RO(\DAX)
\\
-\delta_\beta\DAYP \cdot \tau_\beta \frac{\RO(\DAX) }{|\DAX|^2} \tau_\beta\DAMXY
\tau_\beta \RO(\DAX)
\\
-\DAPXY \cdot \frac{\RO(\DAX) }{|\DAX|^2} \delta_\beta\DAM
\tau_\beta \RO(\DAX)
\\
-\DAYP \cdot \frac{\RO(\DAX) }{|\DAX|^2} \delta_\beta\DAMXY
\tau_\beta \RO(\DAX),
\end{multline}
\begin{multline}\notag
4\pi\mathcal{A}_{1\beta 32}[\bX,\bY]
=
-\delta_\beta\DAYP \cdot \tau_\beta \left(\frac{\RO(\DAX) }{|\DAX|^2} -\frac{\RO(\DAY) }{|\DAY|^2}\right) \tau_\beta\DAYM
\tau_\beta \RO(\DAX)
\\
-\delta_\beta\DAYP \cdot \tau_\beta \frac{\RO(\DAY) }{|\DAY|^2} \tau_\beta\DAYM
\tau_\beta  \left(\RO(\DAX)-\RO(\DAY)\right)
\\
-\DAYP \cdot \left(\frac{\RO(\DAX) }{|\DAX|^2} -\frac{\RO(\DAY) }{|\DAY|^2}\right)  \delta_\beta\DAYM
\tau_\beta \RO(\DAX)
\\
-\DAYP \cdot \frac{\RO(\DAX) }{|\DAX|^2} \delta_\beta\DAYM
\tau_\beta  \left(\RO(\DAX)-\RO(\DAY)\right),
\end{multline}
\begin{multline}\notag
4\pi\mathcal{A}_{1\beta 41}[\bX,\bY]
=
 \delta_\beta\DAPXY \cdot \tau_\beta\frac{(\DO(\DAX) - \IO) }{|\DAX|^2}\tau_\beta\DAM \tau_\beta \DO(\DAX)
\\
+
 \delta_\beta\DAYP \cdot \tau_\beta\frac{(\DO(\DAX) - \IO) }{|\DAX|^2}\tau_\beta\DAMXY \tau_\beta \DO(\DAX)
\\
+ \DAPXY\cdot \frac{(\DO(\DAX) - \IO) }{|\DAX|^2}\delta_\beta\DAM \tau_\beta \DO(\DAX)
\\
+ \DAYP\cdot \frac{(\DO(\DAX) - \IO) }{|\DAX|^2}\delta_\beta\DAMXY \tau_\beta \DO(\DAX),
\end{multline}
\begin{multline}\notag
4\pi\mathcal{A}_{1\beta 42}[\bX,\bY]
\\
=
 \delta_\beta\DAYP \cdot \tau_\beta\left(\frac{(\DO(\DAX) - \IO) }{|\DAX|^2}- \frac{(\DO(\DAY) - \IO) }{|\DAY|^2}\right)\tau_\beta\DAYM \tau_\beta \DO(\DAX)
 \\
 +
  \delta_\beta\DAYP \cdot \tau_\beta\frac{(\DO(\DAY) - \IO) }{|\DAY|^2}\tau_\beta\DAYM \tau_\beta \left(\DO(\DAX) - \DO(\DAY)\right)
\\
+ \DAYP\cdot \left(\frac{(\DO(\DAX) - \IO) }{|\DAX|^2}- \frac{(\DO(\DAY) - \IO) }{|\DAY|^2}\right)\delta_\beta\DAYM \tau_\beta \DO(\DAX)
\\
+ \DAYP\cdot \frac{(\DO(\DAX) - \IO) }{|\DAX|^2}\delta_\beta\DAYM \tau_\beta \left(\DO(\DAX) - \DO(\DAY)\right).
\end{multline}
Therefore by inspection of each term we have 
\begin{multline}\label{A1beta11.upper}
\left| \mathcal{A}_{1\beta11}[\bX,\bY] \right|
+
\left| \mathcal{A}_{1\beta31}[\bX,\bY] \right|
+
\left| \mathcal{A}_{1\beta41}[\bX,\bY] \right|
\\
\lesssim
|\delta_\beta \DAPXY||\tau_\beta \DAM| |\BX|_{*}^{-2}
\\
+
|\delta_\beta \DAYP| 
|\tau_\beta \DAMXY|
|\BX|_{*}^{-2}
+
|\DAPXY||\delta_\beta \DAM| |\BX|_{*}^{-2}
\\
+
|\DAYP||\delta_\beta \delta_\alpha^- (\BX'-\BY')| |\BX|_{*}^{-2},
\end{multline}
\begin{multline}\label{A1beta12.upper}
\left| \mathcal{A}_{1\beta12}[\bX,\bY] \right|
+
\left| \mathcal{A}_{1\beta32}[\bX,\bY] \right|
+
\left| \mathcal{A}_{1\beta42}[\bX,\bY] \right|
\\
\lesssim
|\delta_\beta \DAYP||\tau_\beta \DAYM|
|\tau_\beta \DAL(\bX'-\bY')||\BX,\BY|_{*}^{-3}
\\
+
|\DAYP||\delta_\beta \DAYM| 
|\DAL(\bX'-\bY')| |\BX,\BY|_{*}^{-3},
\end{multline}
\begin{equation}\label{A1beta21.upper}
\left| \mathcal{A}_{1\beta21}[\bX,\bY] \right|
\lesssim
\left( |\delta_\beta \DAPXY|
+
|\delta_\beta \DAMXY|
\right)|\BX|_{*}^{-1},
\end{equation}
\begin{equation}\label{A1beta22.upper}
\left| \mathcal{A}_{1\beta22}[\bX,\bY] \right|
\lesssim
\left( |\delta_\beta \DAYP|
+
|\delta_\beta \DAYM|
\right)
|\tau_\beta \DAL(\bX'-\bY')||\BX,\BY|_{*}^{-2}.
\end{equation}
Then \eqref{A1beta.upper} follows from collecting the estimates in \eqref{A1beta11.upper}, \eqref{A1beta12.upper}, \eqref{A1beta21.upper} and \eqref{A1beta22.upper}.

Next we consider the differences of the form $\mathcal{A}_{2\beta i}[\bX] - \mathcal{A}_{2\beta i}[\bY]$.  We obtain
\begin{multline}\notag
4\pi\mathcal{A}_{2\beta 11}[\bX,\bY]
=
 \DAPXY\cdot \delta_\beta\frac{ \DO(\DAX) }{|\DAX|^2}\tau_\beta \DAM\IO 
\\
+
 \DAYP\cdot \delta_\beta\frac{ \DO(\DAX) }{|\DAX|^2}\tau_\beta \DAMXY\IO,
\end{multline}
\begin{equation}\notag
4\pi\mathcal{A}_{2\beta 12}[\bX,\bY]
=
 \DAYP\cdot \delta_\beta
\left(\frac{ \DO(\DAX) }{|\DAX|^2}-\frac{ \DO(\DAY) }{|\DAY|^2}\right)\tau_\beta \DAYM\IO,
\end{equation}
\begin{multline}\notag
4\pi\mathcal{A}_{2\beta 21}[\bX,\bY]
=
-\DAPXY \cdot \delta_\beta \frac{\RO(\DAX) }{|\DAX|^2} \tau_\beta\DAM
\tau_\beta \RO(\DAX)
\\
-\DAYP \cdot \delta_\beta \frac{\RO(\DAX) }{|\DAX|^2} \tau_\beta\DAMXY
\tau_\beta \RO(\DAX)
\\
-\DAPXY \cdot \frac{\RO(\DAX) }{|\DAX|^2} \DAM
\delta_\beta \RO(\DAX)
\\
-\DAYP \cdot \frac{\RO(\DAX) }{|\DAX|^2} \DAMXY
\delta_\beta \RO(\DAX),
\end{multline}
\begin{multline}\notag
4\pi\mathcal{A}_{2\beta 22}[\bX,\bY]
=
-\DAYP \cdot \delta_\beta \left(\frac{\RO(\DAX) }{|\DAX|^2} -\frac{\RO(\DAY) }{|\DAY|^2}\right) \tau_\beta\DAYM
\tau_\beta \RO(\DAX)
\\
-\DAYP \cdot \delta_\beta \frac{\RO(\DAY) }{|\DAY|^2} \tau_\beta\DAYM
\tau_\beta  \left(\RO(\DAX)-\RO(\DAY)\right)
\\
-\DAYP \cdot \left(\frac{\RO(\DAX) }{|\DAX|^2} -\frac{\RO(\DAY) }{|\DAY|^2}\right)  \DAYM
\delta_\beta \RO(\DAX)
\\
-\DAYP \cdot \frac{\RO(\DAX) }{|\DAX|^2} \DAYM
\delta_\beta  \left(\RO(\DAX)-\RO(\DAY)\right),
\end{multline}
\begin{multline}\notag
4\pi\mathcal{A}_{2\beta 31}[\bX,\bY]
=
(\DAPXY + \DAMXY)\cdot \delta_\beta \frac{\DO(\DAX) \DAX}{|\DAX|^2}\IO
\\ 
- (\DAPXY + \DAMXY)\cdot  \delta_\beta\left( \frac{ \RO(\DAX) \DAX}{|\DAX|^2}
 \RO(\DAX) \right),
\end{multline}
\begin{multline}\notag
4\pi\mathcal{A}_{2\beta 32}[\bX,\bY]
=
(\DAYP + \DAYM)\cdot 
\delta_\beta
\left(\frac{\DO(\DAX) \DAX}{|\DAX|^2} - \frac{\DO(\DAY) \DAY}{|\DAY|^2}\right)\IO
\\ 
- (\DAYP + \DAYM)\cdot  \delta_\beta
\left(\frac{ \RO(\DAX) \DAX}{|\DAX|^2}
 \RO(\DAX) \right)
\\ 
+ (\DAYP + \DAYM)\cdot  \delta_\beta 
\left(\frac{ \RO(\DAY) \DAY}{|\DAY|^2}
 \RO(\DAY) \right),
\end{multline}
\begin{multline}\notag
4\pi\mathcal{A}_{2\beta 41}[\bX,\bY]
=
\DAPXY \cdot  \delta_\beta\frac{(\DO(\DAX) - \IO) }{|\DAX|^2}\tau_\beta\DAM \tau_\beta \DO(\DAX)
\\
+
\DAYP \cdot  \delta_\beta\frac{(\DO(\DAX) - \IO) }{|\DAX|^2}\tau_\beta\DAMXY\tau_\beta \DO(\DAX)
\\
+ \DAPXY \cdot \frac{(\DO(\DAX) - \IO) }{|\DAX|^2}\DAM \delta_\beta \DO(\DAX)
\\
+ \DAYP\cdot \frac{(\DO(\DAX) - \IO) }{|\DAX|^2}\DAMXY\delta_\beta \DO(\DAX),
\end{multline}
\begin{multline}\notag
4\pi\mathcal{A}_{2\beta 42}[\bX,\bY]
\\
=
\DAYP \cdot  \delta_\beta\left(\frac{(\DO(\DAX) - \IO) }{|\DAX|^2}- \frac{(\DO(\DAY) - \IO) }{|\DAY|^2}\right)\tau_\beta\DAYM \tau_\beta \DO(\DAX)
 \\
 +
\DAYP \cdot  \delta_\beta\frac{(\DO(\DAY) - \IO) }{|\DAY|^2}\tau_\beta\DAYM \tau_\beta \left(\DO(\DAX) - \DO(\DAY)\right)
\\
+ \DAYP\cdot \left(\frac{(\DO(\DAX) - \IO) }{|\DAX|^2}- \frac{(\DO(\DAY) - \IO) }{|\DAY|^2}\right)\DAYM \delta_\beta \DO(\DAX)
\\
+ \DAYP\cdot \frac{(\DO(\DAX) - \IO) }{|\DAX|^2}\DAYM \delta_\beta \left(\DO(\DAX) - \DO(\DAY)\right).
\end{multline}
Therefore by inspection of each term we have 
\begin{multline}\label{A2beta11.upper}
\left| \mathcal{A}_{2\beta11}[\bX,\bY] \right|
+
\left| \mathcal{A}_{2\beta21}[\bX,\bY] \right|
+
\left| \mathcal{A}_{2\beta41}[\bX,\bY] \right|
\\
\lesssim
|\DAPXY||\tau_\beta \DAM| \frac{|\delta_\beta \DAX|}{|\BX|_{*}^{3}}
\\
+
\left(
| \DAYP| 
|\tau_\beta \DAMXY|
+
|\DAPXY|| \DAM| 
\right)
\frac{|\delta_\beta \DAX|}{|\BX|_{*}^{3}}
\\
+
|\DAYP|| \delta_\alpha^- (\BX'-\BY')| 
\frac{|\delta_\beta \DAX|}{|\BX|_{*}^{3}},
\end{multline}
\begin{multline}\label{A2beta12.upper}
\left| \mathcal{A}_{2\beta12}[\bX,\bY] \right|
+
\left| \mathcal{A}_{2\beta22}[\bX,\bY] \right|
+
\left| \mathcal{A}_{2\beta42}[\bX,\bY] \right|
\\
\lesssim
| \DAYP||\tau_\beta \DAYM|
\frac{|\delta_\beta \DAL(\bX'-\bY')|}{|\BX,\BY|_{*}^{3}}
\\
+
| \DAYP||\tau_\beta \DAYM|
| \DAL(\bX'-\bY')|
\frac{|\delta_\beta\DAX, \delta_\beta\DAY|}{|\BX,\BY|_{*}^{4}}
\\
+
| \DAYP||\tau_\beta \DAYM|
|\tau_\beta\DAL(\bX'-\bY')|
\frac{|\delta_\beta\DAX, \delta_\beta\DAY|}{|\BX,\BY|_{*}^{4}}
\\
+
|\DAYP|| \DAYM| 
\frac{|\delta_\beta \DAL(\bX'-\bY')|}{|\BX,\BY|_{*}^{3}}
\\
+
|\DAYP|| \DAYM| 
|\DAL(\bX'-\bY')| \frac{|\delta_\beta\DAX, \delta_\beta\DAY|}{|\BX,\BY|_{*}^{4}},
\end{multline}
\begin{equation}\label{A2beta21.upper}
\left| \mathcal{A}_{2\beta31}[\bX,\bY] \right|
\lesssim
\left( | \DAPXY|
+
| \DAMXY|
\right)\frac{|\delta_\beta \DAX|}{|\BX|_{*}^{3}},
\end{equation}
\begin{multline}\label{A2beta22.upper}
\left| \mathcal{A}_{2\beta32}[\bX,\bY] \right|
\lesssim
\left( | \DAYP|
+
| \DAYM|
\right)
\frac{|\delta_\beta \DAL(\bX'-\bY')|}{|\BX,\BY|_{*}^{3}}
\\
+
\left( | \DAYP|
+
| \DAYM|
\right)
|\DAL(\bX'-\bY')| \frac{|\delta_\beta\DAX, \delta_\beta\DAY|}{|\BX,\BY|_{*}^{4}}.
\end{multline}
Then \eqref{A2beta.upper} again follows from collecting the estimates in \eqref{A2beta11.upper}, \eqref{A2beta12.upper}, \eqref{A2beta21.upper} and \eqref{A2beta22.upper}.
\end{proof}

\providecommand{\bysame}{\leavevmode\hbox to3em{\hrulefill}\thinspace}
\providecommand{\href}[2]{#2}


\begin{thebibliography}{10}
\expandafter\ifx\csname urlpdf\endcsname\relax
  \def\urlpdf#1{\url{#1}}\fi
\expandafter\ifx\csname arxiv\endcsname\relax
  \def\arxiv#1{\burlalt{arXiv:#1}{http://arxiv.org/abs/#1}}\fi
\expandafter\ifx\csname doi\endcsname\relax
  \def\doi#1{\burlalt{doi:#1}{http://dx.doi.org/#1}}\fi
\expandafter\ifx\csname href\endcsname\relax
  \def\href#1#2{#2}\fi
\expandafter\ifx\csname burlalt\endcsname\relax
  \def\burlalt#1#2{\href{#2}{#1}}\fi

\bibitem{2010.06915}
Thomas Alazard and Quoc-Hung Nguyen, \emph{{Endpoint Sobolev theory for the
  Muskat equation}},  (2020), \arxiv{2010.06915}.

\bibitem{2009.08442}
\bysame, \emph{{On the Cauchy problem for the Muskat equation. II: Critical
  initial data}}, Ann. {PDE} (2021), \arxiv{2009.08442},
  \doi{10.1007/s40818-021-00099-x}.

\bibitem{AlaNgu2021lipsh}
\bysame, \emph{{On the Cauchy problem for the Muskat equation with
  non-Lipschitz initial data}}, Communications in Partial Differential
  Equations (2021), 1--42, \arxiv{2009.04343},
  \doi{10.1080/03605302.2021.1928700}.

\bibitem{BCD}
Hajer Bahouri, Jean-Yves Chemin, and Rapha\"{e}l Danchin, \emph{Fourier
  analysis and nonlinear partial differential equations}, Grundlehren der
  Mathematischen Wissenschaften [Fundamental Principles of Mathematical
  Sciences], vol. 343, Springer, Heidelberg, 2011,
  \doi{10.1007/978-3-642-16830-7}.

\bibitem{Bazilevs2013}
Yuri Bazilevs, Kenji Takizawa, and Tayfun~E. Tezduyar, \emph{Computational
  fluid-structure interaction: Methods and applications}, John Wiley {\&} Sons,
  Ltd, January 2013, \doi{10.1002/9781118483565}.

\bibitem{FractionalPorousMedium1}
Luis Caffarelli, Fernando Soria, and Juan~Luis V\'{a}zquez, \emph{Regularity of
  solutions of the fractional porous medium flow}, J. Eur. Math. Soc. (JEMS)
  \textbf{15} (2013), no.~5, 1701--1746, \doi{10.4171/JEMS/401}.

\bibitem{Cam17}
Stephen Cameron, \emph{{G}lobal well-posedness for the 2{D} {M}uskat problem
  with slope less than 1}, Anal. PDE \textbf{12} (2019), no.~4, 997--1022,
  \doi{10.2140/apde.2019.12.997}.

\bibitem{Cam20}
Stephen Cameron, \emph{{G}obal wellposedness for the 3{D} {M}uskat problem with
  medium size slope},  (2020), \arxiv{2002.00508}.

\bibitem{CCFGL12}
\'{A}ngel Castro, Diego C\'{o}rdoba, Charles Fefferman, Francisco Gancedo, and
  Mar\'{\i}a L\'{o}pez-Fern\'{a}ndez, \emph{Rayleigh-{T}aylor breakdown for the
  {M}uskat problem with applications to water waves}, Ann. of Math. (2)
  \textbf{175} (2012), no.~2, 909--948, \doi{10.4007/annals.2012.175.2.9}.

\bibitem{CL1995}
J.-Y. Chemin and N.~Lerner, \emph{Flot de champs de vecteurs non lipschitziens
  et \'{e}quations de {N}avier-{S}tokes}, J. Differential Equations
  \textbf{121} (1995), no.~2, 314--328, \doi{10.1006/jdeq.1995.1131}.

\bibitem{2107.13854}
Ke~Chen and Quoc-Hung Nguyen, \emph{{The Peskin problem with $BMO^1$ initial
  data}},  (2021), \arxiv{2107.13854}.

\bibitem{CGS16}
C.~H.~Arthur Cheng, Rafael Granero-Belinch\'{o}n, and Steve Shkoller,
  \emph{Well-posedness of the {M}uskat problem with {$H^2$} initial data}, Adv.
  Math. \textbf{286} (2016), 32--104, \doi{10.1016/j.aim.2015.08.026}.

\bibitem{CCGRS16}
Peter Constantin, Diego C\'{o}rdoba, Francisco Gancedo, Luis
  Rodr\'{i}guez-Piazza, and Robert~M. Strain, \emph{On the {M}uskat problem:
  global in time results in 2{D} and 3{D}}, Amer. J. Math. \textbf{138} (2016),
  no.~6, 1455--1494, \arxiv{1310.0953}, \doi{10.1353/ajm.2016.0044}.

\bibitem{CCGS13}
Peter Constantin, Diego C\'{o}rdoba, Francisco Gancedo, and Robert~M. Strain,
  \emph{On the global existence for the {M}uskat problem}, J. Eur. Math. Soc.
  (JEMS) \textbf{15} (2013), no.~1, 201--227, \arxiv{1007.3744},
  \doi{10.4171/JEMS/360}.

\bibitem{CCG11}
Antonio C\'{o}rdoba, Diego C\'{o}rdoba, and Francisco Gancedo, \emph{Interface
  evolution: the {H}ele-{S}haw and {M}uskat problems}, Ann. of Math. (2)
  \textbf{173} (2011), no.~1, 477--542, \doi{10.4007/annals.2011.173.1.10}.

\bibitem{MR2472040}
Diego C\'{o}rdoba and Francisco Gancedo, \emph{A maximum principle for the
  {M}uskat problem for fluids with different densities}, Comm. Math. Phys.
  \textbf{286} (2009), no.~2, 681--696, \doi{10.1007/s00220-008-0587-1}.

\bibitem{FractionalPorousMedium2}
Arturo de~Pablo, Fernando Quir\'{o}s, Ana Rodr\'{\i}guez, and Juan~Luis
  V\'{a}zquez, \emph{A fractional porous medium equation}, Adv. Math.
  \textbf{226} (2011), no.~2, 1378--1409, \doi{10.1016/j.aim.2010.07.017}.

\bibitem{MR3639321}
Fan Deng, Zhen Lei, and Fanghua Lin, \emph{On the two-dimensional {M}uskat
  problem with monotone large initial data}, Comm. Pure Appl. Math. \textbf{70}
  (2017), no.~6, 1115--1145, \doi{10.1002/cpa.21669}.

\bibitem{FlynnNguyen2020}
Patrick~T. Flynn and Huy~Q. Nguyen, \emph{The vanishing surface tension limit
  of the {M}uskat problem}, Comm. Math. Phys. \textbf{382} (2021), no.~2,
  1205--1241, \arxiv{2001.10473}, \doi{10.1007/s00220-021-03980-9}.

\bibitem{2011.02294}
Francisco Gancedo, Rafael {Granero-Belinch\'on}, and Stefano Scrobogna,
  \emph{{Global existence in the Lipschitz class for the N-Peskin problem}},
  (2020), \arxiv{2011.02294}.

\bibitem{2009.03360}
Eduardo Garc{\'{i}}a-Ju{\'{a}}rez, Yoichiro Mori, and Robert~M. Strain,
  \emph{{The Peskin Problem with Viscosity Contrast}}, Anal. PDE \textbf{in
  press} (2020), 54 pages, \arxiv{2009.03360}.

\bibitem{MR1808257}
Ming-Chih Lai and Zhilin Li, \emph{A remark on jump conditions for the
  three-dimensional {N}avier-{S}tokes equations involving an immersed moving
  membrane}, Appl. Math. Lett. \textbf{14} (2001), no.~2, 149--154,
  \doi{10.1016/S0893-9659(00)00127-0}.

\bibitem{MR2242805}
Zhilin Li and Kazufumi Ito, \emph{The immersed interface method}, Frontiers in
  Applied Mathematics, vol.~33, Society for Industrial and Applied Mathematics
  (SIAM), Philadelphia, PA, 2006, \doi{10.1137/1.9780898717464}.

\bibitem{MR3882225}
Fang-Hua Lin and Jiajun Tong, \emph{Solvability of the {S}tokes immersed
  boundary problem in two dimensions}, Comm. Pure Appl. Math. \textbf{72}
  (2019), no.~1, 159--226, \doi{10.1002/cpa.21764}.

\bibitem{MR1867882}
Andrew~J. Majda and Andrea~L. Bertozzi, \emph{Vorticity and incompressible
  flow}, Cambridge Texts in Applied Mathematics, vol.~27, Cambridge University
  Press, Cambridge, 2002.

\bibitem{MR2115343}
Rajat Mittal and Gianluca Iaccarino, \emph{Immersed boundary methods}, Annual
  review of fluid mechanics. {V}ol. 37, Annu. Rev. Fluid Mech., vol.~37, Annual
  Reviews, Palo Alto, CA, 2005, pp.~239--261,
  \doi{10.1146/annurev.fluid.37.061903.175743}.

\bibitem{MR3935476}
Yoichiro Mori, Analise Rodenberg, and Daniel Spirn, \emph{Well-posedness and
  global behavior of the {P}eskin problem of an immersed elastic filament in
  {S}tokes flow}, Comm. Pure Appl. Math. \textbf{72} (2019), no.~5, 887--980,
  \doi{10.1002/cpa.21802}.

\bibitem{NguyenST2019}
Huy~Q. Nguyen, \emph{On well-posedness of the {M}uskat problem with surface
  tension}, Adv. Math. \textbf{374} (2020), 107344, 35, \arxiv{1907.11552},
  \doi{10.1016/j.aim.2020.107344}.

\bibitem{2103.14535}
\bysame, \emph{{Global solutions for the Muskat problem in the scaling
  invariant Besov space $\dot{B}^1_{\infty, 1}$}},  (2021), \arxiv{2103.14535}.

\bibitem{NguyenPausader2019}
Huy~Q. Nguyen and Beno\^{\i}t Pausader, \emph{A paradifferential approach for
  well-posedness of the {M}uskat problem}, Arch. Ration. Mech. Anal.
  \textbf{237} (2020), no.~1, 35--100, \arxiv{1907.03304},
  \doi{10.1007/s00205-020-01494-7}.

\bibitem{PeskinThesis1972}
Charles~S. Peskin, \emph{Flow patterns around heart valves: a digital computer
  method for solving the equations of motion}, Ph.D. thesis, Yeshiva
  University, 1972.

\bibitem{PESKIN1972252}
\bysame, \emph{Flow patterns around heart valves: A numerical method}, Journal
  of Computational Physics \textbf{10} (1972), no.~2, 252 -- 271,
  \doi{https://doi.org/10.1016/0021-9991(72)90065-4}.

\bibitem{MR2009378}
\bysame, \emph{The immersed boundary method}, Acta Numer. \textbf{11} (2002),
  479--517, \doi{10.1017/S0962492902000077}.

\bibitem{MR1156495}
Constantine Pozrikidis, \emph{Boundary integral and singularity methods for
  linearized viscous flow}, Cambridge Texts in Applied Mathematics, Cambridge
  University Press, Cambridge, 1992, \doi{10.1017/CBO9780511624124}.

\bibitem{richter2017fluid}
Thomas Richter, \emph{Fluid-structure interactions: models, analysis and finite
  elements}, vol. 118, Springer, 2017.

\bibitem{rodenberg_thesis}
Analise Rodenberg, \emph{2d peskin problems of an immersed elastic filament in
  stokes flow}, Ph.D. thesis, University of Minnesota, 2018.

\bibitem{VectorPorous3}
Tariel~A. Sanikidze and Anatoli~F. Tedeev, \emph{On the temporal decay
  estimates for the degenerate parabolic system}, Commun. Pure Appl. Anal.
  \textbf{12} (2013), no.~4, 1755--1768, \doi{10.3934/cpaa.2013.12.1755}.

\bibitem{SchTrie1987}
Hans-J\"{u}rgen Schmeisser and Hans Triebel, \emph{Topics in {F}ourier analysis
  and function spaces}, A Wiley-Interscience Publication, John Wiley \& Sons,
  Ltd., Chichester, 1987.

\bibitem{1904.09528}
Jiajun Tong, \emph{Regularized {S}tokes immersed boundary problems in two
  dimensions: well-posedness, singular limit, and error estimates}, Comm. Pure
  Appl. Math. \textbf{74} (2021), no.~2, 366--449, \arxiv{1904.09528},
  \doi{10.1002/cpa.21968}.

\bibitem{TRYGGVASON2001708}
G.~Tryggvason, B.~Bunner, A.~Esmaeeli, D.~Juric, N.~Al-Rawahi, W.~Tauber,
  J.~Han, S.~Nas, and Y.-J. Jan, \emph{A front-tracking method for the
  computations of multiphase flow}, Journal of Computational Physics
  \textbf{169} (2001), no.~2, 708 -- 759, \doi{10.1006/jcph.2001.6726}.

\bibitem{MR3656476}
Juan~Luis V\'{a}zquez, Arturo de~Pablo, Fernando Quir\'{o}s, and Ana
  Rodr\'{\i}guez, \emph{Classical solutions and higher regularity for nonlinear
  fractional diffusion equations}, J. Eur. Math. Soc. (JEMS) \textbf{19}
  (2017), no.~7, 1949--1975, \doi{10.4171/JEMS/710}.

\bibitem{VectorPorous2}
Hong-Ming Yin, \emph{On a degenerate parabolic system}, J. Differential
  Equations \textbf{245} (2008), no.~3, 722--736,
  \doi{10.1016/j.jde.2008.03.017}.

\bibitem{VectorPorous1}
Hong~Jun Yuan, \emph{The {C}auchy problem for a quasilinear degenerate
  parabolic system}, Nonlinear Anal. \textbf{23} (1994), no.~2, 155--164,
  \doi{10.1016/0362-546X(94)90039-6}.

\end{thebibliography}
\end{document}